\documentclass[11pt,a4paper,reqno]{amsart}
\usepackage[utf8]{inputenc}
\usepackage{amsmath}
\usepackage{amsthm}
\usepackage{import}
\usepackage{geometry}
\usepackage{mathrsfs}
\usepackage{bbm}
\usepackage{libertine}
\usepackage[cal=euler, scr=boondoxo]{mathalfa}
\usepackage{tkz-graph}
\usepackage{subfig}
\usepackage[clock]{ifsym}
\usepackage{amssymb}
\usepackage{enumerate, enumitem} 

\usepackage[hyperindex,breaklinks]{hyperref}
\hypersetup{ colorlinks=true, linkcolor=blue, citecolor=blue,
  filecolor=blue, urlcolor=blue}
\usepackage[comma, sort&compress]{natbib}
\usepackage{float}
\usetikzlibrary{matrix}
\usepackage{comment,accents}
\usepackage{graphicx}
\usepackage{tikz}
\usetikzlibrary{arrows,shapes,positioning}
\usetikzlibrary{decorations.markings}
\usetikzlibrary{calc}

\newtheorem{theorem}{Theorem}[section]

\newtheorem*{lemma*}{Lemma}
\newtheorem{prop}[theorem]{Proposition}

\newtheorem{lemma}[theorem]{Lemma}

\newtheorem{remark}[theorem]{Remark}

\newtheorem{cor}[theorem]{Corollary}
\newtheorem{claim}[theorem]{Claim}

\newcommand{\IND}{\mathbbm{1}}

\newcommand{\ldef}{\ensuremath{\mathrel{\mathop:}=}}
\newcommand{\rdef}{\ensuremath{=\mathrel{\mathop:}}}

\newcommand{\R}{\mathbb{R}}

\newcommand{\N}{\mathbb{N}}
\newcommand{\Z}{\mathbb{Z}}

\newcommand{\cA}{\ensuremath{\mathcal A}}
\newcommand{\cB}{\ensuremath{\mathcal B}}

\newcommand{\cD}{\ensuremath{\mathcal D}}

\newcommand{\cF}{\ensuremath{\mathcal F}}

\newcommand{\cH}{\ensuremath{\mathcal H}}

\newcommand{\cO}{\ensuremath{\mathcal O}}
\newcommand{\cP}{\ensuremath{\mathcal P}}

\newcommand{\cS}{\ensuremath{\mathcal S}}

\newcommand{\cU}{\ensuremath{\mathcal U}}
\newcommand{\cV}{\ensuremath{\mathcal V}}
\newcommand{\cW}{\ensuremath{\mathcal W}}
\newcommand{\cX}{\ensuremath{\mathcal X}}
\newcommand{\cY}{\ensuremath{\mathcal Y}}
\newcommand{\cZ}{\ensuremath{\mathcal Z}}
\newcommand{\bbD}{\ensuremath{\mathbb D}}
\newcommand{\bbE}{\ensuremath{\mathbb E}}

\newcommand{\bbN}{\ensuremath{\mathbb N}}

\newcommand{\bbP}{\ensuremath{\mathbb P}}

\newcommand{\bbR}{\ensuremath{\mathbb R}}

\newcommand{\bbU}{\ensuremath{\mathbb U}}

\newcommand{\bbX}{\ensuremath{\mathbb X}}

\newcommand{\bbZ}{\ensuremath{\mathbb Z}}
\newcommand{\bfP}{\ensuremath{\mathbf P}}

\newcommand{\bsA}{\ensuremath{\boldsymbol A}}

\newcommand{\scrA}{\mathscr{A}}
\newcommand{\scrU}{\mathscr{U}}

\newcommand{\e}{\mathrm{e}}

\newcommand{\De}{\mathrm{d}}

\newcommand{\bes}{\begin{equation*}}
\newcommand{\ees}{\end{equation*}}
\newcommand{\beas}{\begin{eqnarray*}}
\newcommand{\eeas}{\end{eqnarray*}}
\newcommand{\bea}{\begin{eqnarray}}
\newcommand{\eea}{\end{eqnarray}}
\newcommand{\be}{\begin{equation}}
\newcommand{\ee}{\end{equation}}

\newcommand{\bP}{\mathbf{P}}

\newcommand{\bbl}{\begin{block}}
\newcommand{\ebl}{\end{block}}

\newcommand{\teX}{\widetilde{X}}

 \numberwithin{equation}{section}
\begin{document}

\title{Approximating conditional distributions}

\author[A. Chiarini]{Alberto Chiarini}
\address{ETH Z\"urich, Department of Mathematics }
\curraddr{R\"amistrasse 101, 8092, Zurich, Switzerland}
\email{alberto.chiarini@math.ethz.ch}
\thanks{}

\author[A. Cipriani]{Alessandra Cipriani}
\address{Department of Mathematical Sciences, University of Bath}
\curraddr{Claverton Down, Bath, BA2 7AY, United Kingdom}
\email{A.Cipriani@bath.ac.uk}
\thanks{The authors thank their previous academic affiliations (CEMPI Lille, the University of Aix---Marseille, the University of Leipzig, and WIAS Berlin), as well as the Max-Planck-Institut Leipzig, where part of this research was carried out. The second author is supported by the EPSRC grant EP/N004566/1.}

\author[G. Conforti]{Giovanni Conforti}
\address{CMAP, \'Ecole Polytechnique}
\curraddr{Route de Saclay, 91128 Palaiseau Cedex, France.}
\email{giovanni.conforti@polytechnique.edu}
\thanks{}
\makeatletter
\def\paragraph{\@startsection{paragraph}{4}%
  \z@\z@{-\fontdimen2\font}%
  {\normalfont\itshape}}
  \makeatother
\subjclass[2010]{}


\date{\today}

\begin{abstract}
	In this article, we discuss the basic ideas of a general procedure to adapt the Stein--Chen method to bound the distance between conditional distributions.
  From an integration-by-parts formula (IBPF), we derive a Stein operator whose solution can be bounded, for example, via ad hoc couplings.
  This method provides quantitative bounds in several examples: the filtering equation, the distance between bridges of random walks and the distance between bridges and discrete schemes approximating them. Moreover, through the coupling construction for a certain class of random walk bridges we determine samplers, whose convergence to equilibrium is computed explicitly.
\end{abstract}

\maketitle

\tableofcontents
\section{Introduction}

Stein's method is a powerful tool to determine quantitative approximations of random variables in a
wide variety of contexts. It was first introduced by~\cite{Stein:1972} and developed
by~\cite{Chen:1975a,Chen:1975b}, which is why it is often called the Stein--Chen method. Stein originally implemented it for a central limit approximation, but later his idea found a much wider range of applications. In fact, his method has a
big advantage over several other techniques, in that it can be used for approximation in terms of
any distribution on any space, and moreover does not require strong independence assumptions.
Enhancing the method with auxiliary randomization techniques as the method of exchangeable pairs~\citep{Stein:1986} and using the so-called generator
interpretation~\citep{Barbour:1988}, the
Stein-Chen method has had a tremendous impact in the field of probability theory. Its applications
range from Poisson point process approximation~\citep{Barbour/Brown:1992}, to normal approximation~\citep{Bolthausen:1984,Gotze:1991}, eigenfunctions of the Laplacian on a
manifold~\citep{Meckes:2009}, logarithmic combinatorial structures~\citep{Arratia/Barbour/Tavare:2003},
diffusion approximation~\citep{Barbour:1990}, statistical mechanics~\citep{CCHextremal,Eichelsbacher/Reinert:2008}, Wiener chaos and Malliavin calculus~\citep{Nourdin/Peccati:2009,Nourdin/Peccati:2010}.
For a more extensive overview of the method, we refer the reader to~\cite{Barbour/Chen:2005,
	Barbour/Chen:2014}.
	In this article we are interested in comparing conditional distributions. That is, given two laws $P,\,Q$ on the same probability space $\Omega$, an observable $\phi:\Omega \rightarrow E$  and $e \in E$ we aim at bounding the distance $d(P(\cdot | \phi = e),\,Q(\cdot | \phi=e))$.
	This task may be quite demanding, even  when the non-conditional laws are well-understood.
	Here, relying on some simple but quite general observations on conditioning, we propose a way of adapting Stein's method to conditional laws.
	In particular, we obtain a fairly general scheme to construct a characteristic (Stein) operator for $P( \cdot | \phi = e)$ provided that the behavior of $P$ under certain \emph{information preserving} transformations is known. The final estimates, obtained with the classical Stein's method, are quantitative. Thus they are very useful when one wants to implement simulations of stochastic processes with a precise error rate.
	We will see one such example concerning random walk bridges where we characterize the measure of the bridge as the invariant distribution of a stochastic process on path space.
	One can in principle use such dynamics, and the related estimates for convergence to equilibrium, to sample the distribution of the bridge.

	To keep our paper self-contained and better explain this procedure, let us recall some basic notions on the Stein's method.

\subsection{Generalities on Stein's method}\label{sec:generality_stein}
We consider a probability metric of the form
\begin{equation}\label{eq:metric}
	d(P,\,Q) \ldef \sup_{h\in \cH}\left|\int_{\Omega} h \De P - \int_{\Omega} h \De Q\right|,
\end{equation}
where $P$, $Q$ are probability measures on a Polish space $\Omega$ with the Borel $\sigma$-algebra
$\cB(\Omega)$, and $\cH$ is a set of real valued functions large enough so that $d$ is indeed a metric.
Natural choices for $\cH$ are the set of indicator functions of measurable subsets of $\Omega$, which
gives the total variation distance, and the set of $1$-Lipschitz functions, which defines the first
order Wasserstein-Kantorovich distance.
Next, we consider a probability measure $P$ on $\Omega$ which is completely characterized
by a certain operator $\cA$ acting on a class $\cD$ of functions from $\Omega$ to $\bbR$. That is,
\[
	\int_{\Omega} \cA f \,\De Q = 0,\quad \forall \,f\in\cD
\]
if and only $Q = P$. The operator $\cA$ is called characteristic operator, or Stein's operator.

Now suppose that we are able to solve the following equation for any given datum $h\in \cH$:
\begin{equation}\label{eq:Stein}
	\cA f = h - \int_{\Omega} h \,\De P
\end{equation}
and call the solution $f_h\in\cD$. Then, by integrating~\eqref{eq:Stein} with respect to $Q$ and
taking the supremum for $h\in\cH$, we obtain
\begin{equation}\label{eq:metric-stein}
	d(P,Q) =\sup_{h\in \cH}\left|\int_\Omega h\,\De P - \int_\Omega h \,\De Q\right| = \sup_{h\in \cH}
	\left|\int_\Omega \cA f_h \, \De Q\right|.
\end{equation}
A closer look at~\eqref{eq:metric-stein} tells us that we will be able to estimate the distance
between $P$ and $Q$ by a careful analysis of the Stein's operator. Of course, all this discussion
is worth only if the right hand side of~\eqref{eq:metric-stein} is easier to bound than the left
hand side, which turns out to be often the case.

Observe that the mere fact that we ask for existence of solutions to~\eqref{eq:Stein} for $h\in
	\cH$ tells us that the operator $\cA$ is characterizing for $Q$. Indeed,
\[
	E_Q[\cA g] = 0,\, \forall \,g\in \cD \Rightarrow  \int_\Omega h \,\De Q =  \int_\Omega h \,\De P,\,\forall \,
	h\in\cH
\]
which implies $Q=P$, since otherwise $d(P,Q)$ as defined above would not be a metric.
\begin{remark}\label{rem:perturbation_operator}

	The method becomes particularly effective when both measures have a characterizing operator and one
	is a ``perturbation'' of the other. Say that $P$ is characterized by $\cA$ and $Q$ by
	$\tilde{\cA}$. Then using that $\int_\Omega \tilde{\cA} f_h\,\De Q = 0$ we get
	\[
		d(P,Q) = \sup_{h\in \cH} \left|\int_\Omega (\cA-\tilde{\cA}) f_h \,\De Q \right|,
	\]
	which tells us that $P$ and $Q$ are close if their characterizing operators are close.
\end{remark}

\subsection{Outline of the method}\label{sec_outline}

To keep things simple, we assume in this introduction $\Omega$ to be at most countable, that the support of $P$ is $\Omega$ and that $P(\phi = e )>0$.
However, at the price of additional technicalities, the same principles  remain valid in more general setups, such as those considered in this article.
In this section we do not make rigorous proofs but rather give some general ideas, which then have to be implemented ad hoc in the cases of interest.
\begin{enumerate}[label=\Alph*),ref=\Alph*)]
\item\label{A)} If $Q$ is such that $Q(\phi =e)>0$, then $Q(\cdot|\phi=e)$-almost surely we have
\begin{equation}\label{eq:rough_PvZ} \frac{\De Q( \cdot | \phi =e)}{\De P( \cdot | \phi =e)}(\omega) = \frac{1}{Z(e)} \frac{\De Q}{\De P}(\omega), \quad \text{with } \ Z(e) = \frac{P(\phi = e)}{Q(\phi  = e)}. \end{equation}
Thus, if we have the Radon--Nikodym derivative for the unconditional laws, we also have it for the conditional laws upon the computation of a normalization constant. This can be shown rigorously and we refer the reader to \citet[Lemma 1]{Pap/vanZuijlen:1996} for a precise statement of \eqref{eq:rough_PvZ}.
Although the normalisation constant in \eqref{eq:rough_PvZ} may be quite hard to compute, such computation is never required for our method to work.
\item\label{B)} If $ \tau: \Omega \rightarrow \Omega $ is an injective transformation which \emph{preserves the information}, i.e.
 \[ \tau( \{ \omega \in \Omega : \phi(\omega) =e \}) \subseteq \{ \omega \in \Omega : \phi(\omega) =e \} \]
then $P(\cdot |\phi=e)$-almost surely we have
\[  \frac{P( \tau(\omega) | \phi=e ) }{P(\omega |\phi=e) } =  \frac{P( \tau(\omega)  ) }{P(\omega ) }. \]
We can rephrase this by saying that if one has a change-of-measure formula for $\tau$ under $P$, i.e.\ for all $F$ bounded
\begin{equation}\label{eq:tauchange}  E_P [F \circ \tau] = E _P[F G_{\tau}], \end{equation}
then the same formula is valid for the the conditional law
\[ E_{P(\cdot|\phi=e)} [F\circ \tau] = E_{P(\cdot|\phi=e)} [F G_{\tau}]. \]
Indeed, as it easy to see, $G_{\tau} = {P(\tau^{-1}(\omega))}/{P(\omega)}$, where $\tau^{-1}$ is a left inverse of $\tau$.
\end{enumerate}

Let us now see how~\ref{A)} and~\ref{B)} are useful for our purposes.
We assume that $\mathcal{T}_0$ is a family of injective transformations of $\Omega$ such that that for all $\tau \in \mathcal{T}_0$ the change of measure formula~\eqref{eq:tauchange} is known explicitly for $P$.
For instance, one might think to the case when $P$ is the Wiener measure and $\mathcal{T}_0$ is the family of translations by Cameron-Martin paths.

Then, by concatenating different formulas, it is possible to deduce~\eqref{eq:tauchange} for $\tau \in \mathcal{T}$, where
\begin{equation}\label{eq:enlargedinjections} \mathcal{T} = \{ \tau_n \circ \cdots \circ \tau_1:\, \tau_1,\ldots,\tau_n \in \mathcal{T}_0 \} .\
\end{equation}
In the example of Brownian motion, obviously $\mathcal{T}_0 =\mathcal{T}$. However, there are situations where $\mathcal{T}_0 \subsetneqq
\mathcal{T}$, and the elements in $\mathcal{T} \setminus \mathcal{T}_0$ are those which we use for the construction of the characteristic operator. A toy example for this is given in Subsection~\ref{subsec:PoiDiag}; more elaborate examples are in Section~\ref{sec:ale}. \\
If $\mathcal{T}_{\phi,e} \subseteq \mathcal{T}$ is the subset of transformations which preserve the observation, then~\ref{B)} tells that for all $F$ bounded and $\tau \in \mathcal{T}_{\phi,e}$
 \begin{equation}\label{eq:condchar}
 E_{P(\cdot|\phi=e)} [ F \circ \tau ]= E_{P(\cdot |\phi=e)} [ F G_{\tau} ].
 \end{equation}

If $\mathcal{T}_{\phi,e}$ is large enough to span the whole space, in the sense that for all $\omega,\,\omega'$ with $\phi(\omega')=\phi(\omega)=e$ there exist $\tau_{1},\,\ldots,\tau_n,\,\tau'_{1},\,\ldots\,,\tau'_{m} \in \mathcal{T}_{\phi,e}$ such that

 \begin{equation}\label{eq:connecting} \tau'_m \circ \cdots \circ \tau'_1 (\omega') = \tau_n \circ \cdots \circ \tau_1 (\omega),
 \end{equation}
 then~\eqref{eq:condchar} together with the obvious requirement that $P( \{ \phi = e \} |\phi=e)=1$ is indeed a \emph{characterization} of $P(\cdot | \phi = e)$.
Clearly, the smaller $\mathcal{T}_{\phi,e}$, the better the characterization.
The construction of a characteristic operator is now straightforward and follows a kind of ``randomization'' procedure.
That is, for $f,\tau$ fixed we consider~\eqref{eq:condchar} with $ F(\omega)= (f(\omega) - f \circ \tau^{-1}(\omega))\mathbbm{1}_{\tau(\Omega)}(\omega)$, and then sum over $\tau \in \mathcal{T}_{\phi,e}$.
We arrive at
\[   E_{P(\cdot|\phi=e)} \left[\sum_{\tau \in \mathcal{T}_{\phi,e}} (f \circ \tau-f)  + G_{\tau}(f \circ \tau^{-1} -f)\mathbbm{1}_{ \tau(\Omega)}\right] = 0
\]
for all functions $f$.
Thus, the characteristic operator is
\begin{equation}\label{eq:charop} \mathcal{A}f := \sum_{\tau \in \mathcal{T}_{\phi,e}} (f \circ \tau-f)  + G_{\tau}(f \circ \tau^{-1} -f)\mathbbm{1}_{\tau(\Omega)},
\end{equation}
which is the generator of a continuous Markov chain whose dynamics is the following:

\begin{itemize}
\item once at state $\omega$, the chain waits for an exponential random time of parameter $|\mathcal{T}_{\phi,e}| + \sum_{\tau \in \mathcal{T}_{\phi,e}} G_{\tau}(\omega)$ and then jumps to a new state.
\item The new state $\omega'$ is chosen according to the following law:
\begin{equation*}
 \begin{cases}
 \text{Prob}(\omega' = \tau(\omega)) =\frac{1}{|\mathcal{T}_{\phi,e}| + \sum_{\tau' \in \mathcal{T}_{\phi,e}} G_{\tau'}\mathbbm{1}_{\tau'(\Omega)}}(\omega) \quad & \text{for $\tau \in \mathcal{T}_{\phi,e}$}, \\
 \text{Prob}(\omega' =  \tau^{-1}(\omega) )= \frac{G_{\tau}}{|\mathcal{T}_{\phi,e}|
 + \sum_{\tau' \in \mathcal{T}_{\phi,e}}  G_{\tau'} \mathbbm{1}_{\tau'(\Omega)}}(\omega) \quad & \text{for $\tau \in \mathcal{T}_{\phi,e},\, \omega \in \tau(\Omega)$}
\end{cases}.
\end{equation*}
\end{itemize}
Once the characteristic operator has been found, it is possible to follow the classical ideas of Stein's method to bound the distance between the conditional laws.
Let us remark that the explicit description of the dynamics associated to the Markov generator $\mathcal{A}$ turns out to be very useful in order to bound the derivatives of the solution to $ \mathcal{A} f =g $ by means of couplings.
In Section~\ref{sec:ale} in the context of random walk bridges we construct some ad hoc couplings, which may be of independent interest and, we believe, are among the novelties of this article. \\

The use of observation~\ref{A)} is to ``bootstrap'' a characteristic operator for a conditional distribution provided we know one for another, typically simpler, conditional distribution.
Indeed, assume the knowledge of the density $\frac{\De Q}{\De P}:=M$ and of a characteristic operator $\mathcal{A}$ for $P( \cdot |\phi=e)$ in the form~\eqref{eq:charop}.
Since $\mathcal{A}$ satisfies a kind of product rule
\begin{equation}\label{eq:genLeibniz} \mathcal{A}(fg) = f \mathcal{A}g + g \mathcal{A} f + \Gamma(f,g)
\end{equation}
with
\[ \Gamma(f,g) = \sum_{\tau \in \mathcal{T}_{\phi,e}  } (f \circ \tau -f)(g \circ \tau - g )+(f \circ \tau^{-1} -f)G_{\tau} (g \circ \tau^{-1} - g )\mathbbm{1}_{\tau(\Omega)}  \]
then we can write, for all $f$,

\begin{eqnarray*}
E_{Q(\cdot |\phi=e)}[ \mathcal{A} f  ] &\stackrel{\ref{A)}}{=}& \frac{1}{Z(e)} E_{P(\cdot |\phi=e)}[M \mathcal{A}f]\\
									&=& \frac{1}{Z(e)} E_{P(\cdot |\phi=e)}[\mathcal{A}(fM)- f (\mathcal{A}M) - \Gamma(f,M) ] \\
									&=& -\frac{1}{Z(e)} E_{P(\cdot |\phi=e)}[  M (\mathcal{A}f) + \Gamma(f,M) ]  \\
									&=&-E_{Q(\cdot |\phi=e)}[  \mathcal{A}f + \frac{1}{M} \Gamma(f,M) ],
\end{eqnarray*}
where we used that $P(\cdot |\phi=e)$ is the reversible measure for $\cA$ in order to write $E[f(\cA M)]=E[(\cA f)M]$ in the third equality.
Thus, the operator $\tilde{\mathcal{A}} = \mathcal{A}f + {1}/({2M})\, \Gamma(f,M)  $ is a characteristic operator for $Q(\cdot |\phi=e)$.
Clearly, the operator $\Gamma$ in~\eqref{eq:genLeibniz} depends a lot on the underlying space $\Omega$ and on the operator $\mathcal{A}$, and is typically easier to handle in continuous rather than discrete spaces.
For instance, when $\mathcal{A}$ is a diffusion operator, it is well known that $\Gamma(f,g)= \nabla f \cdot \nabla g$, so that ${1}/{M}\, \Gamma(f,M)= \nabla f \cdot \nabla \log M$.
In Section~\ref{sec:filtering} we will use the procedure just described in the context of filtering.

\subsubsection*{Conditional equivalence}

Let us reformulate~\ref{A)} in a slightly more accurate way.

\begin{enumerate}[label=\Alph*'),ref=\Alph*')]
\item\label{A')} If $Q$ is such that $Q(\phi =e)>0$, and $\frac{\De Q}{\De P}$ takes the form
\[ \frac{\De Q}{\De P}(\omega) = h( \phi(\omega ) )M(\omega)  \]
for some $h:E \rightarrow [0,\,+\infty)$ and $M:\Omega \rightarrow [0,\,+\infty)$, then $Q(\cdot|\phi=e)$-almost surely we have
\[ \frac{\De Q( \cdot | \phi =e)}{\De P( \cdot | \phi =e)}(\omega) = \frac{1}{Z(e)} M(\omega), \quad \text{with} \ Z(e) =  E_{P(\cdot|\phi=e)}(M). \]
\end{enumerate}
In addition to what could be deduced from~\ref{A)}, we can see that there may be different probabilities whose conditional laws are equal. It suffices that the density is \emph{measurable with respect to the observation}, i.e.
\[\frac{\De Q}{\De P}(\omega) = h(\phi(\omega)) \]
for some $h:E \rightarrow [0,\,+\infty)$. This is not so surprising, since conditioning is often seen as a kind of projection. Several explicit examples of conditional equivalence are known, especially for bridges, see for instance~\cite{Blee},~\cite{Cl91},~\cite{CL15},~\cite{Fitz}.  These considerations suggest that whatever bound is obtained for conditional probabilities, it has to be compatible with this equivalence in order to be satisfactory. That is, if it is of the form
\[ d (P(\cdot| \phi=e), Q(\cdot|\phi=e) ) \leq \mathcal{K}(P,Q,\phi,e), \]
 for some metric $d$ on the space of probability measures, then the ``function'' $\mathcal{K}$ has to be such that
 \[ \mathcal{K}(P,Q,\phi,e)=\mathcal{K}(P,Q',\phi,e) \]
 whenever $Q,\,Q'$ are conditionally equivalent in the sense above. A nice feature of the bounds we propose in this article  is that they comply with the compatibility requirement.

\subsection{A toy example: Poisson conditioned on the diagonal}\label{subsec:PoiDiag}
To illustrate more concretely the previous ideas we shall describe the special case of a two-dimensional vector with Poisson components conditioned to be on the diagonal of $\N^2$. Even though the computations are quite straightforward, this example can be considered paradigmatic, since it contains the key ideas behind our method.
\paragraph*{Finding the characteristic operator}
Let $\lambda_1,\,\lambda_2>0$ and $P\sim \mathrm{Poi}(\lambda_1)\otimes \mathrm{Poi}(\lambda_2)$ so that for this example $\Omega = \bbN^2$.
Let us set the observable $\phi(n_1,\,n_2):=n_1-n_2$. We are interested in  $P(\cdot|\phi = 0)$. Notice that such conditional law can be computed explicitly:
 \begin{equation}\label{eq:rep_Q_Poisson}
 P( (n_1,n_2) =(n,n) | \phi=0)=\frac{1}{\sum_{k\ge 0}\frac{{(\lambda_1\lambda_2)}^k}{{(k!)}^2}}\frac{{(\lambda_1\lambda_2)}^n}{{(n!)}^2}=\frac{1}{I_0(2\sqrt{\lambda_1\lambda_2})}\frac{{(\lambda_1\lambda_2)}^n}{{(n!)}^2}
\end{equation}
with $I_0$ the modified Bessel function of the first kind. However the knowledge of the distribution will not be needed below.
Our goal is to find a characteristic operator for the conditional probability exploiting observation~\ref{B)}. For this, consider the family of injections $ \mathcal{T}_0 = \{ \tau_1,\tau_2\}$ with $\tau_1(n) = n+{(1,\,0)}^T$, $\tau_2(n) = n+{(0,\,1)}^T$ for $n=(n_1,\,n_2)\in\N^2$. The change-of-measure formulas~\eqref{eq:tauchange} are well known, see  \citep{Chen:1975a}:
\begin{equation}\label{eq:Chen}
	E_P[F \circ \tau_i(n)]= E_P[F(n) n_i {\lambda_i}^{-1} ],\quad i=1,\,2
\end{equation}
for every bounded function $F:\,\bbN^2\to\R$.
However, neither $\tau_1$ nor $\tau_2$ are information preserving, i.e. $\phi \circ \tau_i \neq \phi $ for $i=1,2$. This is an example where iterating the formulas~\eqref{eq:Chen} helps in producing new ones, which in turn can be used to characterize the conditional law. Indeed, in the current setup, $\mathcal{T}$ defined in~\eqref{eq:enlargedinjections} is
\[ \mathcal{T} = {\{ \tau_{v} \}}_{v \in \mathbb{N}^2}, \quad \text{with} \, \, \tau_{v}(n) = n+v \]
and $\tau_w$ preserves the information when $w={(1,1)}^T$. The change of measure formula for $\tau_w$ under $P$ is easily derived concatenating~\eqref{eq:Chen} for $i=1,2$: for all $F$ bounded
\[  E_{P}[F (n+{(1,1)}^T) ] = {(\lambda_1 \lambda_2)}^{-1} E_{P}[F (n) \, n_1n_2]. \]
Moreover, one can check that the set $\mathcal{T}_{\phi,e} =\{\tau_{w} \}$ is connecting in the sense of~\eqref{eq:connecting}.
 Thus, the conditional law $P(\cdot  | \phi=0)$ is characterized by the change of measure formula
\[  E_{P(\cdot | \phi=0)}[F (n+{(1,1)}^T) ] = {(\lambda_1 \lambda_2)}^{-1} E_{P(\cdot | \phi=0)}[F (n)n_1 n_2]  \]
for $F$ bounded, and the Stein's operator for it is
\[\cA f(n_1,n_2) = (\lambda_1 \lambda_2)(f(n+{(1,1)}^T)-f(n))+ (f(n-{(1,1)}^T)-f(n))\mathbbm{1}_{\{n_1,n_2 \geq 1\}}.  \]
With a slight abuse of notation we identify $P(\cdot|\phi=0)$ with its push-forward through the map $(n,n)\mapsto n$ and regard it as a measure on $\bbN$. In this case $\cA$ acts on bounded $f:\bbN\to \bbR$ and reads
\begin{equation}\label{eq:projectedgen}
	\cA f(n):=  (\lambda_1 \lambda_2)(f(n+1)-f(n)) - n^2(f(n-1)-f(n))\mathbbm{1}_{\{n \geq 1\}},\quad n\in\bbN,
\end{equation}
i.e. $\cA$ is the generator of a birth-death chain with birth rate $(\lambda_1 \lambda_2)$ and death rate $n^2$.\newline

\paragraph{{Bounding the distance} } Assume we have two other parameters $\mu_1,\,\mu_2>0$, that  $Q \sim \mathrm{Poi}(\mu_1)\otimes \mathrm{Poi}(\mu_2)$ and that we wish to bound, say, the 1-Wasserstein  distance $d_{W,1}(P(\cdot | \phi=0),Q(\cdot|\phi=0))$. This situation falls in the framework of Remark~\ref{rem:perturbation_operator};
indeed a characteristic operator for $Q(\cdot | \phi=e)$ can be obtained as we did for $P(\cdot | \phi=e)$. Therefore we can deduce the following result.
\begin{lemma}
For all $\mu_1,\mu_2>0$ we have
	\[
		d_{W,1}(Q(\cdot| \phi=0),\, P(\cdot|\phi=0) )\le  9 \left|\lambda_1\lambda_2-\mu_1\mu_2\right|
	\]
\end{lemma}
\begin{proof}
	Let $f$ be the solution of the Stein's equation $\cA f = g$ with input datum a 1-Lipschitz function $g$.
	We can bound
	\[
		d_{W,1}(Q(\cdot| \phi=0),\, P(\cdot|\phi=0))\le \sup_{g\in \mathrm{Lip}_1}\left|E_{Q(\cdot | \phi=0)}\left[\cA f-\widetilde{\cA}f\right]\right|
	\]
	where $\widetilde{\cA}$ is the generator~\eqref{eq:projectedgen} with $\mu_1,\mu_2$ in place of $\lambda_1,\lambda_2$.
	Hence
	\begin{align}
		\left|E_{Q(\cdot | \phi=0)}\left[\cA f-\widetilde{\cA}f\right]\right| & \le \left|\lambda_1\lambda_2-\mu_1\mu_2\right|
		E_{Q(\cdot | \phi=0)}\left[ |f(n+1)-f(n)|\right]\nonumber  \\
		  & \le 9 \left|\lambda_1\lambda_2-\mu_1\mu_2\right|.
			\label{eq:combine}\end{align}
	In the last line we have used the bound $\sup_{n\in \N}|f(n+1)-f(n)| \leq 9 $ on the gradient of the Stein solution $f$; this bound can be deduced from Proposition~\ref{prop:SteinSol}, which we prove later on in the article with a coupling argument.
\end{proof}
Finally, let us observe that the bound obtained is compatible with what is known about conditional equivalence and mentioned in \ref{A')}. Indeed, in~\citet[Example 4.3.1]{Conforti:2015} it is shown that $P(\cdot|\phi=0) = Q(\cdot | \phi =0)$ if and only if $\lambda_1 \lambda_2 = \mu_1 \mu_2$. \newline

\paragraph{Structure of the paper.} The paper consists of two main parts. Section~\ref{sec:filtering} is devoted to the study of the classical one-dimensional filtering problem. We present the setup and preparatory results in Subsections~\ref{subsec:notation_fil}-\ref{subsec:non_lin_fil} and show our main Theorems in Subsections~\ref{subsec:main1_filt} and \ref{subsec:main2_filt}.

In Section~\ref{sec:ale} we are concerned with the study of bridges of random walks. We begin by considering random walks on the hypercube in Subsection~\ref{subsec:bridge01}. We then pass to the random walk on the euclidean lattice in Subsection~\ref{subsec:CTRW}, and extend the results to homogeneous and non-homogeneus (Subsection~\ref{subsec:revCTRW}) jump rates. We conclude by analysing the speed of convergence of a scheme approximating the continuous-time simple random walk in Subsection~\ref{subsec:schemesRW}. 

\section*{Notation}
We write $\N_0:=\N\cup \{0\}=\{0,\,1,\,2,\,\ldots \}$.
We denote by $\mathbb{D}([a,\, b], G)$, for $G$ a metric space,
the space of c\`adl\`ag paths on $[a,\,b]$ for the topology induced by $G$. $d_{W,\,1}$ denotes the $1$-Wasserstein distance.
When we have a piecewise-constant trajectory ${(X_t)}_{t\ge 0}$ we use the notation $X_{t^-}:=\lim_{s\uparrow t}X_s$.
The set of smooth and bounded functions on a set $X$ is called $C_b^\infty(X)$.
For functions $f,\,g$, we will use the abbreviation $fg(x):=f(x)g(x)$. The set of non-negative reals $[0,\,\infty)$ is called $\bbR_+$.
The set of all probability measures on a measurable space $(A,\cA)$ shall be denoted by $\mathcal{P}(A)$. The maximum between $a,\,b\in \R$ is denoted as $a\vee b$, and the minimum $a\wedge b$.
Given two measurable spaces $(X_1,\,\Sigma_1)$ and $ (X_2,\,\Sigma_2)$ the notation $f_\# \mu$ denotes the push-forward of the measure $\mu:\Sigma_1\to\R_+$ through the measurable function $f:\,X_1\to X_2$.


\section{Filtering problem}\label{sec:filtering}
\subsection{Setup and main result}\label{subsec:notation_fil}

\paragraph{The model.}In this Section, we consider the filtering problem in one dimension (see for example \citet[Chapter~6]{Oksendal:2003}). In this classical problem, one is interested in estimating the state of a $1$-dimensional diffusion process, the \emph{signal}, given the trajectory of another stochastic process, the \emph{observation},
which is obtained applying a random perturbation to the signal. More precisely, fix a time horizon $T>0$, $\alpha\in \R$ and denote by $C_0([0,T];\bbR^2)$ the set of continuous functions defined on $[0,T]$ with values in $\bbR^2$ and vanishing at zero.
We denote by ${(X_t,Z_t)}_{t\in[0,T]}$ the canonical process in $C_0([0,T];\bbR^2)$ which we endow with the canonical filtration.\\

We consider a first system of signal and observation $(X,\,Z)$ whose law $P$
on the space $C_0([0,T];\bbR^2)$ is governed by the system of equations
\be\label{eq:lin_fil}
\begin{cases}\De X_t=  \De V_t,& X_0=0,\\
\De Z_t=\alpha X_t \De t+\De U_t,& Z_0=0.
\end{cases}
\ee
We will call such a system the linear one. We then consider a second system whose law $P^b$ is characterized by the SDE
\begin{equation}\label{eq:non_lin_fil}
\begin{cases}\De X_t=  b(X_t)\De t+\De V_t,& X_0=0,\\
\De Z_t=\alpha X_t \De t+\De U_t,& Z_0=0,
\end{cases}
\end{equation}
and call it the non-linear system. Here $U,\,V$ are one-dimensional independent standard Brownian motions  under $P$ resp.~$P^b$. In filtering one is concerned with the study of the conditional laws $P_{z}$, $P^{b}_{z} \in \mathcal{P}(C_0([0,T];\bbR) )$ defined by
\[ P_z(\cdot):= P(X \in \cdot |Z =z) , \quad  P^b_z(\cdot):= P^b(X \in \cdot |Z =z),\]
where $z$ lies in a subset of $C_0([0,T];\bbR)$ such that both conditional laws are well defined. It is not hard to see that there exists a subset of measure $1$ for the Wiener measure where $z$ can be chosen.
Typical quantities of interest are the conditional mean, also known as filter, and the conditional variance. Since explicit calculations can be done only for the linear case and few others, it is common in applications to approximate systems as~\eqref{eq:non_lin_fil} through linear ones such as~\eqref{eq:lin_fil}.
This allows chiefly to ``forget'' the drift $b(\cdot)$ which naturally complicates the control on the conditional laws.

\paragraph{Quantifying the error in the linear approximation.}
Our goal is to understand how big the error we are making in neglecting the drift is. Thus, for a given $z$ we aim at finding bounds for $d_{W,1}(P^b_z,P_z)$, where $d_{W,1}$ is the 1-Wasserstein distance associated to the supremum norm $\| \cdot \|_{\infty}$ on $C_0([0,1];\bbR)$.
Although some assumptions on $b$ have to be made to provide concrete bounds, we stress that our aim is to look at cases outside the asymptotic regime where $b$ is a small perturbation. Actually, our analysis covers up to the case when $b(x)$ grows sublinearly.
Since we work with distances on the path space, our results allow to go well beyond the one-point marginals, and they apply to a much wider class of functionals than the conditional mean.
What we can say is that, under a sublinear growth assumption on $b(\cdot)$, the approximation can be explicitly given and depends on the behavior of the drift and its derivatives up to second order (Theorems~\ref{thm:main_filtering}-\ref{thm:main2_filtering}).

\paragraph{Notation.} We shall denote by $\varphi={(\varphi_t)}_{t\in[0,\, T]}$ the mean of the Gaussian process $P_z$, that is, $\varphi_t: = E_{P_z}[X_t]$ and by $\sigma_{s,t}$ its covariance, i.e. $\sigma_{s,t} = E_{P_z} \left[ (X_t-\varphi_t)(X_s-\varphi_s)  \right]$. When $s=t$, we simply write $\sigma_t$.
We define $P^0_z$ to be the centered version of $P_z$, that is, $P^0_z(X\in\cdot): = P_z(X - \varphi \in \cdot)$. Finally, we use the constant $1/Z$ to normalize a measure, and note that it may vary from occurrence to occurrence.
It will be clear from the context that we are not referring to the observation process ${(Z_t)}_{t\in[0,\,T]}$.

\paragraph{Main result.}
We assume that  $b : \bbR\to \bbR$ is twice continuosly differentiable and
\begin{enumerate}[label= (\roman*),ref= (\roman*)]
\item\label{item:(a)} there exists a constant $K\geq0$ and $\gamma \in (0,1)$ such that for all $x\in \bbR$
  \begin{equation}\label{eq:boundderivative}
   |b'(x)|\leq K{(1+|x|)}^{-\gamma}.
  \end{equation}
\item\label{item:(b)} there exists a constant $M\geq 0$ such that $\|b''\|_\infty\leq M$.
\end{enumerate}
In the following results we shall distinguish between the cases $\gamma \in(0,1/2)$ and $\gamma \in[1/2,1)$.
Under the above stated conditions we are able to prove the following bounds.

\begin{theorem}\label{thm:main_filtering} Let $\gamma \in[1/2,1)$. Almost surely in the random observation $z\in C_0([0,\,T],\,\bbR)$
\begin{align}
  d_{W,\,1}(P^b_z,P_z) &\leq E_{P^0_z}[\|X\|^2_\infty]\left \{|b(0)|+T\cW+ \frac{K}{1-\gamma} \cV^{1-\gamma} \right \}
  \label{eq:distance}
\end{align}
where
\begin{enumerate}[label= (\alph*),ref= (\alph*)]
\item\label{item:W} $\cW:= K|b(0)|+{K^2}/({1-\gamma})+{M}/{2}$,
\item\label{item:poly} $\cV$ is the maximal positive root of the polynomial
\begin{equation}\label{eq:polynomial_root}
p(x):= x^2- \zeta x^{2-\gamma}- \eta x-\sigma_T   .
\end{equation}

with
\begin{equation}\label{eq:etabeta}\eta:= \mathcal{W} \int_{0}^T \sigma_{s,T} \De s + \sigma_T|b(0)|+ |\varphi_T|, \quad \zeta := \frac{\sigma_T K}{1-\gamma}.\end{equation}
\end{enumerate}
\end{theorem}

\begin{theorem}\label{thm:main2_filtering}
Let $\gamma \in (0,\,{1}/{2})$. Almost surely in the random observation $z\in C_0([0,\,T],\,\bbR)$
\begin{equation*}
d_{W,1}(P^b_z,P_z) \leq E_{P^0_z} \left[ \| X\|^2_{\infty} \right] \left\{  |b(0)|  + (c_2+c_3)T+c_3 T^{1/2+\gamma} \mathcal{V}^{1-2\gamma} +c_1 \mathcal{V}^{1-\gamma}  \right\}
\end{equation*}
where
\begin{enumerate}[label= (\Roman*),ref= (\Roman*)]
\item the constants $c_1,\,c_2,\,c_3$ are defined by\begin{equation}
\label{eq:def_costanti}  c_1 := \frac{K}{1-\gamma}, \quad c_2 := K|b(0)|+\frac{M}{2}, \quad c_3 :=\frac{K^2}{1-\gamma}.
\end{equation}
\item\label{item:defCV} $\mathcal{V}$ is the largest positive root of the polynomial
\begin{align*}
 p(x) &:=x^2 - \bar{\sigma}- \left(\bar{\sigma}\sqrt{2}|b(0)|+ \bar{\sigma}\sqrt{2 T(c_2^2+2c_3^2)}+\Psi(\varphi)^{\frac12}\right) x \\
 &-2 \bar{\sigma} c_3 T^\gamma x^{2-2\gamma} -\sqrt{2}\bar{\sigma} c_1 x^{2-\gamma}
 \end{align*}
with $\Psi : C_0([0,1];\bbR) \to \bbR_+$ being defined by
\[
\Psi(X) := \int_{0}^T|X_s|^2 \De s +|X_T|^2 ,\quad X\in \Omega
\]
and
\begin{equation*}
 \bar{\sigma}:= \int_0^T \sigma_s \De s + \sigma_T.
\end{equation*}
\end{enumerate}
\end{theorem}

\begin{remark}[The bound is explicit]\label{rem_explicit_bound}
The bounds in Theorem~\ref{thm:main_filtering} and Theorem~\ref{thm:main2_filtering} are given in terms of the conditional mean $\varphi$ and covariances $\sigma_{s,t}$ for the linear system,
and the constants $M,\gamma,K$ from the hypothesis. Note that the functions $\varphi_t$ and $\sigma_{s,t}$  can be calculated explicitly, using \citet[Theorem 4.1 and Lemma 4.3]{hairer2005analysis}.
We have
\begin{eqnarray*}
\varphi_T &:=& \frac{1}{\cosh(T \alpha)} \int_0^T \sinh(s \alpha) \De Z_s, \\
  \sigma_{s,t} &\ldef & \frac{1}{2\alpha} \cdot  \frac{\sinh(\alpha T - \alpha |t-s|)-\sinh(\alpha T - \alpha (s+t))}{\cosh(\alpha T) }.
\end{eqnarray*}
Some explanation is due concerning $\cV$ and $E_{P^0_z}[\|X\|^2_\infty]$. For $\cV$, some simple algebraic manipulations allow to get explicit bounds as a function of $\zeta,\,\sigma_T$ and $\eta$. Concerning $E_{P^0_z}[\|X\|^2_\infty]$, we observe that it is independent of $z$ and that, drawing from the large literature about maxima of centered Gaussian random variables, several bounds for it can be derived.
 Thus the estimate in Theorem~\ref{thm:main_filtering} is totally explicit.
\end{remark}

\begin{remark}[A remark on the density bounds of~\cite{Zei:1988}] In the vast literature on filtering, especially relevant to our work is \citet[Theorem 1 and following Remark]{Zei:1988} which proves density bounds for the unnormalised one-time marginal density.
These may be in fact an alternative starting point to prove approximation results as the ones we present. Although these bounds are available in a more general setting than the one considered in Theorem~\ref{thm:main_filtering}, to obtain a quantitative result one must deal with the normalisation constant and estimate it.
Typically good bounds for such constant are very hard to obtain unless one works in an asymptotic regime whereas our approach is independent of normalisations, as pointed out in the Introduction. Moreover, our approximation results cover more than the one-time marginals.
\end{remark}

\paragraph{Outline of the proof}
The proof is done comparing a Stein operator for the linear and the non-linear filter, following Remark~\ref{rem:perturbation_operator}. Since the covariance structure and mean of the Gaussian process $P_z$ can be given explicitly, a Stein operator is readily obtained following~\citet{Meckes:2009Stein}.
However, for the sake of completeness, we will also provide an alternative derivation of this result, following point~\ref{A)}. A Stein operator for $P^b_z$ can then be obtained from a Stein operator for $P_z$ and Girsanov theorem, thus following~\ref{B)}.
Once we have the Stein operators, we need to estimate their difference, which involves studying the moments of the canonical process under $P_z^b$.
Note that Stein operators for both the linear and non linear filter may be deduced from~\cite{hairer2005analysis}, \cite{hairer2007analysis}; however, we will work with different characteristic operators, which naturally generalize the finite-dimensional approach of~\citet{Meckes:2009Stein}. We will distinguish our result into two cases, according to the exponent $\gamma$ being larger or smaller than $1/2$. This is due to the fact that for $\gamma\ge 1/2$ the quantity $\beta'(\cdot):=bb^\prime(\cdot)+b^{\prime \prime}(\cdot)/2$ is bounded, and therefore only an estimate on the one-time marginal $X_T$ is needed. In the complementary case instead, the estimate involves the whole trajectory, therefore we have to introduce a norm on the path space to evaluate the required moments.


\subsection{Linear filter}
For the linear case many results are already at our disposal. We think chiefly of~\cite{hairer2005analysis}, which gives formulas for the conditional mean and covariance, and characterizes $P_z$ as the invariant measure of an SPDE\@. For the sake of completeness, we would like to sketch how one can obtain the formulas for conditional means and covariances using the observations at the basis of this article.
To simplify the exposition we restrict the attention to the finite-dimensional case, determining the conditional distribution of a multivariate Gaussian.

Let $\cX=\R^N$, $\cZ=\R^M$, $\cY=\cX \otimes \cZ$ and $P$ be a Gaussian law on $\cY$.
We denote as $Y=(X,\,Z)$ the typical element of $\cY$, $\langle \cdot, \cdot \rangle$ the inner product on $\cY$ and $\langle \cdot, \cdot \rangle_{X},\langle \cdot, \cdot \rangle_{Z}$
the inner products on $\cX$ and $\cZ$ respectively. The covariance matrix and mean of $P$ are, in block form,
\[ \Sigma = \begin{pmatrix} \Sigma_{XX} & \Sigma_{XZ} \\ \Sigma_{ZX} & \Sigma_{ZZ} \end{pmatrix}, \quad m = \begin{pmatrix} m_X \\ m_Z \end{pmatrix} \]
Let us also define the matrix $\Gamma := \Sigma^{-1}$, for which we adopt the block notation as well
\[ \Gamma = \begin{pmatrix} \Gamma_{XX} & \Gamma_{XZ} \\ \Gamma_{ZX} & \Gamma_{ZZ} \end{pmatrix} \]
The following integration-by-parts formula can be seen as the ``limit'' as $\varepsilon \rightarrow 0$ of the change of measure~\eqref{eq:tauchange} for $\tau^{\varepsilon}_v=y+\varepsilon v$.
For all directions of differentiation $v=(v_X,v_Z)$ and test functions $f$ it holds that~\cite[Lemma 1 (1)]{Meckes:2009Stein}
\[ E_{P} \Big( \langle \nabla f(Y) ,v\rangle \Big) = E_{P} \Big(f \langle v,\Gamma (Y-m)  \rangle \Big) \]
 If we want to study $P_z=P(X \in \cdot | Z=z )$, we look at the transformations $\tau^{\varepsilon}_v$ associated to vectors of the form $(v_X,0)$.
 Using the notation above, the integration by parts can be rewritten for one such vector as
 \[ E_{P}( \langle \nabla^X  f, v_X \rangle_X  ) = E(f \langle v^X, \Gamma_{XX}(X-m^X) \rangle_X  + \langle v^X , \Gamma_{XZ}(Z-m_Z) \rangle_X  ).  \]
According to the general paradigma (namely~\ref{A)}), this formula characterises $P_z$. Upon setting $m_{X|Z} := m^X- \Gamma^{-1}_{XX} \Gamma_{XZ}(z-m_Z)$, it holds that

\[ E_{P_z}( \langle \nabla^X  f, v_X \rangle_X  ) = E_{P_z} \Big(f \langle v^X, \Gamma_{XX}( X- m_{X|Z}) \rangle_X \Big). \]
From this we deduce that $P_z$ is a Gaussian with mean $m_{X|Z}$ and inverse covariance matrix $\Gamma_{XX}$.
Using standard results for inverting block matrices we obtain that the mean of $P_z$ is
\[ m_{X|Z} = m_X + \Sigma_{XZ}\Sigma^{-1}_{ZZ}(z-m_Z) \]
and its covariance matrix is
\[\Sigma_{X|Z} = \Gamma_{XX}^{-1} = \Sigma_{XX}- \Sigma_{XZ}\Sigma^{-1}_{ZZ}\Sigma_{ZX}. \]
The same result is derived in greater generality in~\citet[Lemma 4.3]{hairer2005analysis}.

\subsection{Non-linear filter}\label{subsec:non_lin_fil}
\subsubsection{Lifting the Stein operator via densities from the linear to the non-linear filter}
As we saw in the Introduction, probability ratios are preserved by conditioning, and point~\ref{A)} informally states that Radon--Nikodym derivatives of conditional measures can be found easily once we know those of the unconditional laws.
In the context of the linear and non-linear filter point~\ref{A)} is translated into the following.
\begin{lemma}[Girsanov theorem for filters]\label{lem:Girsanov_filter}
The following holds for almost every $z$:
\begin{align} \dfrac{\De P^b_z}{\De P_{z} }(X) &=\frac{1}{Z}\exp\left(B(X_T)-\int_0^T \beta (X_s)\De s\right) \label{eq:def_J},\end{align}
where $B(\cdot)$ is a primitive of $b(\cdot)$ and $\beta(\cdot): =(b'+b^2)(\cdot)/2$.
\end{lemma}

This Lemma is not an original result of this article, see for instance \citet[Eq. (2.5) and Eq. (2.6)]{Zei:1988}. For this reason, we do not make its proof.

\subsubsection{Stein equation for the non-linear filter}\label{subsubsec:stein_non_linear}
Let $\Omega := C_0([0,T];\mathbb{R})$.
We say that a function $F: \Omega \rightarrow \mathbb{R}$ is $1$-Lipschitz if
\[
|F(X)-F(Y)| \leq \|X-Y \|_{\infty},\quad \text{for all } \, X,\,Y \in  \Omega.
\]
Let $\Phi$ be the set of smooth cylindrical functionals with bounded second derivative defined by
\begin{align*}
 \Phi \ldef   \Big \{ F:\Omega & \rightarrow \mathbb{R} : \, F(X) = f(X_{t_1},..,X_{t_N})\,\text{for some }N\in \N, \\
 &  \, 0\leq t_1<\cdots<t_N\leq T,\,
f \in C^{2}(\R^N)\text{ such that }\|f''\|_\infty <\infty\Big \},
\end{align*}
and let $\cS$ be the set of functions in  $\Phi$ that are also $1$-Lipschitz.
We set for any $F(X)= f(X_{t_1},..,X_{t_N})\in \Phi$ and for any $Y\in \Omega$,
\begin{equation}\label{eq:diffrel} DF(X)[Y] := \sum_{i=1}^N \partial_{i}f(X_{t_1},..,X_{t_N}) Y_{t_i}, \quad D^2F(X)[Y] := \sum_{i,j=1}^N \partial_{ij}f(X_{t_1},..,X_{t_N}) Y_{t_j} Y_{t_i}.
\end{equation}
As a remark, it is immediate to see that any $F\in \Phi$ is twice Frech\'et differentiable in $(\Omega,{\|\cdot \|}_\infty)$ and that  the derivatives correspond to those in~\eqref{eq:diffrel}.

Recalling that $\varphi$ is the mean of $P_z$, we define for any $F\in \Phi$ the operator
\begin{equation}\label{def:SPDEgenbrbr}
\scrA F(X) := -DF(X)[X - \varphi]+  E_{P^0_z} \left[D^2F(X)[\teX]\right],
\end{equation}
where the expectation is taken with respect to $\widetilde X\in \Omega$, that is,
\[
E_{P^0_z} \left[D^2F(X)[\teX]\right] = \int_{\Omega} D^2F(X)[\teX] \De P^0_z(\teX).
\]

\begin{lemma}
In the above setting, the following hold.

\begin{enumerate}[label= (\arabic*),ref= (\arabic*)]
\item\label{item:2fil} $P_{z}$ satisfies the integration-by-parts formula
\begin{equation}\label{IPBFOUsemigroup}
E_{P_z}(G \scrA F) = E_{P_z}( G \scrA F)
\end{equation}
for all $F,\,G \in \Phi$. In particular, $E_{P_z}(\scrA F) =0$ for all $F\in \Phi$.

\item\label{item:3fil} Let $F \in \mathscr{S}$ be such that $E_{P_z}[F]=0$. Then the equation
\[
\scrA G(X) = F(X)
\]
admits as solution
\begin{equation}\label{eq:solutionsteineq}
G(X) = - \int_{0}^{1}\frac{1}{2t} E_{P^0_z}\left[ F\left(\sqrt{t}X + \sqrt{1-t} \widetilde{X} +\left(1-\sqrt{t}\right) \varphi \right)\right ]\De t.
\end{equation}
Moreover, $G\in\mathscr{S}$.

\item\label{item:4fil} $P^{b}_{z}$ satisfies the formula
\begin{equation}\label{eq:Steinop_nonlinearfilt}
E_{P^{b}_z} \Big(\scrA_{b} F \Big) = 0
\end{equation}
for all $F\in \Phi$, where $\scrA_{b}$ is defined by
\begin{align*}
\nonumber \scrA_{b} F(X) \ldef \scrA F(X) -    E_{P^0_z} \Bigg[D F(X)[\widetilde X] \Bigg( \int_{0}^T  \beta'(X_s) \teX_{s} \De s -b(X_T)\teX_T\Bigg) \, \Bigg].
\end{align*}
\end{enumerate}
\end{lemma}
\begin{proof}
In the whole proof fix $F(X)=f(X_{t_1},\,\ldots,\,X_{t_N})$ with~$f\in C^2(\bbR^N)$ such that $\| f''\|_\infty<\infty$. Furthermore, set
$x:=(X_{t_1},..,X_{t_N})$, $\gamma\ldef (\varphi_{t_1},..,\varphi_{t_N})$, $p := {(X_{t_1},..,X_{t_N})}_{\#} P_z$ and $p^0 := {(X_{t_1},..,X_{t_N})}_{\#} P^0_z$.

Let us start with the proof of~\ref{item:2fil}.
If we define $\sigma_{ij}:= E_{{P}_z}(X_{t_i} X_{t_j} ) $ it is seen, using~\eqref{eq:diffrel}, that
\[    \scrA F(X) = \sum_{i,j=1}^N \sigma_{ij}  \partial_{ij}f(x) -\sum_{i=1}^N \partial_{i}f(x) (x_i  - \gamma_i).
\]
Observe now that $p$ is a Gaussian law on $\bbR^N$ with covariance matrix ${(\sigma_{ij})}_{1\leq i,j \leq N}$ and mean vector  $\gamma$. Thus,~\ref{item:2fil} is a simple consequence of the well-known results about finite dimensional Gaussian distributions.

Let us now show~\ref{item:3fil}. Since $F\in \cS$, we can rewrite~\eqref{eq:solutionsteineq} as
\[
G(X) = - \int_{0}^{1}\frac{1}{2t}\int_{\mathbb{R}^N} f(\sqrt{t}(x-\gamma)+ \sqrt{1-t} \widetilde{x} + \gamma) p^0(\De \widetilde{x})\rdef g(x)
\]
From the formula above, using that $f\in C^2(\bbR^N)$ and that $F$ is $1$-Lipschitz, it is straightforward to show that $g\in C^2(\bbR^N)$, that $\|g''\|_\infty<\infty$ and that $G$ is $1$-Lipschitz, in particular $G\in \mathcal{S}$; for more details we refer to \citet[Lemma 2]{Meckes:2009Stein}.

We now show that $G$ solves $\mathcal{A} G=F$. Since $p^0$ is a centered Gaussian law and $f$ is such that $E_{p^0}[f(\cdot+\gamma)] = E_{P_z}[F] = 0$,
an application of \citet[Lemma 1 (3)]{Meckes:2009Stein} shows that
$g(\cdot+\gamma)$ solves
\[
\sum_{i,j=1}^N \sigma_{i,j}\partial_{ij}g(x+\gamma) - \sum_i \partial_i g(x+\gamma) x_i = f(x+\gamma).
\]
We underline that Meckes's result, although stated for smooth functions, works when one requires the less restrictive condition $f\in C^2(\bbR^N)$. 
The change of variables $x\mapsto x-\gamma$ implies that $g$ solves
\[
Ag(x) \ldef \sum_{i,j=1}^N \sigma_{ij}\partial_{ij}g(x) - \sum_i \partial_i g(x) (x_i-\gamma_i) = f(x).
\]
The conclusion follows observing that for all $X \in \Omega$
\[
\scrA G (X) = Ag(x) = f(x) = F(X).
\]
To show~\ref{item:4fil}, we first observe that, according to Lemma~\ref{lem:Girsanov_filter} we have
\begin{equation}\label{eq:RN}
\frac{\De P^{b}_z}{ \De P_z}(X) \propto \exp\left(B(X_T) - \int_{0}^{T} \beta(X_s)\De s \right).
 \end{equation}
Next, for any $N$ define
\[
j^N(x_{1},\,\ldots,\,x_N) \ldef -\frac{T}{N} \sum_{i=1}^{N} \beta(x_i) + B(x_N), \quad J^N(X) \ldef j^N(X_{T/N},\ldots, X_{T}).
\]
We would like to use~\eqref{IPBFOUsemigroup} with $\exp(J^N)$, however $\exp(J^N)$ does not belong to $\Phi$ in general.
To circumvent this issue, we define for $\epsilon>0$ and $R>0$ the regularized function $j^N_{\epsilon,R}\ldef (\rho_\epsilon \ast j^N) \eta_R$, where $\rho_\epsilon$ is an approximation of the identity as $\epsilon\to 0$ and $\eta_R \in C^\infty(\bbR^N)$
is such that $\eta_R \equiv 1$ on $\{|x|\leq R\}$, $\eta_R \equiv 0$ on $\{|x|> R+1\}$, $0\leq \eta_R\leq 1$ and $\|\nabla \eta_R\|_\infty \leq 2$.
We set $J^N_{\epsilon,R}(X) \ldef j^N_{\epsilon,R}(X_{T/N},\ldots, X_{T})$, and observe that $\exp(J^N_{\epsilon,R})\in \Phi$.

Using~\eqref{eq:diffrel} and the definition of $\scrA$ we get the following equality, valid for all $X \in \Omega$:
\begin{align*}
\scrA(F(X) \exp(J^N_{\epsilon,R}(X)) ) & = (\scrA F(X)) \exp(J^N_{\epsilon,R}(X))  + F(X) (\scrA \exp(J^N_{\epsilon,R}(X))) \\ & + 2 E_{ P^0_z} \big[ DF(X)[\teX] DJ^N_{\epsilon,R}(X)[\teX]\big] \exp(J^N_{\epsilon,R}(X)).
\end{align*}
We employ to show that
\begin{align*}
E_{P_z}\left[(\scrA F) \exp(J^N_{\epsilon,R})\right ]
					& = E_{P_z}\Big[\scrA(  F \exp(J^N_{\epsilon,R}) ) - F (\scrA\exp(J^N_{\epsilon,R})) \Big]\\& -  2 E_{ P^0_z} \big[DF[\teX] DJ^N_{\epsilon,R}[\widetilde{X}] \big] E_{P_z}\Big[ \exp(J^N_{\epsilon,R})\Big].
\end{align*}
Using~\ref{item:2fil} and rearranging terms give
\begin{equation}\label{eq:prelimeps}
  E_{P_z}\Big[\scrA F \exp(J^N_{\epsilon,R}) + E_{ P^0_z} \big[DF[\teX] DJ^N_{\epsilon,R}[\widetilde{X}] \big]  \exp(J^N_{\epsilon,R}) \Big] = 0.
\end{equation}
Next, we send first $\epsilon\to 0$, using that $\rho_\epsilon\ast j^N$ (together with the gradient) converges uniformly on compact sets to $j^N$, and then $R\to \infty$ to obtain by dominated convergence that
\begin{equation}\label{eq:prelimitN}
E_{P_z}\Big[\scrA F \exp(J^N)+ E_{ P^0_z} \big[DF[\teX] DJ^N[\widetilde{X}] \big]  \exp(J^N)\Big] = 0.
\end{equation}
Dominated convergence is easily justified by the fact that $\|\eta_R\|_\infty\leq 1$, $\|\nabla \eta_R\|_\infty\leq 1$ and that for some constant $C>0$ and all $N\in \bbN$
\begin{equation}\label{eq:domconv}
 |\exp(J^N(X))|\vee \|DJ^N(X) \exp(J^N(X))\|\leq C(1+\|X\|_{\infty}^{1/2})\exp\Big(C(1+\|X\|_\infty^{2-\gamma})\Big),
\end{equation}
which follows from~\ref{item:(a)}-\ref{item:(b)} and Lemma~\ref{lem:est_on_b}. The right hand side of~\eqref{eq:domconv} is clearly integrable under the Gaussian measure $P_z$.
The fact that~\ref{item:4fil} holds follows by letting $N \rightarrow \infty$ in~\eqref{eq:prelimitN}, by~\eqref{eq:RN}, the definition of $J^N$ and dominated convergence as above.
\end{proof}

\subsection{Proof of Theorem~\ref{thm:main_filtering}}\label{subsec:main1_filt}
We need two preparatory Lemmas; the first one is a technical and rather straightforward estimate on the drift coefficient $b$.
\begin{lemma}\label{lem:est_on_b} Under~\ref{item:(a)}-\ref{item:(b)} we have the following inequalities valid for all $x\in\R$:
  \begin{align}
    |b(x)-b(0)| &\leq \frac{K}{1-\gamma}\left({(1+|x|)}^{1-\gamma}-1\right), \label{eq:boundzeroderivative}\\
      |\beta'(x)|=\left|bb'(x)+\frac{b''(x)}{2}\right|&\leq K\cdot|b(0)|+\frac{M}{2}+\frac{K^2}{1-\gamma}(1+|x|)^{1-2\gamma}.\label{eq:boundsimplified2}
\end{align}
In particular, for $\gamma \in[1/2,1)$
\begin{equation}\label{eq:boundsimplified}
  |\beta'(x)|\leq K\cdot|b(0)|+\frac{K^2}{1-\gamma}+\frac{M}{2} \rdef \cW.
\end{equation}
\end{lemma}
\begin{proof}
  We start with~\eqref{eq:boundzeroderivative}. We consider only the case $x>0$ as $x<0$ is completely analogous. By integration
  \[
  |b(x)-b(0)| \leq \int_0^x |b'(y)|\,\De y\leq \int_0^x K{(1+y)}^{-\gamma} \, \De y = \frac{K}{1-\gamma}\left({(1+x)}^{1-\gamma} -1\right),
  \]
  which leads to the conclusion. For what concerns~\eqref{eq:boundsimplified}, by using the triangular inequality,~\eqref{eq:boundzeroderivative} and the assumptions on $b'$ and $b''$ we get
  \begin{equation}\label{eq:boundinvariant}
    \left|bb'(x)+\frac{b''(x)}{2}\right|\leq K \cdot \left(|b(0)| - \frac{K}{1-\gamma}\right) {(1+|x|)}^{-\gamma} +\frac{K^2}{1-\gamma} {(1+|x|)}^{1-2\gamma} + \frac{M}{2}.
  \end{equation}
  The bound is readily obtained by recalling that $\gamma\in[1/2,1)$.
\end{proof}

\begin{remark}
  From~\eqref{eq:boundinvariant} we find that the following slightly improved estimate
  holds:
  \[
  \left|bb'(x)+\frac{b''(x)}{2}\right|\leq K \IND_{\left \{ |b(0)|\geq\frac{K}{1-\gamma}\right \}}\cdot\left(|b(0)|-\frac{K}{1-\gamma}\right)+\frac{K^2}{1-\gamma}+\frac{M}{2}.
  \]
\end{remark}
Next, we need a bound on the moments of $P^b_z$.
\begin{lemma}\label{lem:gamma_moments}
Let $\cV$ be as in~\ref{item:poly}. Then
\begin{equation}\label{eq:momentbndnnlinearfilt}
E_{P^b_z}\left[ |X_T|^{1-\gamma} \right] \leq \cV^{1-\gamma}.
\end{equation}
\end{lemma}
\begin{proof}
Let us consider a function of the form $F(X) = f(X_T)$. Then~\ref{item:4fil} reduces to
\begin{eqnarray*}
E_{P^b_z} \left[ \sigma_T f''(X_T) - f'(X_T) (X_T - \varphi_T) - f'(X_T) \Big( \int_{0}^T\beta'(X_s) \sigma_{s,T} \De s -\sigma_T b(X_T)  \Big)          \right] =0 .
\end{eqnarray*}
If we choose $f(x):=x^2/2$, then $F\in \Phi$ and we obtain, after rearranging some terms,
\begin{equation*}
E_{P^b_z} \left[ X^2_T \right] = \sigma_T +  \varphi_T E_{P^b_z} \left[ X_T \right] -  E_{P^b_z} \left[X_T \,\int_{0}^T\beta'(X_s) \sigma_{s,T} \De s \right] + \sigma_T E_{P^b_z} \left[ X_T b(X_T) \right]
\end{equation*}
Using the bounds~\eqref{eq:boundzeroderivative}, \eqref{eq:boundsimplified} and $(1+|x|)^{1-\gamma} \leq |x|^{1-\gamma} + 1 $ we get the inequality
\begin{equation*}
 E_{P^b_z} \left[ X^2_T \right] \leq \sigma_T +  \eta E_{P^b_z} \left[ |X_T|    \right]+ \zeta E_{P^b_z} \left[ |X_T|^{2-\gamma}    \right]
\end{equation*}
with $\eta,\,\zeta$ as in \eqref{eq:etabeta}.
Using Jensen's inequality and setting $x:= E_{P^b_z} {\left[ X^2_T \right]}^{1/2}$ we obtain
\[x^2 \leq \sigma_T + \eta x + \zeta x^{2-\gamma}, \]
from which it follows that $x \leq \cV$. The desired conclusion then follows with another application of Jensen's inequality.
\end{proof}
We are ready to give the final proof.
\begin{proof}[Proof of Theorem~\ref{thm:main_filtering}]
First we notice that, with an approximation argument, the Wasserstein distance can be computed by taking the supremum over the set $\mathcal{S}$ defined in Subsubsection~\ref{subsubsec:stein_non_linear}, instead of all $1$-Lipschitz functions. In the spirit of Remark~\ref{rem:perturbation_operator}, as a consequence of~\ref{item:3fil}-\ref{item:4fil} and the previous observation, we obtain
\begin{equation}\label{eq:from_Stein}
d_{W,\,1}\left(P_z,\,P^b_z\right)\le\sup_{G\in \mathcal S} \left|E_{P^b_z} {E}_{P^0_z}\left[DG(X)[\widetilde{X}] \Bigg( \int_{0}^T  \beta'(X_s) \teX_{s} \De s -b(X_T)\teX_T\Bigg)\right]\right|.
\end{equation}
  Since $G$ is 1-Lipschitz, $|DG(X)[\teX]| \leq \|\teX \|_{\infty}$. Combining this with~\eqref{eq:boundsimplified} and some standard calculations we see that the right hand side of~\eqref{eq:from_Stein} can be bounded above by
  \begin{equation}\label{eq:estimate}
    E_{P^0_z}[\|\widetilde{X}\|^2_\infty]\left[ E_{P^b_z}[|b(X_T)|] + T\cW \right].
  \end{equation}
Using the bound~\eqref{eq:boundzeroderivative} we are left with computing
\begin{equation}\label{eq:first_bound_b}
 E_{P^b_z}[|b(X_T)|]\leq |b(0)| + \frac{K}{1-\gamma} E_{P^b_z}\Big[{(1+|X_T|)}^{1-\gamma}-1\Big]\le |b(0)| + \frac{K}{1-\gamma} E_{P^b_z}\Big[{|X_T|}^{1-\gamma}\Big]
\end{equation}
being $\gamma<1$. Thanks to Lemma~\ref{lem:gamma_moments} we have $E_{P^b_z}\Big[{|X_T|}^{1-\gamma}\Big] \leq \cV^{1-\gamma} $, from which the conclusion follows.
\end{proof}
\subsection{Proof of Theorem~\ref{thm:main2_filtering}}\label{subsec:main2_filt} The proof of the Theorem is based on Lemmas \ref{lem:est_on_b} and \ref{lem:momentsmallgamma}. Define the constants $c_1,\,c_2,\,c_3$ by \eqref{eq:def_costanti}.
Then from Lemma~\ref{lem:est_on_b} we deduce
\begin{align}
|b(x)-b(0)|&\le c_1|x|^{1-\gamma},\label{eq:smallgammabound1}\\
|\beta'(x)|&\le c_2+c_3(1+|x|)^{1-2\gamma}.\label{eq:smallgammabound2}
\end{align}
In the next Lemma, we aim at finding a bound for $E_{P^b_z}[\Psi]$.

\begin{lemma}\label{lem:momentsmallgamma}
We have
\begin{equation}
E_{P^b_z}[\Psi(X)]^{\frac{1}{2}} \leq \mathcal{V}
\end{equation}
where $\cV$ has been defined in \ref{item:defCV}.
\end{lemma}
\begin{proof} To obtain a bound for $\Psi$ we shall use the fact that $P^b_z$ is invariant for $\mathscr{A}_b$. To be precise, by considering the Riemann sum approximation to the integral part of $\Psi$, applying \eqref{eq:Steinop_nonlinearfilt} and then passing to the limit under the integral sign we get
\begin{align}\label{eq:intbyparts}
  E_{P^b_z}\left[E_{P^0_z}[D^2\Psi(X)[\widetilde{X}]]\right]& - E_{P^b_z}[D\Psi(X)[X-\varphi]]  \nonumber\\ &= E_{P_b^z}\left[ E_{P^0_z} \left[ D \Psi(X)[\widetilde{X}] \left( \int_{0}^T \beta'(X_s) \widetilde{X}_s \De s  + b(X_T) \widetilde{X}_T\right)  \right]\right].
\end{align}
Here $D\Psi(X)[\widetilde{X}]$ is meant to be the Fr\'echet derivate of $\Psi$ at $X$ in the direction $\widetilde{X}$.

A simple calculation gives
\[ D \Psi(X)[\widetilde{X}] = 2\int_0^T X_s \widetilde{X}_s \De s +2 X_T\widetilde{X}_T, \quad D \Psi(X)[X] = 2 \Psi(X) ,\quad  D^2 \Psi(X)[\widetilde{X}] = 2 \Psi(\widetilde{X}). \]
This entails in particular
\begin{equation}\label{eq:momentsmallgamma1}
E_{P^0_z}\left[ D^2 \Psi(X)[\widetilde{X}] \right] =2 \int_{0}^T \sigma_s \De s + 2 \sigma_T\rdef 2 \bar{\sigma}.
 \end{equation}
For convenience, call $\Theta$ the right-hand side of~\eqref{eq:intbyparts}. Then, through rearranging, taking absolute values and using Cauchy--Schwartz we obtain
\begin{equation}\label{eq:stepsome}
  E_{P^b_z}[\Psi(X)] - \bar{\sigma} \leq  E_{P_z^b}[\Psi(X)]^{\frac{1}{2}}\Psi(\varphi)^{\frac{1}{2}} + \frac{|\Theta|}{2}
\end{equation}
 where we used that $E_{P_z^b}[|D\Psi(X)[\varphi]|]\leq 2 E_{P_z^b}[\Psi(X)]^{{1}/{2}}\Psi(\varphi)^{{1}/{2}}$.

Let us now look at $|\Theta|/2$. First observe that, using the explicit form of $D \Psi(X)[\widetilde{X}] $, we get
\begin{align*}
&\frac{1}{2}\left|E_{P^0_z} \left[ D \Psi(X)[\widetilde{X}] \left( \int_{0}^T \beta'(X_s) \widetilde{X}_s \De s  + b(X_T) \widetilde{X}_T\right)  \right] \right| \\
&\qquad\leq \int_{0}^T\int_{0}^T| X_s| |\beta'(X_r)| |\sigma_{s,r}| \De s \De r +| X_T| \int_{0}^T|\beta'(X_r)| |\sigma_{r,T}| \De r\\
&\qquad+|b( X_T)| \int_{0}^T|X_s| |\sigma_{s,T}| \De s + \sigma_T |b(X_T)||X_T|.
\end{align*}
Using repeatedly the inequality $|\sigma_{r,s}| \leq \sigma^{1/2}_s \sigma^{1/2}_r$, Cauchy--Schwartz and some algebraic manipulation allow to bound the
above expression by
\begin{equation}\label{eq:momentsmallgamma2} \bar{\sigma} \left(\int_{0}^{T} |X_s|^2\De s + |X_T|^2 \right)^{\frac{1}{2}} \left(  \int_{0}^{T} |\beta'(X_s)|^2 \De s + |b(X_T)|^2 \right)^{\frac{1}{2}}.
\end{equation}
Taking the expectation with respect to $P_z^b$ in~\eqref{eq:momentsmallgamma2} and using Cauchy--Schwartz gives that $|\Theta|/2$ is bounded above by
\[
\bar{\sigma} E_{P_z^b}[\Psi(X)]^{\frac{1}{2}} E_{P_z^b}\left[ \int_{0}^{T} |\beta'(X_s)|^2 \De s  + |b(X_T)|^2 \right]^{1/2}.
\]
We now use the bounds~\eqref{eq:smallgammabound1},~\eqref{eq:smallgammabound2} and the simple inequalities $$
\begin{array}{lr}
(a+b)^2\le 2a^2+2b^2 & a,\,b\in \R,\\
(1+a)^{2-4\gamma}\le 2+2a^{2-4\gamma} & a\ge 0
\end{array}
$$
to obtain
\begin{align*}
\int_{0}^{T}& |\beta'(X_s)|^2 \De s  + |b(X_T)|^2  \\
&\leq 2c_2^2 T+4c_3^2\left(T+\int_0^T |X_s|^{2-4\gamma}\De s\right)+2 |b(0)|^2+2 c_1^2 |X_T|^{2-2\gamma}\\
&\leq 2\left(|b(0)|^2+(c_2^2+2c_3^2) T\right)+ 4 c_3^2 T^{2\gamma} \left(  \int_{0}^T |X_s|^2 \De s \right)^{1-2\gamma} + 2 c_1^2 |X_T|^{2-2\gamma}\\
&\leq 2\left(|b(0)|^2+(c_2^2+2c_3^2) T\right)+ 4 c_3^2 T^{2\gamma} \Psi(X)^{1-2\gamma} + 2 c_1^2 \Psi(X)^{1-\gamma},
\end{align*}
where in the second inequality we used Jensen's inequality.
Thus we can bound $|\Theta|/2$ by
\begin{equation}
  \bar{\sigma} \sqrt{2} E_{P_z^b}[\Psi(X)]^{\frac{1}{2}} \left(|b(0)|^2+(c_2^2+2c_3^2) T + 2c_3^2 T^{2\gamma} E_{P_z^b}[\Psi(X)]^{1-2\gamma} +  c_1^2 E_{P_z^b} [\Psi(X)]^{1-\gamma} \right)^{\frac{1}{2}}.
\end{equation}
Finally, setting $x := E_{P^b_z}[ \Psi(X)]^{\frac{1}{2}}$, and incorporating the above bound in~\eqref{eq:stepsome}, we arrive at the inequality
 \begin{eqnarray}
x^2 - \bar{\sigma} \leq x  \Psi(\varphi)^{\frac{1}{2}}  + \bar{\sigma}\sqrt{2} x \left(|b(0)|^2+(c_2^2+2c_3^2) T + 2c_3^2 T^{2\gamma}  x^{2-4\gamma} +  c_1^2 x^{2-2\gamma} \right)^{\frac{1}{2}}.
 \end{eqnarray}
Via the inequality $(a_1^2+\ldots +a_k^2)^{1/2}\leq a_1+\ldots+a_k$ we get that $x$ satisfies
\[ x^2 \leq \bar{\sigma}+ \left(\bar{\sigma}\sqrt{2}|b(0)|+ \bar{\sigma}\sqrt{2 T(c_2^2+2c_3^2)}+\Psi(\varphi)^{\frac12}\right) x +2 \bar{\sigma} c_3 T^\gamma x^{2-2\gamma} +\sqrt{2}\bar{\sigma} c_1 x^{2-\gamma} \]
from which the conclusion follows.
\end{proof}
We have now gathered all the tools to show the final Theorem concerning filtering.
\begin{proof}[Proof of Theorem~\ref{thm:main2_filtering}]
As before we get
\begin{equation}
d_{W,1}(P^b_z,P_z) \leq  E_{P^0_z}[ \|X\|^2_{\infty}]  E_{P^b_z}\left[ \int_0^T |\beta'(X_s)|\De s  + |b(X_T)|\right],
\end{equation}
which can be bounded thanks to~\eqref{eq:smallgammabound1} and~\eqref{eq:smallgammabound2} by
\begin{align*}
E_{P^0_z}&[ \|X\|^2_{\infty}] \left(|b(0)| + c_1 E_{P^b_z}[|X_T|^{1-\gamma}]+(c_2+c_3) T + c_3 E_{P^b_z}\left[\int_0^T |X_s|^{1-2\gamma}\De s\right] \right) \\
&\leq E_{P^0_z}[ \|X\|^2_{\infty}] \left(|b(0)| + c_1 E_{P^b_z}[\Psi(X)]^{\frac{1-\gamma}{2}}+(c_2+c_3) T + c_3 T^{{1}/{2}+\gamma} E_{P^b_z}[\Psi(X)]^{\frac{1-2\gamma}{2}} \right).
\end{align*}
As $E_{P^b_z}[\Psi(X)]^{{1}/{2}}\leq \cV$ by Lemma~\ref{lem:momentsmallgamma}, we have shown our result.
\end{proof}

\section{Random walk bridges}\label{sec:ale}
\subsection{Bridge of the random walk on \texorpdfstring{${\{0,1 \}}^d$}{}}\label{subsec:bridge01}
\subsubsection{Setting and notation}\label{subsubsec:not_hyper}

In this Subsection we are interested in studying a continuous time random walk on the hypercube ${\{0,\,1\}}^d$, $d\ge 1$. We assume that the walker jumps in the direction $e_i$ with rate $\alpha_i\geq 0$, $i=1,\ldots,d$.
To obtain bounds on this object, we will start with the walk on $\{0,\,1\}$ and then use the fact that the random walk on the $d$-dimensional hypercube is a product of $1$-dimensional random walks.
We denote by $P$ the law on the space of c\`adl\`ag paths $\bbD([0,1];\{ 0,\,1\})$ of the continuous time random walk $X$ on $\{ 0,\, 1\}$ with jump rate $\alpha$ and time horizon $T=1$. The bridge of the random walk from and to the origin is given by
\[
 P^{00}(\cdot) \ldef P(\,\cdot \,| X_0 =0,\, X_1=0).
\]
We observe that the space of c\`adl\`ag paths with initial and terminal point at the origin, which we denote by $\bbD_0([0,1];\{ 0,\,1\})$, is in bijection with the set of all subsets of $(0,1)$ with even cardinality,
\[
 \mathscr{U}:= \{ U \subseteq (0,1):\, |U|<+\infty,\, |U| \in 2\N \},
\]
where, for a set $A$, $|A|$ denotes its cardinality.
In fact, the bijection is simply given by the map $\bbU : \bbD_0([0,1];\{ 0,\,1\})\to \mathscr{U}$ that associates to each path its jump times; we denote its inverse by $\bbX \ldef \bbU^{-1}$.
We shall endow $\mathscr{U}$ with the $\sigma$-algebra $\cU$ induced by $\bbU$, that is, we say that $A\in \cU$ if and only if $\bbX^{-1}(A)$ belongs to the Borel $\sigma$-algebra of $\bbD_0([0,1];\{ 0,\,1\})$.
With a little abuse of notation, we will still denote by $P^{00}$ the probability measure on $(\mathscr{U},\cU)$ given by the pushforward of $P^{00}$ via $\bbU$.
Note that since $\bbU$ is only defined on $\bbD_0([0,1];\{0,1\}) \subsetneq  \bbD([0,1];\{0,1\})$ and $P^{00}$ is a measure on $\bbD([0,1];\{0,1\})$, the pushforward may not be well-defined.
However here we do not have to worry since $P^{00}$ is supported on $\bbD_0([0,1];\{0,1\})$.

In order to characterize $P^{00}$ as the unique invariant distribution of a given generator, we introduce a set of perturbations of $\mathscr{U}$ which allows the complete exploration of the support. For $r\neq s \in (0,1)$, we define $\Psi_{r,s} : \mathscr{U}\to \mathscr{U}$ by
\[
 \Psi_{r,s}(U) \ldef \begin{cases}
  U\cup \{r,s\},      & \mbox{ if } \{r,s\}\cap U = \emptyset, \\
  U\setminus \{r,s\}, & \mbox{ if }  \{r,s\}\subset U,         \\
  U,                  & \mbox{ otherwise. }
 \end{cases}.
\]
\begin{remark}
 Let $U\in \mathscr{U}$ be the set of jump times of a sample path $X\in \bbD_0([0,1];\{ 0,\,1\})$. It is easy to see that $\Psi_{r,s} U$, $r<s$,
 corresponds to the path $X+\mathbbm{1}_{[r, s)}$ if $\{r,s\}\cap U = \emptyset$, to  $X-\mathbbm{1}_{[r, s)}$ if $\{ r,s\}\subset U$ and to $X$ otherwise.
\end{remark}

For convenience in the exposition, we will need the following additional notation.
\begin{itemize}
 \item $\mathscr{A} := \{ (r,s)\in {(0,1)}^2:\, r<s\}$.
 \item For $U \in \scrU$, we denote by ${[U]}^2:= \{ (r,s) \in \mathscr{A}:\, r,s\in U \}$. In words, ${[U]}^2$ is the set of pairs of elements of $U$.
\end{itemize}

\paragraph{Choice of the distance}\label{par:choice_dist}
We equip $\scrU$ with the graph distance $d$ induced by $\Psi$. That is, we say that $U$ and $V$ are at distance one if and only if there exist $(r,s)\in \mathscr{A} $ such that $ \Psi_{r,s}V = U$.
The distance between two arbitrary trajectories $U,V\in\mathscr{U}$ is defined to be the length of shortest path joining them.
It is worth to remark that $\mathscr{U}$ is a highly non-trivial graph, as every vertex has uncountably many neighbors. Nonetheless, the distance is well-defined: by removing one pair after the other, we notice that any $U\in \mathscr{U}$ has distance $|U|/2$ from the empty set.
It follows in particular that the graph is connected.
\subsubsection{Identification of the generator}
As stated in the Introduction, our goal is to obtain a Markovian dynamics stemming from a change-of-measure formula, already present in~\citet[Example 30]{CR}. We can exploit it to obtain the following proposition.
\begin{prop}\label{prop:hypergen}$P^{00}$ is the only invariant measure of a Markov process ${\{ U_t \}}_{t\geq0}$ on $\scrU$ whose generator is
 \begin{equation}\label{eq:hypergen}
  \mathscr{L}f(U):= \alpha^2\int_{\mathscr{A}} \left(f(\Psi_{r,s} U ) - f(U) \right)\De r  \De s + \sum_{A \in {[U]}^2}\left( f(\Psi_{A}U)-f(U)\right)
 \end{equation}
 for all $f:\mathscr{U}\to \bbR$ bounded measurable functions.
\end{prop}

\begin{proof}
To show that $P^{00}$ is invariant for $\mathscr{L}$, we show that for any bounded measurable function
\[ E_{P^{00}}( \mathscr{L} f ) =0, \] which yields the conclusion.
An application of~\citet[Theorem 12]{CR} gives a characterization of $P^{00}$ as the only measure on $\bbD([0,1];\{ 0,1\})$ such that $P^{00}(X_0=X_1=0)=1$ and for all bounded measurable functions $F:  \bbD([0,1];\{ 0,1\} ) \times \scrA \rightarrow \bbR$
\[\alpha^2 E_{P^{00}} \left( \int_{\mathscr{A}} F(X + \mathbbm{1}_{[r,s)}, r,s) \De s \De r  \right) = E_{P^{00}}  \left( \sum_{r<s, r,s \in \bbU(X) } F(X,r,s) \right), \]
where the symbol $+$ stands for the sum in $\bbZ/2\bbZ$ and

\begin{equation*}
{(X+\mathbbm{1}_{[r,s)})}_t= \begin{cases} X_t+1 \quad & \mbox{if $ t \in [r,s)$ } \\
X_t \quad & \mbox{otherwise.}
\end{cases}
\end{equation*}

Passing to the image measure, that is, considering functionals of the type $F(X,r,s)=G(\bbU(X),r,s)$ we obtain that for all $G:\mathscr{U}\times \scrA\rightarrow \bbR$ bounded and measurable
\begin{equation}\label{eq:hypergen1} \alpha^2 E_{P^{00}} \left( \int_{\mathscr{A}} G(\Psi_{r,s}U,r,s) \De s \De r  \right) = E_{P^{00}}  \left( \sum_{ (r,s) \in {[U]}^2 } G(U,r,s) \right),
\end{equation}
where we took advantage of the fact that, for any $X \in \bbD_0([0,1];\{0,1\})$, we have that $\bbU\big(X+ \mathbbm{1}_{[r,s)}\Big) = \Psi_{r,s}U $ for almost every $r,s \in \scrA$.
If we now fix $f:\scrU \rightarrow \bbR$ bounded and measurable, define $ G(U,r,s) = f(U) -f(\Psi_{r,s}U)$ and plug it back into~\eqref{eq:hypergen1}, we obtain the desired result, observing that $\Psi_{r,s}(\Psi_{r,s} U) = U$. We do not prove uniqueness here, as it is implied by Proposition~\ref{prop:lip}, which we prove later.
\end{proof}

In the next pages we will construct explicitly a dynamics ${(U_t)}_{t\geq 0}$ starting in $U\in \mathscr{U}$ whose infinitesimal generator is $\mathscr L$. We denote by $\bbP^U$ the law of such process, by $\bbE^U$ the corresponding expectation and, for any $f:\mathscr{U}\to \bbR$ bounded measurable function, by
\[
S_t f(U) \ldef \bbE^U\left[f(U_t)\right]
\]
the semigroup associated to ${(U_t)}_{t\geq 0}$.
The proof that $\mathscr{L}$ is characterizing boils down to showing that for any $f \in \mathrm{Lip}_1(\mathscr{U})$ such that $E_{P^{00}}[f]=0$, the Stein equation
\[
\mathscr{L} g = f,
\]
has a solution.
This is achieved with the following fundamental proposition.

\begin{prop}\label{prop:lip} For any  $f\in \mathrm{Lip}_1(\mathscr U)$, all $U,V\in \mathscr{U}$ and all $t\geq 0$
\begin{equation}\label{eq:bound}
  \left| S_t f(U)-S_t f(V)\right| \leq (4\exp(-t/2)+\exp(-t)) d(U,V).
\end{equation}
\end{prop}

The proof of Proposition~\ref{prop:lip} is based on a coupling argument. It will suffice to construct two Markov chains ${(U_t)}_{t\geq 0}$, ${(V_t)}_{t\geq 0}$ with generator $\mathscr{L}$ starting from neighbouring points $U,V\in \mathscr{U}$ such that $U_t$, $V_t$ are at most at distance two and coalesce within an exponentially distributed time.

As a remarkable consequence of Proposition~\ref{prop:lip} we can show that for any probability measure $\nu\in \cP(\mathscr{U})$,
the measure $\nu_\# S_t$, determined by $\nu_\# S_t(A) \ldef E_\nu[S_t \mathbbm{1}_A]$, converges exponentially fast to $P^{00}$ in the $1$-Wasserstein distance on $(\mathscr{U},d)$.
In particular, this implies that for any $f\in \mathrm{Lip}_1(\mathscr{U})$ with $E_{P^{00}}[f]=0$ the function
\begin{equation}\label{eq:sol}
g(U)\ldef -\int_0^\infty S_t f(U)\,\De t,\quad U\in \mathscr{U}
\end{equation}
is well-defined and solves the Stein equation $\mathscr{L} g = f$ (see Proposition~\ref{prop:SteinSol} below). This allows for the following quantitative estimate of the distance between two bridges of random walks on the hypercube with different jump rates.

\begin{prop}\label{prop:1wassone} Let $P^{00}$ and $Q^{00}$ be the law on $\mathscr{U}$ of the bridges from and to the origin of random walks on $\{ 0, 1\}$ with rates $\alpha$ and $\beta$ respectively. Then,
\[
 d_{W,\,1}\left(P^{00},\,Q^{00}\right)\le \frac{9}{2}\left|\alpha^2-\beta^2\right|.
\]
\end{prop}
\begin{proof} We shall see that the proof is an easy application of Proposition~\ref{prop:lip} and of~\eqref{eq:sol}. To simplify notation, let us write $P$ and $Q$ rather than $P^{00}$ and $Q^{00}$.
Let $\mathscr L^Q$, $\mathscr L^P$ be as in~\eqref{eq:hypergen} with associated semigroup ${(S^Q_t)}_{t\ge 0}$ and ${(S^P_t)}_{t\ge 0}$.
By definition of Wasserstein distance we have that
\[
d_{W,1}(P,Q) = \sup_{f\in \mathrm{Lip}_1(\mathscr{U}), E_{P}[f]=0} \left|E_{Q}[f]\right|.
\]
Next, fix any $f\in \mathrm{Lip}_1(\mathscr{U})$ such that $E_{P}[f]=0$. We  have that $\mathscr{L}^P g = f$, where $g$ is given by~\eqref{eq:sol}.
Using $E_Q[\mathscr{L}^Q g]=0$, Tonelli's theorem and invoking Proposition~\ref{prop:lip} we deduce that
 \begin{align*}
  \left|E_{Q}[f]\right| & =\left|E_{}[\mathscr L^Q g-\mathscr L^P g]\right| \\ & \le \left|\alpha^2-\beta^2\right|\int_0^\infty E_{Q}\left[\int_{\mathscr{A}} |S^P_t f(\Psi_{r,s}U)-S^P_t f(U)|\,\De r \De s\right]\, \De t \\
  &\leq \left|\alpha^2-\beta^2\right| \int_0^\infty \int_\mathscr{A} 4\exp(-t/2)+\exp(-t) \,\De r \De s\, \De t\\
  & = \frac{9}{2}\left|\alpha^2-\beta^2\right|,
 \end{align*}
 which is a uniform bound in $f$ and thus proves the Proposition.
\end{proof}
\begin{remark}
  The bound obtained is compatible with what is known about conditional equivalence. In fact it is shown in~\citet{CR} that two random walks on $\{ 0,1\}$ with jump rates $\alpha$ and $\beta$ have the same bridges if and only if $\alpha=\beta$.
\end{remark}
\begin{remark} Clearly, the same inequality of Proposition~\ref{prop:1wassone} holds also for $P^{00}$ and $Q^{00}$ as measures on
$\bbD_0([0,1];\{ 0,1\})$ with metric $d_\bbD(X,Y) \ldef d(\bbU(X),\bbU(Y))$ for all paths $X,Y\in \bbD_0([0,1];\{ 0,1\})$.
Here, $\bbU$ is the bijection between $\bbD_0([0,1];\{ 0,1\})$ and $\mathscr{U}$ described above.
\end{remark}

\begin{remark}[Extensions]
  The scope of application of Proposition~\ref{prop:lip} and Proposition~\ref{prop:hypergen} can go well beyond comparing two walks with homogeneous jump rates.
  Arguing as in Subsection~\ref{subsec:revCTRW} and Subsection~\ref{subsec:schemesRW}, it is possible to derive distance bounds between simple random walk bridges on the hypercube and bridges of random walks with non-homogeneous and possibly time-dependent rates, as well as to show convergence rates for certain approximation schemes. Another extension one may want to consider is to bridges whose terminal point is different from the origin.
  For brevity we do not include in this paper such bounds as they do not present any additional difficulty with respect to those for bridges of walks on $\bbZ$.
\end{remark}

Proposition~\ref{prop:1wassone} can be easily extended to random walks on the $d$-dimensional hypercube. In fact, we have the following corollary of which we only sketch the proof.
\begin{cor}\label{cor:higher_dim}
Let $d\ge 2$ and let $P^{0,\,d}$ and $Q^{0,\,d}$ be the laws of two bridges of random walks on ${\{0,\,1 \}}^d$ with jump rates $\alpha_i$ resp. $\beta_i$ in the direction $e_i$, for $i = 1,\ldots,d$. Then
\[
 d_{W,\,1}(P^{0,\,d},\,Q^{0,\,d})\le \frac{9}{2}\sum_{i=1}^d\left|\alpha_i^2-\beta_i^2\right|.
\]
where the Wasserstein distance is taken on $(\mathscr{U}^d,d_{\mathscr{U}^d})$ with the metric given by
\[
d_{\mathscr{U}^d}(U,V) \ldef \sum_{i=1}^d d(U_i,\,V_i),\quad U,V\in\mathscr{U}^d.
\]
\end{cor}
\begin{proof} The proof is in fact a straightforward consequence of the fact that the random walk on the $d$-dimensional hypercube is just a product of one-dimensional walks.
This allows to construct a dynamic on $\mathscr{U}^d$ by considering simply $d$ independent processes ${(U_{1,t},\ldots,U_{d,t})}_{t\geq 0}$ with ${(U_{i,t})}_{t\geq 0}$ associated to a generator $\mathscr{L}^i$ as in Proposition~\ref{prop:hypergen} with parameter $\alpha_i$.
The generator of ${(U_{1,t},\ldots,U_{d,t})}_{t\geq 0}$ is then just $\mathscr{L}f(U) \ldef \sum_i^d \mathscr{L}^i f(U)$, with  $\mathscr{L}^i$ acting only on the $i$-th coordinate. This allows to conclude together with the estimate~\eqref{eq:bound}.
\end{proof}

The next subsections are devoted to the proofs of Proposition~\ref{prop:hypergen} and Proposition~\ref{prop:lip}.
\newline

\paragraph{Holding times and jump kernel}\label{par:holding_times}

A continuous time Markov chain can equivalently be described via its generator or through a function
$c: U \rightarrow \bbR_+$ and a jump kernel $\mu\ldef{\{\mu_{U}(\cdot)\}}_{U \in \mathscr{U}} \subseteq \cP(\scrU)$. Once these have been chosen the Markov dynamics is obtained by the following simple rules~\cite[Chapter 9, Section 3]{Bre13Markov}:
\begin{itemize}
 \item the chain sits in its current state $U$ for a time which is exponentially distributed with parameter $c(U)$, and then makes a jump.
 \item The next state is chosen according to the probability law $\mu_U$.
\end{itemize}
We call this dynamics a $(c,\mu)$-Markov chain.
In the next lines we shall define a pair $(c,\mu)$ for describing our Markov chain. We define $c$ via
\[
c(U) :=  \frac{\alpha^2}{2}+\binom{|U|}{2}.
\]
To define $\mu_U$, we first introduce the measure $\lambda_U \in \cP(\scrU)$ through
\[
  E_{\lambda_{U}} (f) :=  \int_{\mathscr{A}} f(\Psi_{ u,v }U ) \De u  \De v
\]
and then
\begin{equation}\label{def:jumpkernel}
 \mu_{U} := \frac{1}{c(U)} \Big(\alpha^2 \lambda_U+\sum_{A \in {[U]}^2 } \delta_{\Psi_A U}\Big)
\end{equation}
Let us note that $\mu_{U}$ is supported on the set
\[ N(U) := {\{\Psi_{A}U \}}_{A \in \mathscr{A}} . \]

In the following Subsection we construct concretely a Markov process which is a $(c,\mu)$-Markov chain. We show in particular that a $(c,\mu)$-Markov chain has generator $\mathscr{L}$  (see Proposition~\ref{prop:hypergen}).

An informal description of a Markov chain ${(U_t)}_{t \geq 0}$ admitting $\mathscr{L}$ as generator is as follows: at any time, each pair of points in $U_t$ dies at rate $1$ and a new pair of uniformly distributed points in ${[0,1]}^2$ is added to $U$ at rate $\alpha^2$.
When one of such events occurs, everything starts afresh.
\subsubsection{Construction of the dynamics}\label{subs:condynhyp} To prove Proposition~\ref{prop:hypergen} we will employ a rather constructive approach.
More precisely, we build a $(c,\mu)$-Markov chain inductively by defining all the interarrival times, we will then show that such a process has generator $\mathscr{L}$.
This approach is rather convenient since it allows to construct couplings which we use to perform ``convergence to equilibrium'' estimates. Finally, these estimates will be used to solve Stein equation and show uniqueness of the invariant distribution for $\mathscr{L}$.

We begin by defining noise sources, that is, by introducing the clocks that force points to appear or disappear in $U_t$.
\paragraph{Noise sources}\label{par:noise_sources}
We consider a probability space $({\Omega},\,{\cF},\,{\bbP})$  on which a family $\Xi:={\{ \xi^A \}}_{A \in \mathscr{A} }$ of independent Poisson processes of rate $1$ each is defined, together with a Poisson random measure $\beta $ on $[0,+\infty) \times \mathscr{A}$, which is independent from the family ${\{\xi^A \}}_{A \in \mathscr{A}} $
and whose intensity measure is $ \alpha^2 \lambda \otimes \lambda_{\mathscr{A}}$, where $\lambda$ is the Lebesgue measure on $[0,+\infty)$ and $\lambda_{\scrA}$ is the measure on $\scrA$ given by
\[
\lambda_{\scrA}(\bsA) = \int_{\bsA \cap \mathscr{A}} \De u  \De v
\]
for all Borel sets $\bsA\in \mathscr{B}(\mathscr{A})$ (recall that $\mathscr{A}$ is an open subset of ${(0,1)}^2$).
For each $A\in \mathscr{A}$, the process $\xi^A$ will account for the dying clock of the pair $A$, and the process $\beta$ will indicate the rate at which a pair of new jump times is added.
The canonical filtration is defined as usual as
\[
\cF_t \ldef \sigma\Big( {\{ \xi^A_s \}}_{s\leq t , A \in \mathscr{A}} \cup {\{ \beta_s(B)\}}_{s \leq t, B \in \mathscr{B}(\mathscr{A})}\Big)
\]
where we use the notation ${\{ \beta_s(B) \}}_{s\geq 0}$ for $\beta([0,s] \times B)$. We remark that $\beta_s(B)$ is a Poisson process with intensity $\alpha^2\lambda_{\mathscr{A}}(B)$. Similarly, we define for any $t$ the family ${\{ \xi^{A,t} \}}_{A \in \mathscr{A}}$ and the random measure $\beta^t$ via:
\begin{equation}
 \xi^{A,t}_s := \xi^A_{t+s} - \xi^{A}_t, \quad \beta^t_s(B):= \beta( (t,t+s] \times B )\label{eq:xi_with_indices} \end{equation}
$\Xi^t$ is then ${\{ \xi^{A,t} \}}_{A \in \mathscr{A}}$. Moreover, for any $\mathscr{A}'\subseteq \mathscr{A} $, and any $t>0$ we define
\[ \Xi^{\mathscr{A}',t}={ \{\xi^{A,t}\}}_{A \in \mathscr{A}'} . \]

The following proposition is a version of the Markov property for $\Xi$ and $\beta$~\cite[Chapter~9, Section~1.1]{Bre13Markov}.
\begin{lemma}\label{lem:Markovppty}
 Let $T$ be a stopping time for the filtration ${(\cF_t)}_{t\ge 0}$ and let $\cF_T$ be the associated sigma algebra. Then $(\Xi^T,\beta^T )$ is independent from $\cF_T$ and distributed as $(\Xi,\beta)$.
\end{lemma}

For any finite $F\subseteq \mathscr{A}$ we define
\[
 \tau(F)=\tau( \Xi^{F}, \beta ) := \inf \{ t\geq 0:\, \xi^A_t = 1 \,  \text{for some $A \in F$ or $\beta_t(\mathscr{A})$ =1}\}
\]
and, shortening $\tau( \Xi^F, \beta )$ as $\tau$,
\begin{equation*}
\bsA(F)=\bsA(\Xi^{F},\beta) \ldef \begin{cases} A \quad & \mbox{if $\beta_{\tau}(\{ A\}) =1$ for some $A \in \scrA$ } \\ \text{argmax}_{A \in F} \, \xi^A_{\tau} \quad &
\mbox{otherwise}  \end{cases}
\end{equation*}

In words, $\tau(F)$ is the first time when one between the Poisson processes ${\{\xi^A\}}_{A \in F}$ and $\beta(\mathscr A)$ jumps and $\bsA(F)$ identifies which Poisson process has jumped at first.

The following is obtained as an application of the competition theorem for Poisson processes (see e.g.~\citet[Chapter 8, Theorem 1.3]{Bre13Markov}).

\begin{prop}\label{thm:competition}
 Let  $U \in \scrU$. Then, for any $t \geq 0$, $\cO\subset N(U)$ measurable, we have
 \[
  \bbP\left(\Psi_{\bsA({[U]}^2)} U \in \cO, \tau({[U]}^2) \geq t\right)=\mu_{U}(\cO) \exp(-c(U)t ) .
\]
\end{prop}

\begin{proof}
Observe that it is enough to show the statement for measurable sets $\mathscr{O}$ for which there exist $\cA_1, \cA_2 \in \cB(\mathscr{A})$ such that
\[
\cO = {\{ \Psi_A U \}}_{A \in \cA_1} \cup {\{ \Psi_{A}U \}}_{ A \in \cA_2}, \quad \cA_1 \subseteq {[U]}^2,\,
\cA_2 \subseteq \mathscr{A} \setminus {[U]}^2 .
\]
By linearity, we can restrict the attention to the cases when $|\cA_1|=1,\cA_2=\emptyset$ or $\cA_1=\emptyset$.

Let us start by analyzing the first case. Pick $A \in {[U]}^2$ and consequently let $\cO := \Psi_{A}U $. We have, by definition of $\bsA$,
\[
  \left \{\Psi_{\bsA({[U]}^2)}U \in \cO\right \} = \left \{ \bsA({[U]}^2) = A \right \} = \left \{ \xi^A_{\tau({[U]}^2)} =1  \right \}.
\]
First, recall that ${\{\xi^A_t\}}_{A \in {[U]}^2}$, $\beta_t(\scrA)$ are independent Poisson process with rates $1$ and $\alpha^2/2$ respectively. Therefore,
\begin{align}
  \bbP\left(\xi^A_{\tau({[U]}^2)} =1,\,\tau({[U]}^2) \geq t \right) & =
  \frac{1}{ |{[U]}^2| +\alpha^2/2   } \exp( -t(|{[U]}^2| +\alpha^2/2 ) )\nonumber   \\
  & = \frac{1}{c(U)} \exp(-c(U)t).\label{eq:page6}
\end{align}
On the other hand, since $A\in {[U]}^2$, it is easy to verify that $\mu_U(\Psi_A U )= {1}/{c(U)}$ from~\eqref{def:jumpkernel}.

Let us now consider the second case, that is, $\cO = \{ \Psi_{u,v}U \,:\,(u,v) \in \cA_2\}$ and $\cA_2\in \cB(\cA)$ such that $\cA_2 \cap {[U]}^2 = \emptyset$. We have
\[
  \left \{ \Psi_{\bsA({[U]}^2)}U \in \cO \right \} = \left \{ \bsA({[U]}^2) \in \cA_2 \right \} = \left \{ \beta_{\tau({[U]}^2)}(\cA_2)=1 \right \}.
\]
As before, the processes ${\{\xi^A\}}_{A\in[U^2]}$, $\beta(\cA_2 )$ and $ \beta( \mathscr{A} \setminus \cA_2 )$
are independent Poisson processes with rates $1$, $\alpha^2\lambda_{\mathscr{A}}(\cA_2)$ and $\alpha^2\lambda_{\mathscr{A}}(\mathscr{A} \setminus \cA_2 )$ respectively. Therefore
\begin{align*} \bbP \left( \beta_{\tau({[U]}^2)}(\cA_2) = 1,\, \tau({[U]}^2) \geq t \right) & = \frac{   \alpha^2\lambda_{\mathscr{A}}(\cA_2) }{c(U)} \exp( -c(U)t ).
\end{align*}
Now let us compute $\mu_U(\cO)$. Since $\cA_2 \cap {[U]}^2=\emptyset$ we have
\begin{equation*}
  \mu_{U}(\cO) = \frac{\alpha^2}{c(U)} \int_{\mathscr{A}} \mathbbm{1}_{\cO}(\Psi_{u,v}U )\De u  \De v=
  \frac{\alpha^2}{c(U)}\int_{\mathscr{A}\cap \cA_2} \De u  \De v  = \frac{\alpha^2\lambda_{\mathscr{A}}(\cA_2)}{c(U)}
 \end{equation*}
which is the desired conclusion.
\end{proof}

In the previous Lemma, we set up how the first step of the $(c,\mu)$-chain works. We proceed by defining the successive steps by induction. For a given $U\in \mathscr{U}$, we first set $T^U_0 :=0, Z_{0}:=U$. We then set recursively the jump times
\begin{equation}
 \label{def:Uchain} T^{U}_{n+1} -T^{U}_n  := \tau\Big(\Xi^{{[Z_n]}^2,\,T^U_n},\beta^{T^U_n}\Big),\quad \bsA^{U}_{n+1}:= \bsA\Big( \Xi^{{[Z_n]}^{2}, \,T^U_n} ,\beta^{T^U_n}\Big)
\end{equation}
and the jump chain
\[
 Z_{n+1} := \Psi_{\bsA^{U}_{n+1}}Z_n,\quad n\geq 0.
\]
In words, $T_{n+1}^U$ is the first instant after $T_n^U$ when one between the clocks $\xi^A$ with $A\in {[U_{T_n}]}^2$ and $\beta(\mathscr A)$ rings, while $\bsA^U_{n+1}$ represents the corresponding pair which is going to be respectively added or removed.
Finally we define the continuous time process ${(U_t)}_{t\geq 0}$ by
\begin{equation}\label{def2:Uchain}
 U_t \ldef Z_{n},\quad\mbox{if}\quad T^U_n\leq t< T^U_{n+1},\, n\geq0.
\end{equation}
For all $U\in\mathscr{U}$, we denote by $\bbP^U$ the law of ${(U_t)}_{t\geq 0}$ on $\bbD(\bbR_+,\mathscr{U})$ and by $\bbE^U$ the corresponding expectation.

\begin{lemma}\label{lem:genident}The process ${(U_t)}_{t\geq 0}$ defined in~\eqref{def2:Uchain} is a $(c,\mu)$-Markov chain with $\mathscr{L}$ as generator and $P^{00}$ as invariant distribution.
\end{lemma}
\begin{proof}
We first show that ${(U_t)}_{t \geq 0}$ is a $(c,\mu)$-Markov chain and then that its generator is $\mathscr{L}$.
In the proof, since there is no ambiguity, we drop the superscript $U$ from $T^U_n$ and $\bsA^U_{n}$, that is, we simply write $T_n$ and $\bsA_n$.
We shall prove that for any $n \in \N$, any bounded and measurable $f:\scrU \rightarrow \R $ and $t \geq 0$, we have almost surely that
\begin{equation}\label{eq:Markovchain1} \bbE\Big[f(U_{T_{n+1}} ) \mathbbm{1}_{\{T_{n+1}-T_n \geq t\}}  \Big| \mathcal{F}_{T_n} \Big] = E_{\mu_{U_{T_n}}}[f] \exp\big(-c(U_{T_n}) t \big),
\end{equation}
where $\bbE$ denotes the expectation with respect to $\bbP$.
From~\eqref{eq:Markovchain1}, by choosing $f\equiv 1$ we obtain that, conditionally on $ \cF_{T_n}$, $T_{n+1}-T_n$ is distributed as an exponential random variable of parameter $c(U_{T_n})$.
By setting $t=0$ we obtain that, conditionally on $\mathcal{F}_{T_n} $, $U_{T_{n+1}}$ is chosen according to $\mu_{U_{T_n}}$. By letting $f$ and $t$ vary, we also get that conditionally on $\mathcal{F}_{T_n}$, $T_{n+1}-T_n $ is independent from $U_{T_{n+1}}$.
Now let us observe that, by construction, $\bbP [U_{T_{n+1}} \in N(U_{T_n})]=1$. Moreover, we also have that $\mu_{U_{T_n}}$ is supported on $ N(U_{T_n})$.
As a consequence we can reduce ourselves to proving~\eqref{eq:Markovchain1} when $f$ supported on $N(U_{T_n})$. In particular, it suffices to check the formula for $f(U) = \mathbbm{1}_{\{U \in \cO \}}$ for some measurable $\cO \subseteq N(U)$.
Using~\eqref{def:Uchain} we can rewrite
\begin{align*} \bbE & \Big[f(U_{T_{n+1}} ) \mathbbm{1}_{\{T_{n+1}-T_n \geq t\}}  \Big| \mathcal{F}_{T_n} \Big] \\
             & = \bbP\Big[ \Psi_{\bsA(\Xi^{{[U_{T_n}]}^2,T_n},\beta^{T_n})}U_{T_n} \in \cO, \tau(\Xi^{{[U_{T_n}]}^2,T_n},\beta^{T_n}) \geq t \vert \mathcal{F}_{T_n} \Big].\end{align*}
Thanks to the Markov property of Lemma~\ref{lem:Markovppty}, $(\Xi^{T_n},\beta^{T_n})$ is distributed as $(\Xi,\beta)$ and independent from $\cF_{T_n}$.
Therefore, we can apply Proposition~\ref{thm:competition} to conclude that
\begin{align*}
  \bbP\Big[ \Psi_{\bsA(\Xi^{{[U_{T_n}]}^2,T_n},\beta^{T_n})}U_{T_n}  \in \cO, \tau(\Xi^{{[U_{T_n}]}^2,T_n},\beta^{T_n}) \geq t \vert \mathcal{F}_{T_n} \Big] \\
  = \mu_{U_{T_n} }( \cO) \exp(-c(U_{T_n})t  )
\end{align*}
which is what we wanted to prove.

We now show that ${(U_t)}_{t \geq 0}$ admits $\mathscr{L}$ as generator. Let $f:\mathscr{U}\to \bbR$ be bounded and measurable.
Using that $T_1\sim \mathrm{Exp}(-c(U))$ and that $T_2-T_1\sim \mathrm{Exp}(-c(U_{T_1}))$ conditionally to $\cF_{T_1}$, it is not hard to show that the chance that there are two or more jumps before time $t$ is $\bbP[T_2\leq t]=O\left(t^2\right)$. Thus,
\[
  \left|\bbE\left[(f(U_{t}) - f(U))\mathbbm{1}_{\{ T_2\leq t\}}\right]\right|\le 2\|f\|_{\infty}O\left(t^2\right).
\]
Therefore,
\begin{align}
\lim_{t\downarrow 0} & \frac{\bbE[f(U_t) - f(U)]}{t} \nonumber \\
   & =\lim_{t\downarrow 0}\frac{\bbE \left[(f(U_{T_1}) - f(U))\mathbbm{1}_{\{T_1\le t<T_2\}}\right]}{t}
  \stackrel{\eqref{eq:Markovchain1}}{=}\lim_{t\downarrow 0}\frac{ (E_{\mu_U}(f)-f(U)) \left(1-\e^{-c(U)t}\right)}{t} \nonumber \\
   & =\alpha^2\int_{\mathscr{A}} \left[f(\Psi_{r,s} U)-f(U)\right] \De r  \De s +\sum_{A \in {[U]}^2} (f(\Psi_A U )-f(U)) = \mathscr{L}f(U)\label{eq:compu_with_generator}
\end{align}
and we conclude.
\end{proof}

\subsubsection{Construction and analysis of the coupling}\label{subsec:coupling_hyper}

In this paragraph we aim at constructing a coupling between two Markov chains associated to $\mathscr{L}$ which start from neighboring points in $\mathscr{U}$.
We begin by fixing a pair $U,\,V\in \mathscr{U}$ such that $ \Psi_{r,s}V=U$ with $r<s$,  $r,\,s\notin V$ (see Figure~\ref{fig:coup1}). Next, we define the Markov chain ${(U_t)}_{t \geq 0 }$ such that $U_0 = U$ as we did in the former Subsection.
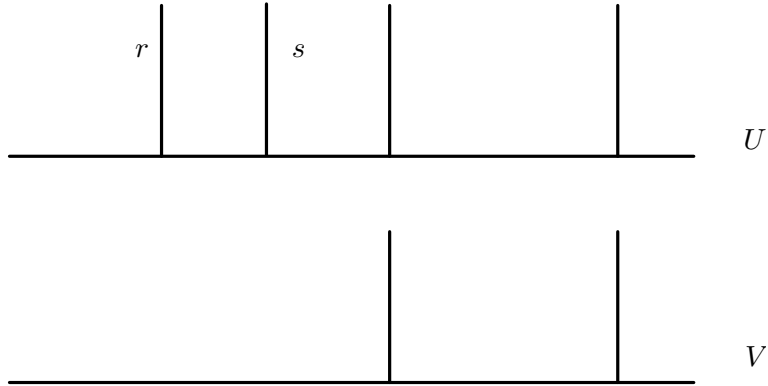
\begin{figure}[ht!]
 \definecolor{ffqqtt}{rgb}{1.,0.,0.2}
 \begin{tikzpicture}[line cap=round,line join=round,>=triangle 45,x=1.0cm,y=1.0cm]
  \clip(-4.42,-1.02) rectangle (5.94,5.32);
  \draw [line width=1.2pt] (-4.,3.)-- (5.,3.);
  \draw [line width=1.2pt] (-4.,0.)-- (5.,0.);
  \draw (5.52,3.5) node[anchor=north west] {$U$};
  \draw (5.54,0.62) node[anchor=north west] {$V$};
  \draw [line width=1.2pt] (-2.,5.)-- (-2.,3.);
  \draw [line width=1.2pt] (-0.62,5.02)-- (-0.62,3.02);
  \draw [line width=1.2pt] (1.,5.)-- (1.,3.);
  \draw [line width=1.2pt] (1.,2.)-- (1.,0.);
  \draw [line width=1.2pt] (4.,5.)-- (4.,3.);
  \draw [line width=1.2pt] (4.,2.)-- (4.,0.);
  \draw (-2.48,4.62) node[anchor=north west] {$r$};
  \draw (-0.42,4.62) node[anchor=north west] {$s$};
 \end{tikzpicture}
 \caption{$U=\Psi_{r,\,s}V$.}\label{fig:coup1}
\end{figure}
To construct the chain ${(V_t)}_{t \geq 0} $ started at $V\in \mathscr{U}$
we use the same noise sources that determined ${(U_t)}_{t\geq0}$.
More precisely, pairs that are added or removed from $V_t$ are exactly those added or removed from $U_t$ up to the time $T_m^U$ that a pair containing either $r$ or $s$ is removed from $U_t$. At time $T_m^U$ we have two possibilities:
\begin{itemize}
  \item[1)] the pair $(r,s)$ is removed from $U_{t}$. As $(r,s)$ does not belong to $V_t$, the two processes now coincide and will continue moving together.
  \item[2)] either $(r,u)$ or $(u,s)$, $u\notin \{ r,s\}$, is removed from $U_{t}$, say for the sake of example $(r,u)$ is removed. Nothing happens to $V_t$ at time $T_m^U$.
  At later times, pairs that are added or removed from $V_t$ are exactly those added or removed from $U_t$ up to the time $T_M^U \geq T_m^U$ when a pair containing $s$, say $(v,s)$ for some $v$, is removed from $U_t$.
  At this time the pair $(u,v)$ is removed from $V_t$. Now the two processes coincide and will move henceforth together.
\end{itemize}

Let us now describe the above construction more rigorously.
Clearly, there is a bijection between $\mathscr{A}$ and the set of unordered pairs $\{ \{ u,v\}:\,u\neq v,\,u,\,v\in(0,1)\}$. With a little abuse of notation, we will at times regard $A\in \mathscr{A}$ as a subset of $(0,1)$ with two elements.
First recall the notation in~\eqref{def:Uchain}. We define the random variable
\[
m := \min \{ k : \{ r,s\} \nsubseteq U_{T^U_k}\}.
\]
In this way,
\begin{equation}\label{def:TUm}
 T^U_{m}: = \inf \{ t \geq 0 : \{ r,s\} \nsubseteq U_t \}
\end{equation}
is the first time a clock associated to a pair present in $U$ but not in $V$ rings. Further, define
\[
  \zeta :=  \{r,s \}  \setminus \bsA^U_{m} , \quad \eta :=  \bsA^U_{m} \setminus \{r,s\}.
\]
$\zeta$ represents the point between $r$ and $s$ (if any) which is not removed at time $T^U_m$, and $\eta$ is the point that was removed together with $\{ r,\,s\}\setminus\zeta$ (see Figure~\ref{fig:syncrocoupl}).
Note that the sets $\zeta$ and $\eta$ have at most one element, when they are non-empty we shall at times regard them as elements of $(0,1)$. We define
\[
M:= \min \{  k \geq m : \zeta \cap U_{T^U_k} = \emptyset \}.
\]
In this way,
\begin{equation}\label{def:TUM}
T^U_M = \inf \{t \geq  T^U_{m } : \zeta \cap U_t = \emptyset \},
\end{equation}
is the first time after $T_m^U$ when a clock involving $\zeta$ rings.
\begin{figure}[!htb]
 \centering\subfloat[Clock of $(r,\,u)$ rings, $u \neq s$. Here $\zeta=s$, $\eta=u$.]{
  \definecolor{ffqqtt}{rgb}{1.,0.,0.2}
  \begin{tikzpicture}[line cap=round,line join=round,>=triangle 45,x=1.0cm,y=1.0cm]
   \clip(-4.5,-0.58) rectangle (6.64,5.92);
   \draw [line width=1.2pt] (-4.,3.)-- (5.,3.);
   \draw [line width=1.2pt] (-4.,0.)-- (5.,0.);
   \draw (5.52,3.5) node[anchor=north west] {$U_{T^U_m}$};
   \draw (5.54,0.62) node[anchor=north west] {$V_{T^U_m}$};
   \draw [line width=1.2pt] (-2.,5.)-- (-2.,3.);
   \draw [line width=1.2pt] (-0.62,5.02)-- (-0.62,3.02);
   \draw [line width=1.2pt] (1.,5.)-- (1.,3.);
   \draw [line width=1.2pt] (1.,2.)-- (1.,0.);
   \draw [line width=1.2pt] (4.,5.)-- (4.,3.);
   \draw [line width=1.2pt] (4.,2.)-- (4.,0.);
   \draw (-2.48,4.62) node[anchor=north west] {r};
   \draw (-0.42,4.62) node[anchor=north west] {$s=\zeta$};
   \draw (1.26,4.62) node[anchor=north west] {$u=\eta$ };
   \draw (4.26,4.62) node[anchor=north west] {};
   \draw (-4.2,5.5) node[anchor=north west] {\showclock{3}{00}};
   \draw [shift={(-2.26,4.98)}] plot[domain=2.9812172096138427:4.769469762790954,variable=\t]({1.*1.377679207943562*cos(\t r)+0.*1.377679207943562*sin(\t r)},{0.*1.377679207943562*cos(\t r)+1.*1.377679207943562*sin(\t r)});
   \draw [->] (-2.2796791235706597,3.6024613500538396) -- (-2.1,3.64);
   \draw [->] (-3.62,5.2) to [out=70] (0.8,4);
  \end{tikzpicture}
 }
 \quad
 \subfloat[The clocks of $(\zeta,\,v)$ in $U_t$ and $(\eta,\,v)$ in $V_t$ are synchronised after $(r,\,\eta)$ dies at $U_{T^U_m}$.]{
  \definecolor{qqqqff}{rgb}{0.,0.,1.}
  \definecolor{ffqqtt}{rgb}{1.,0.,0.2}
  \begin{tikzpicture}[line cap=round,line join=round,>=triangle 45,x=1.0cm,y=1.0cm]
   \clip(-4.5,-0.58) rectangle (6.64,5.92);
   \draw [line width=1.2pt,dashed] (-2.,5.)-- (-2.,3.);
   \draw [line width=1.2pt,dashed] (1.,5.)-- (1.,3.);
   \draw [line width=1.2pt] (-4.,3.)-- (5.,3.);
   \draw [line width=1.2pt] (-4.,0.)-- (5.,0.);
   \draw (5.52,3.5) node[anchor=north west] {$U_{T^U_M}$};
   \draw (5.54,0.62) node[anchor=north west] {$V_{T^U_M}$};
   \draw [line width=1.2pt] (-0.62,5.02)-- (-0.62,3.02);
   \draw [line width=1.2pt] (1.,2.)-- (1.,0.);
   \draw [line width=1.2pt] (4.,5.)-- (4.,3.);
   \draw [line width=1.2pt] (4.,2.)-- (4.,0.);
   \draw (-0.42,4.36) node[anchor=north west] {$s=\zeta$};
   \draw (4.26,4.36) node[anchor=north west] {$v$};
   \draw (1.26,1.38) node[anchor=north west] {$u=\eta$};
   \draw [->] (-0.58,2.28) -- (-0.62,2.92);
   \draw [->] (-0.58,2.28) to [out=70] (4.,3.1);
   \draw [->] (0.74,2.03) to [out=90] (1.0,0.66);
   \draw [->] (0.73,2.12) to [out=70] (3.92,1.26);
   \begin{scriptsize}
    \draw (-0.35,2.10) node {\showclock{4}{50}       \quad=};
    \draw (0.50,2.05) node {\showclock{4}{50}};
   \end{scriptsize}
  \end{tikzpicture}
 }
 \caption{An illustration of the coupling dynamics.}\label{fig:syncrocoupl}
\end{figure}
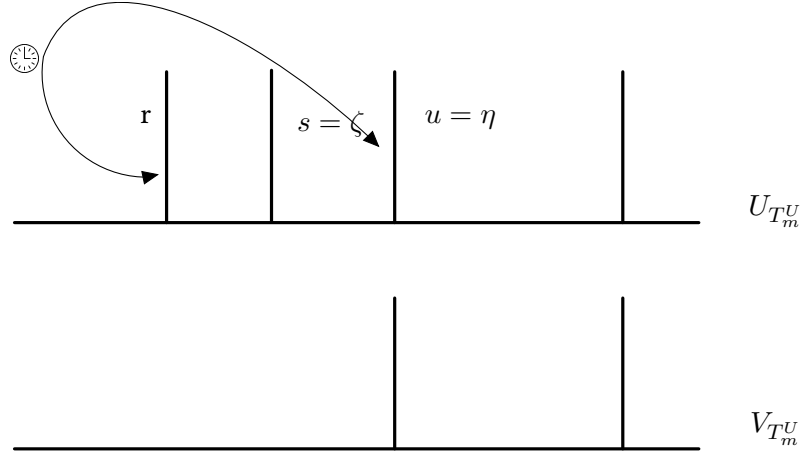
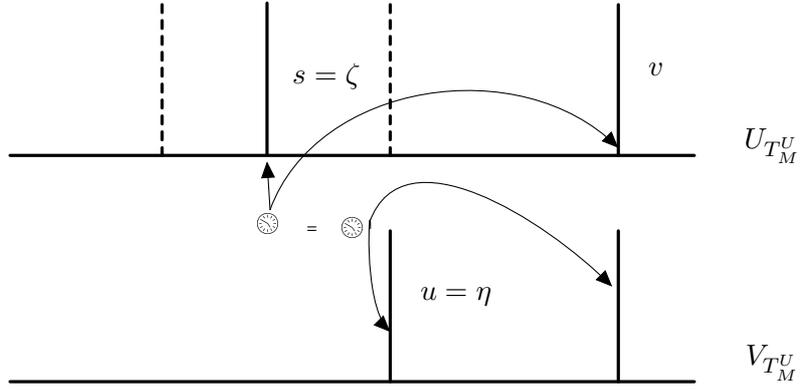
Both $T^U_m$ and $T^U_M$ are $\cF_t$-stopping times, and we have $T^U_m = T^U_M$ if and only if $\zeta = \emptyset$. The event $\{\zeta=\emptyset \}$ just means that the clock associated to the pair of $(r,\,s)$ has rung.
Observe that to implement 2), the process $V_t$ must use the clocks $\xi^{(\zeta,u)}$ with $u\in(0,1)$ in place of the clocks $\xi^{(\eta,u)}$ after $T_m^U$.
We shall define a random bijection $\sigma: \Omega \rightarrow \mathscr{A}^{\mathscr{A}}$ that implements this idea of ``switching the clocks''. This allows to write a formula (see~\eqref{eq:gammaclock} below) for the noises that determine $V_t$ via the noises $(\Xi,\beta)$ that determine $U_t$.
We set $\sigma :=  \mathbf{id}_{\mathscr{A}}$ on the event $ \{ \zeta =\eta= \emptyset \}$. Otherwise, $\zeta$ and $\eta$ are singletons and we set
\begin{equation}\label{def:sigmahcube} \sigma(A) = \begin{cases}  A & \quad \mbox{if $\zeta, \eta  \notin A$ or $\zeta \cup \eta =A   $ } \\ (A \setminus \zeta) \cup \eta & \quad \mbox{if $\zeta \in A, \eta \notin A $}
  \\ (A \setminus \eta) \cup \zeta & \quad \mbox{if $\eta \in A, \zeta \notin A$ }
 \end{cases},
\end{equation}
where we used the aforementioned convention of understanding $A\in\mathscr{A}$ as a subset of $(0,1)$ with two elements, and $\zeta$ and $\eta$ as elements of $(0,1)$.
Note that $\sigma$ is $\mathcal{F}_{T^U_m }$-measurable, where we precise that on $\mathscr{A}^{\mathscr{A}}$ we put the standard cylinder $\sigma$-algebra.
We define the family $\Gamma={\{ \gamma^A \}}_{A \in \mathscr{A}}$ by
\begin{equation}\label{eq:gammaclock}  \gamma^{A}_t  = \xi^A_t \mathbbm{1}_{\{t < T^U_m  \}} +\left( \xi^{A}_{T^U_m}+ \xi^{\sigma(A), T^U_m}_{t-T^U_m}\right)
 \mathbbm{1}_{\{ T^U_{m } \leq  t < T^U_M \}} +\left( \xi^{A}_{T^U_m}+ \xi^{\sigma(A), T^U_m}_{T^U_M-T^U_m} + \xi^{\sigma(A), T^U_M }_{t-T^U_M} \right) \mathbbm{1}_{\{  t \geq  T^U_M \}}. \end{equation}
Finally, we define the process ${(V_t)}_{t \geq 0 }$ in the same way as in~\eqref{def:Uchain} and~\eqref{def2:Uchain} by replacing $U$ with $V$ and $\Xi$ with $\Gamma$. Namely, fixing $T_0^V =0$ and $W_0 = V$, we set recursively the jump times
\begin{equation}\label{def:Vchain} T^V_{n+1} -T^V_n  := \tau(\Gamma^{{[W_n]}^2,T_n^V},\beta^{T_n^V})\quad \bsA^V_{n+1}\ldef \bsA( \Gamma^{{[W_n]}^{2}, T_n^V} ,\beta^{T_n^V})
\end{equation}
and the jump chain
\[
W_{n+1} \ldef \Psi_{\bsA^V_{n+1}}W_n,\quad n\geq 0.
\]
As before, we define the continuous time process ${(V)}_{t\geq 0}$ by
\begin{equation}\label{def2:Vchain}
 V_t \ldef W_{n},\quad\mbox{if}\quad T^V_n\leq t< T^V_{n+1},\, n\geq0.
\end{equation}

We have not shown yet that ${(V_t)}_{t\geq 0}$ is a $(c,\mu)$-Markov chain started in $V$. The next Lemma will be fundamental to show that the pair ${(U_t,\,V_t)}_{t \geq 0}$ is indeed a coupling. It asserts that if we construct another family $\Gamma$
by  exchanging the increments of $\xi^A$ with the increments of $\xi^{\sigma(A)}$ after a certain time $T$, the distribution of $\Gamma$ is the same of $\Xi$, provided that $\sigma(\omega): \scrA \rightarrow \scrA$ is a bijection.
The proof, being rather technical but standard, is postponed to Appendix~\ref{app:clockswitch}.
\begin{lemma}\label{lem:clockswitch}
Let $T$ be a $\cF_t$-stopping time and $\sigma:\Omega\rightarrow \mathscr{A}^{\mathscr{A}}$ a random bijection which is $\cF_T$-measurable.
Define the family ${\{\rho^A \}}_{A \in \mathscr{A}}$ by (recall~\eqref{eq:xi_with_indices})
\[ \rho^A_t:= \xi^A_t \mathbbm{1}_{\{t< T\}} + ( \xi^A_T + \xi^{\sigma(A),T}_{t-T} ) \mathbbm{1}_{\{t \geq T\}}. \]
Then ${\{ \rho^A \}}_{A \in \mathscr{A}}$ is distributed as $\Xi$.
\end{lemma}

\begin{prop}\label{lem:hcubecoupling}
 The pair ${(U_t,\,V_t)}_{t \geq 0}$ is a coupling of $\bbP^U$ and $\bbP^V$.
\end{prop}
\begin{proof} By definition $\bbP^U$ is the law of ${(U_t)}_{t\geq0}$ on $\bbD(\bbR_+,\,\mathscr{U})$.
Therefore, the only thing to show is that $\bbP^V$ is the law of ${(V_t)}_{t\geq0}$ on $\bbD(\bbR_+,\,\mathscr{U})$.
For that it is enough to prove that $\Gamma = \Xi$ in distribution, since ${(V_t)}_{t \geq 0}$ is constructed as ${(U_t)}_{t \geq 0}$ by simply replacing the driving noise $\Xi$ with $\Gamma$ and $U$ with $V$.

An application of Lemma~\ref{lem:clockswitch} for $T = T^U_{m}$  and $\sigma$ as in~\eqref{def:sigmahcube} tells that $ \Theta\ldef {\{ \theta^A \}}_{A \in \mathscr{A}} = \Xi$
in distribution, where
\[
  \theta^A_t \ldef \xi^A_t \mathbbm{1}_{\{t<T^U_{ m}\}} + \left(\xi^{A}_{T^U_m} + \xi^{\sigma(A),T^U_m}_{t-T_m^U} \right) \mathbbm{1}_{\{t \geq T^U_m \}}.
\]
Applying again Lemma~\ref{lem:clockswitch} for $T := T^U_M$ and $\sigma$ as in~\eqref{def:sigmahcube} we obtain that the process $ \overline{\Theta} \ldef {\{ \overline{\theta}^A \}}_{A \in \mathscr{A}} = \Theta$ in distribution, where
 \[
  \overline{\theta}^A_t \ldef \theta^A_t \mathbbm{1}_{\{t<T^U_{ M}\}} + \left(\theta^{A}_{T^U_M} + \theta^{\sigma(A),T^U_M}_{t-T_M^U}\right) \mathbbm{1}_{\{t \geq T^U_M \}}
 \]
(this is~\eqref{eq:gammaclock}). Using the fact that $\sigma (\sigma(A)) =A$ for all $A \in \mathscr{A}$, it is possible to see that $\overline{\Theta} = \Gamma$, from which the conclusion follows.
\end{proof}

 Let us collect below some properties of the coupling $(U_t,\,V_t)$ defined above which follow readily from the construction. From now on, since there is no ambiguity, we write $T_m$ and $T_M$ instead of $T^U_m$,  $T^U_M$.
 \begin{enumerate}[leftmargin=*,label= ({\roman*})]
  \item\label{item:point1}
        For $t<T_{m}$ we have $U_t = \Psi_{r,s}V_t = V_t \cup \{r,s \}$ and $\{r,s\} \cap V_t = \emptyset$.
  \item\label{item:point2}
        For any $T_m \leq t < T_M $ we have
        $U_t = (V_t \setminus \eta) \cup \zeta$.
  \item\label{item:point3}
        For any $t \geq T_M$ we have $U_t = V_t$.
  \item\label{item:point4}
        In particular, for all $t\geq 0$
        \begin{equation}\label{eq:bound_pre_mean_dist}
          d(U_t,V_t) = \mathbbm{1}_{\{t <T_m\}} + 2 \mathbbm{1}_{\{T_m \leq t < T_M\}}.
        \end{equation}
 \end{enumerate}

We finally come to the proof of Proposition~\ref{prop:lip}, of formula~\eqref{eq:sol} as well as to the proof that $P^{00}$ is the unique invariant distribution of the Markov chain ${(U_t)}_{t\geq 0}$. We start by proving Proposition~\ref{prop:lip}.

\begin{proof}[Proof of Proposition~\ref{prop:lip}]
The first step is to prove that for all $t\ge 0$
 \begin{equation}\label{eq:bound_mean_dist}
  \bbE[d(U_t,V_t)] \leq 4\exp(-t/2)+\exp(-t),
 \end{equation}
 for which we shall use~\eqref{eq:bound_pre_mean_dist}. We bound $\bbP[t <T_m]$ by $\bbP[\xi^{(r,\,s)}_t=0]\leq \exp(-t)$,
 since $\xi^{(r,\,s)}$ is a Poisson process with rate $1$.
 The second summand of~\eqref{eq:bound_pre_mean_dist} will give a contribution of $4 \exp(-t/2)$, using that
 \[
  \{T_m\le t<T_M\}\subseteq \left \{T_m\ge\frac{t}{2} \right \}\cup\left \{T_M-T_m\ge\frac{t}{2} \right \}.
 \]
 Note that the second event implies that the coupling has not been successful within time $t/2$ from $T_m$, meaning no one of the clocks $\{\xi^A:\, A\in {[U_{t'}]}^2,\,\zeta\in A,\,t'\in[T_m,\,T_m+t/2)\}$ has rung yet.
The proof of~\eqref{eq:bound_mean_dist} is now complete.

 Let us now show~\eqref{eq:bound}. Let $W_1,\,W_2\in \mathscr{U}$ and assume first $d(W_1,W_2) = 1$. If we denote by ${(W_{i,t})}_{t\geq 0}$ the process with law $\bbP^{W_i}$, $i=1,2$, then by~\eqref{eq:bound_mean_dist} and the Lipschitz continuity
 \begin{align*}
 \left|S_t f(W_1)-S_t f(W_2)\right| &\leq \left|\bbE[f(W_{1,t})- f(W_{2,t})]\right| \\
 &\leq \bbE[d(W_{1,t},W_{2,t})] \leq 4\exp(-t/2)+\exp(-t).
 \end{align*}
 In the case when $d(W_1,W_2)>1$, it suffices to consider a path of length $d(W_1,W_2)$ from $W_1$ to $W_2$ and use the triangular inequality.
\end{proof}

We now show that the Stein equation $\mathscr{L}g=f$ admits a solution for all $f\in\mathrm{Lip}_1(\mathscr{U})$ with $E_{P^{00}}[\mathscr{U}] = 0$ given by the formula~\eqref{eq:sol}. This follows from convergence-to-equilibrium estimates included in the next Proposition.
For any probability measure $\nu\in \cP(\mathscr{U})$, recall that we denote by $\nu_\# S_t$ the measure determined by $\nu_\# S_t(A) \ldef E_\nu[S_t \mathbbm{1}_A]$.
\begin{prop}\label{prop:SteinSol} Let $\mu,\,\nu\in \mathscr{P}(\mathscr{U})$ be probability measures. Then
\begin{equation}\label{eq:conveqhyp}
  d_{W,1}(\nu_\# S_t,\mu_\# S_t) \leq (4\exp(-t/2)+\exp(-t)) d_{W,1}(\mu,\nu).
\end{equation}
In particular, $P^{00}$ is the only invariant distribution of $S_t$. Furthermore, for any $f\in \mathrm{Lip}_1(\mathscr{U})$ such that $E_{P^{00}}[f]=0$ the function
\[
g(U)\ldef -\int_0^\infty S_t f(U)\De t,\quad U\in \mathscr{U},
\]
solves $\mathscr{L}g(U)=f(U)$ for all $U\in \mathscr{U}$. Moreover $g$ is a $9$-Lipschitz function.
\end{prop}
\begin{proof} By definition of $1$-Wasserstein distance
\begin{align*}
d_{W,1}(\nu_\# & S_t,\mu_\# S_t) = \sup_{f\in\mathrm{Lip}_1(\mathscr{U})}\left|\int S_t f \,\De\nu-\int S_t f \,\De\mu\right| \\
 &\leq (4\exp(-t/2)+\exp(-t)) \sup_{g\in\mathrm{Lip}_1(\mathscr{U})}\left|\int g \,\De\nu-\int g \,\De\mu\right| \\
 & = (4\exp(-t/2)+\exp(-t)) d_{W,1}(\nu,\mu),
\end{align*}
where we used that $S_t f$ is Lipschitz with constant $ 4\exp(-t/2)+\exp(-t)$. The uniqueness of the invariant distribution is obvious from~\eqref{eq:conveqhyp}.
The fact that $g$ solves $\mathscr{L}g=f$  for $f\in\mathrm{Lip}_1(\mathscr{U})$, $E_{P^{00}}[f]=0$ is a simple consequence of four steps: passing to the limit in the equality
\[
 f(U) - S_u f(U) = -\int_0^u \mathscr{L}S_t f(U)\,\De t,
\]
using Fubini's Theorem, the particular form of $\mathscr{L}$ and
\[
|S_u f(U)|=|S_u f(U)- E_{P^{00}}[f]| \leq (4\exp(-t/2)+\exp(-t)) d_{W,1}(\delta_U,P^{00}),
\]
which holds thanks to~\eqref{eq:conveqhyp}. The Lipschitz constant of $g$ is obtained also from Proposition~\ref{prop:lip}.
\end{proof}


\subsection{The continuous time random walk on \texorpdfstring{$\Z$}{}}\label{subsec:CTRW}
\subsubsection{Setting and notation}\label{subsubsec:notCTRW}
In this Section we discuss how the ideas for the random walk on the hypercube can be transported to the case of a continuous time random walk on $\bbZ$ (and more generally on $\bbZ^d$, see Corollary~\ref{cor:boundWasCTRW}).
We state the results without detailed proofs, as everything can be done by repeating almost word by word the arguments of the previous section.
We assume that the walker starts in $0$, jumps up by one at rate $j_+$ and down by one at rate $j_-$. We denote by $P^0$ the law on $\bbD([0,1];\bbZ)$ of such walk up to time $T:=1$. The bridge of the random walk from and to the origin is given by
\[
P^{00}(\cdot)\ldef P^0(\cdot|X_0=0,\,X_1=0)
\]
and supported on the space of piecewise constant c\`adl\`ag paths with initial and terminal point at the origin and jumps of sizes $\pm 1$, which we denote by $\Pi([0,1]; \bbZ)$. Let
\begin{equation*}
  \mathscr{V}\ldef \Big \{U= (U^+, U^-) :\, |U^+|,\,|U^-|<\infty \text{ and }\, U^+\times U^- \subset {(0,1)}^2\setminus\Delta\Big \}
\end{equation*}
where $\Delta \ldef \{ (u,u),\;u\in(0,1)\}$.
We consider the map $\bbU: \bbD([0,1]; \bbZ)\to \mathscr{V}$ that to each path $X\in\bbD([0,1]; \bbZ)$ associates $(U^+,\, U^-)$ where $U^+ \subset (0,1)$ is the set of times of positive jumps of $X$ and
$U^-\subset (0,1)$ is the set of times of negative jumps of $X$.

As for the case of the hypercube, it will be convenient to characterize $P^{00}$ as a measure on the set of jump times. We observe that $\Pi([0,1]; \bbZ)$ is in bijection with
\begin{equation*}
  \mathscr{U}\ldef \Big \{U= (U^+, U^-)\in \mathscr{V} :\, |U^+|=|U^-|\Big \}.
\end{equation*}
 The bijection is given by the restriction of $\bbU$ to $\Pi([0,1]; \bbZ)$, we denote by $\bbX : \mathscr{U}\to \Pi([0,1]; \bbZ)$ the inverse. We endow $\mathscr{U}$ with the $\sigma$-algebra $\cU$ of sets $A$ such that $\bbU^{-1}(A)$ belongs to the Borel $\sigma$-algebra of $\bbD([0,1]; \bbZ)$.

The perturbations that we choose to characterize $P^{00}$ are those preserving the ``parity'' of the path, meaning that they add or remove simultaneously a positive and negative jump. More precisely we redefine $\mathscr A:={(0,1)}^2\setminus\Delta $ (and from now on, this notation will be assumed throughout the rest of the Section) and for $(r,s)\in \mathscr A$
we define $\Psi_{r,s} : \mathscr{U}\to \mathscr{U}$ by
\[
 \Psi_{r,s}U = \Psi_{r,s}(U^+,U^-)\ldef
 \begin{cases}
  (U^+\cup \{r\},\,U^-\cup \{s\})      & \mbox{ if } r\notin U^+,\,s\notin U^- ,\\
  (U^+\setminus \{r\},\,U^-\setminus \{s\}) & \mbox{ if }   r\in U^+,\,s\in U^-  ,      \\
  U                 & \mbox{ otherwise. }
 \end{cases}
\]
We endow $\mathscr{U}$ with the graph structure induced by the maps $\{\Psi_{r,s}:\,(r,s)\in \mathscr A\}$. That is, we say that $U,\,V\in \mathscr{U}$ are neighbors if there is $(r,s)\in{(0,1)}^2\setminus\Delta $ such that $U=\Psi_{r,s}V$,
see Figure~\ref{fig:coup1RWZ} for an example of two neirest-neighbor paths. We put on $\mathscr{U}$ the graph distance $d:\mathscr{U}\times \mathscr{U}\to\bbN$. Observe that $\mathscr{U}$ is connected, since any point $U\in \mathscr{U}$ has distance $|U^+|$ to $\mathbf{0}\ldef (\emptyset,\emptyset)$.
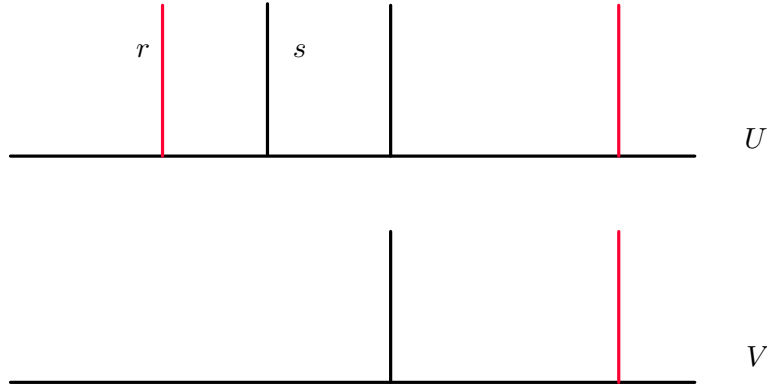
\begin{figure}[ht!]
 \definecolor{ffqqtt}{rgb}{1.,0.,0.2}
 \begin{tikzpicture}[line cap=round,line join=round,>=triangle 45,x=1.0cm,y=1.0cm]
  \clip(-4.42,-1.02) rectangle (5.94,5.32);
  \draw [line width=1.2pt] (-4.,3.)-- (5.,3.);
  \draw [line width=1.2pt] (-4.,0.)-- (5.,0.);
  \draw (5.52,3.5) node[anchor=north west] {$U$};
  \draw (5.54,0.62) node[anchor=north west] {$V$};
  \draw [line width=1.2pt,color=ffqqtt] (-2.,5.)-- (-2.,3.);
  \draw [line width=1.2pt] (-0.62,5.02)-- (-0.62,3.02);
  \draw [line width=1.2pt] (1.,5.)-- (1.,3.);
  \draw [line width=1.2pt] (1.,2.)-- (1.,0.);
  \draw [line width=1.2pt,color=ffqqtt] (4.,5.)-- (4.,3.);
  \draw [line width=1.2pt,color=ffqqtt] (4.,2.)-- (4.,0.);
  \draw (-2.48,4.62) node[anchor=north west] {$r$};
  \draw (-0.42,4.62) node[anchor=north west] {$s$};
 \end{tikzpicture}
 \caption{ $U=\Psi_{r,\,s}V$. The red segment is a $+1$ jump and the black a $-1$ jump.}\label{fig:coup1RWZ}
\end{figure}
\begin{prop}\label{prop:genRWonZ}
  $P^{00}$ is the only invariant measure of a Markov process ${\{ U_t \}}_{t\geq 0}$ on $\mathscr{U}$ with generator
  \begin{equation}\label{eq:genRWonZ}
    \mathscr{L} f(U)\ldef j_+ j_- \int_\mathscr{A} (f(\Psi_{r,s} U)- f(U)) \De r\De s + \sum_{(r,s)\in U^+\times U^-}  (f(\Psi_{r,s} U)- f(U))
  \end{equation}
  for all $f : \mathscr{U}\to \bbR$ bounded and measurable.
\end{prop}
\begin{proof} As in the hypercube example (proof of Prop.~\ref{prop:hypergen}) we want to show $E_{P^{00}}[\mathscr L f]=0$ for all functions $f$ bounded and measurable.
Again we rely on~\citet[Example 28]{CR}, who give the following integration-by-parts characterization of the bridge measure $P^{00}$ on $\bbD([0,1];\bbZ)$: for all bounded and measurable $F:\bbD([0,1];\bbZ)\times(0,1)\times(0,1)\to\R$
\begin{align}
{j_+j_-}E_{P^{00}}&\left[\sum_{(t_1,\,t_2)\in \bbU{(X)}^+\times\bbU{(X)}^- }F(X,\,t_1,\,t_2)\right]\nonumber \\
&=\int_{{[0,\,1]}^2}E_{P^{00}}\left[F\left(X+\mathbbm{1}_{[t_1,\,1]}-\mathbbm{1}_{[t_2,\,1]},\,t_1,\,t_2\right)\right]\De t_1\De t_2 \label{eq:RWIBP}
\end{align}
where $(t_1,\,t_2) \in \bbU{(X)}^+\times\bbU{(X)}^-$ means that $X_{t_1^-}+1=X_{t_1},\,X_{t_2^-}-1=X_{t_2}$ (the reader can compare the notation with the proof of Prop.~\ref{prop:hypergen}).
Again we can consider functionals  $F(X,\,t_1,\,t_2)$ of the form $G(\bbU(X),\,t_1,\,t_2)$ with $G:\mathscr{U}\times(0,1)\times(0,1)\to \bbR$, and note that
$\bbU(X+\mathbbm{1}_{[r,\,1]}-\mathbbm{1}_{[s,\,1]})=\Psi_{r,\,s}U$ almost everywhere in $r$ and $s$.
Taking $G$ to be a difference, we can conclude in the same way as in Prop.~\ref{prop:hypergen} that $P^{00}$ is indeed invariant for $\mathscr{L}$ (the proof of uniqueness will follow as a consequence to Prop.~\ref{prop:lipCRWZ} as we shall see below).
\end{proof}

Below we will rapidly discuss how to construct for any $U\in\mathscr{U}$ a continuous time Markov chain ${\{ U_t \}}_{t\geq 0}$ on $\mathscr{U}$ with generator $\mathscr{L}$ and started from $U$.
We will denote by $\bbP^U$ the law of such process on $\bbD(\bbR_+;\mathscr{U})$, by $\bbE^U$ the corresponding expectation and by
\[
S_t f(U)\ldef \bbE^U[f(U_t)],\quad f:\mathscr{U}\to \bbR
\]
its semigroup. The construction of $U_t$ via Poisson processes will be quite convenient in showing the following key proposition.
\begin{prop}\label{prop:lipCRWZ}
For any $f\in \mathrm{Lip}_1(\mathscr{U})$ with $E_{P^{00}}[f]=0$, any $U,V\in\mathscr{U}$ and all $t\geq 0$
\begin{equation}\label{eq:lipCTRWZ}
  |S_t f(U) - S_t f(V)| \leq (4 \exp(-t/2)+\exp(-t)) d(U,V).
\end{equation}
\end{prop}
Proposition~\ref{prop:lipCRWZ} can be proved via a coupling argument. As in the preceding Section, with a few changes, we construct two processes ${(U_t)}_{t\geq 0}$ and ${(V_t)}_{t\geq 0}$
with generator $\mathscr{L}$ and starting from neighbouring points $U, V\in\mathscr{U}$
in such a way that they are at most at distance two and coalesce in an exponential time.
We will provide few details in the next paragraph.

The consequences of Proposition~\ref{prop:lipCRWZ} are the same as those of the preceding section.
In fact, using~\eqref{eq:lipCTRWZ} and the same argument as in Proposition~\ref{prop:SteinSol}, we can prove that for any $f\in \mathrm{Lip}_1(\mathscr{U})$ such that $E_{P^{00}}[f]=0$
\begin{equation}\label{eq:sol_CTRWZ}
g(U)\ldef - \int_0^\infty S_t f(U)\De t,
\end{equation}
is well-defined and solves $\mathscr{L}g =f$. This allows to obtain the following bound in the Wasserstein distance on $(\mathscr{U},d)$ for bridges of random walks on $\bbZ$ with spatially homogeneous jump rates.
\begin{prop}\label{prop:boundWasCTRW}
 Let $P^{00}$, $Q^{00}$ be the laws of two continuous-time random walk bridges on $[0,\,1]$ with jump rates $j_+,\,j_-$ and $h_+,\,h_-$ respectively. Then
 \[
  d_{W,\,1}(P^{00},\,Q^{00})\le 9\left|j_+j_- -h_+h_-\right|.
 \]
\end{prop}
\begin{proof}
Given Proposition~\ref{prop:lipCRWZ}, the proof is analogous to the hypercube case. The difference in a factor two in the constant comes from the fact that we are integrating over ${(0,1)}^2\setminus\Delta$ rather than $\{(u,v)\in{(0,1)}^2:\,u<v\}$.
This is better explained by saying that in the hypercube case jumping up or down is the same thing.
\end{proof}

 The same argument as in Corollary~\ref{cor:higher_dim} leads to a bound for the distance between bridges of random walks on $\bbZ^d$. We will omit the proof.

\begin{cor}\label{cor:boundWasCTRW}
 Let $d\ge 2$ and let $P^{0,\,d}$ and $Q^{0,\,d}$ be the laws of two bridges of random walks on $\Z^d$ with jump rates $j^{(i)}_+,\,j^{(i)}_-$ resp. $h^{(i)}_+,\,h^{(i)}_-$ in the $i$-th coordinate. Then,
 \[
  d_{W,\,1}(P^{0,\,d},\,Q^{0,\,d})\le 9\sum_{i=1}^d\left|j^{(i)}_+ j^{(i)}_- - h^{(i)}_+ h^{(i)}_-\right|.
 \]
\end{cor}

\begin{remark}
  Once again, we wish to stress that the bound in Proposition~\ref{prop:boundWasCTRW} (resp.\ Corollary~\ref{cor:boundWasCTRW}) is compatible with what is known about conditional equivalence for bridges of random walks on $\bbZ$ (resp.\ $\bbZ^d$) with spatially homogeneous jump rates. Indeed, two random walk share their bridges if and only if $j^+j^-=h^+h^-$.
\end{remark}

\paragraph{Coupling construction}
The construction of a Markov process with generator $\mathscr{L}$ can be performed similarly to the previous section defining on a common probability space $(\Omega,\cF,\bbP)$ a family of independent identically distributed Poisson processes $\Xi\ldef{\{ \xi^A\}}_{A\in\mathscr{A}}$
with rate one and a Poisson random measure $\beta$ on $\bbR_+\times \mathscr{A}$ with intensity  $j_+ j_-\lambda\otimes\lambda_\mathscr{A}$, where $\lambda$ is the Lebesgue measure on $\bbR_+$ and  $\lambda_\mathscr{A}$ is the Lebesgue measure on $\mathscr{A}$.
Using the noises $(\Xi,\beta)$, it is now straightforward to construct inductively a continuous-time Markov chain ${(U_t)}_{t\geq 0}$ started in $U\in\mathscr{U}$ by sampling the interarrival times as in~\eqref{def:Uchain} and~\eqref{def2:Uchain}.
In words, the dynamics follows a birth-and-death mechanism.  Birth occurs after an exponentially distributed time of rate $j_+j_-$, when a pair positive-negative jump $(r,s)$ is sampled uniformly from $\mathscr{A}$, $r$ is added to $U^+_t$, and $s$ is added to $U^-_t$.
Death occurs at rate $U^+_t U^-_t $ when a pair $(r,s)$ is sampled uniformly from $U^+_t \times U^-_t$, $r$ is removed form $U^+_t$ and $s$ from $U^-_s$.
It follows from the same argument of Lemma~\ref{lem:genident} that ${(U_t)}_{t\geq 0}$ has generator $\mathscr{L}$. We denote by $\bbP^U$ its law on $\bbD(\bbR_+;\mathscr{U})$, by $\bbE^U$ the corresponding expectation and by ${(S_t)}_{t\geq 0}$ its semigroup.

We will describe in words the coupling construction which is based on the one for the hypercube of Subsubsection~\ref{subsec:coupling_hyper}. To simplify the exposition, by ``adding (removing) $(u,v)$  to (from) $U_{T^U_k}$'' we mean that $u$ is added to (removed from) $U^+_{T^U_k}$ and $v$ is added to (removed from) $U^-_{T^U_k}$.
Following closely the notation used for the hypercube, we consider \ the times $\{T^U_k:\,k\in \bbN \}$, representing the jump times of the chain ${(U_t)}_{t\geq 0}$ and the sequence $\{ \mathbf{A}_k^U:\,k\in \bbN \}$, representing the pair in $\mathscr{A}$ which is either added or removed from the chain at time $T^U_k$.

Let us begin by fixing $U,\,V\in \mathscr{U}$ such that $U=\Psi_{r,s}V$ with $r\notin V^+$ and $s\notin V^-$ as in Figure~\ref{fig:coup1RWZ}. We want to construct a coupling ${(U_t,V_t)}_{t\geq 0}$ of $\bbP^U$ and $\bbP^V$. Our coupling works algorithmically as follows: we start at time $k=0$.
For all $k$ such that  both $r\in U^+_{T^U_k}$ and $s\in U^-_{T^U_k}$
\begin{enumerate}[label=\arabic*)]
\item add(remove) simultaneously the same points to(from) $V^+_{T^U_k},V^-_{T^U_k}$ that are added to(removed from) $U^+_{T^U_k},U^-_{T^U_k}$. In other words we use the clocks $\{\xi^{(u,v)},\, (u,v)\in U^+_{T^U_k} \times U^-_{T^U_k}\}$ to remove $r$ from $V^+_{T^U_k}$ and $s$ from $V^{-}_{T^U_k}$ and the process $\beta$ in order to add new pairs.
\end{enumerate}
Let $m\ldef \inf \{k:\,\text{either }r\notin U^+_{T^U_k}\text{ or }s\notin U^-_{T^U_k}\}$ so that, as before,
\[
T_m^U \ldef \inf \{t>0:\; \mbox{either $r\notin U_t^+$ or $s\notin U_t^-$}\},
\]
and the pair $\mathbf{A}_m^U$ is removed at time $T_m^U$ is of the form $(r,s)$ or $(r,w)$ or $(w,s)$ for some $w \neq r,\,s$.
For the sake of example, say that $\mathbf{A}_m^U=(r,\,w),\,w\neq s$, is the pair that is removed at time $T^U_m$. Then
\begin{enumerate}
\item[2)] $\mathbf{A}_m^U$ is removed from $U_{T_m^U}$ and nothing happens in $V_{T_m^U}$.
Set $\zeta$ to be the point between $r$ and $s$ which is not removed (in our example $\zeta:=s)$ and $\eta$ the point that is neither $r$ or $s$ and that is removed (in our example $\eta:=w$) (see Figure~\ref{fig:syncrocouplRWZ}).
\item[3)] For $t>T_m^U$ repeat 1) with the difference that for the dynamic ${(V_t)}_{t\geq 0}$ we replace each clock $\xi^{(u,\,\eta)}$ with the clock $\xi^{(u,\,\zeta)}$ for any $u\in (0,\,1)$.
\end{enumerate}
The algorithm is built in such a way that at the first instant $T_M^U>T_m^U$ when a Poisson clock involving $\zeta$ rings the two dynamics will coincide and continue together almost surely.
By construction, we have that almost surely for all $t\geq 0$
\begin{equation}\label{eq:distance_coupling}
d(U_t,V_t) = \mathbbm{1}_{\{t<T_m^U\}}+2 \mathbbm{1}_{\{T^U_m\leq t<T_M^U\}},
\end{equation}
which leads immediately to the following.
\begin{proof}[Proof of Proposition~\ref{prop:lipCRWZ}] As a first step one uses~\eqref{eq:distance_coupling} to show that for all $U,\,V\in \mathscr{U}$ such that $d(U,V)=1$ and all $t\geq 0$
\[
\bbE[d(U_t,V_t)] \leq (4 \exp(-t/2)+\exp(-t)).
\]
From here,~\eqref{eq:lipCTRWZ} is derived in the same way as for the hypercube case.
\end{proof}

\begin{figure}[!ht]
 \centering\subfloat[Clock of $(r,\,w)$ rings. Here $\zeta=s$, $\eta=w$.]{
  \definecolor{ffqqtt}{rgb}{1.,0.,0.2}
  \begin{tikzpicture}[line cap=round,line join=round,>=triangle 45,x=1.0cm,y=1.0cm]
   \clip(-4.5,-0.58) rectangle (6.64,5.92);
   \draw [line width=1.2pt] (-4.,3.)-- (5.,3.);
   \draw [line width=1.2pt] (-4.,0.)-- (5.,0.);
   \draw (5.52,3.5) node[anchor=north west] {$U_{T^U_m}$};
   \draw (5.54,0.62) node[anchor=north west] {$V_{T^U_m}$};
   \draw [line width=1.2pt,color=ffqqtt] (-2.,5.)-- (-2.,3.);
   \draw [line width=1.2pt] (-0.62,5.02)-- (-0.62,3.02);
   \draw [line width=1.2pt] (1.,5.)-- (1.,3.);
   \draw [line width=1.2pt] (1.,2.)-- (1.,0.);
   \draw [line width=1.2pt,color=ffqqtt] (4.,5.)-- (4.,3.);
   \draw [line width=1.2pt,color=ffqqtt] (4.,2.)-- (4.,0.);
   \draw (-2.48,4.62) node[anchor=north west] {$r$};
   \draw (-0.48,4.62) node[anchor=north west] {$s=\zeta$};
   \draw (1.26,4.62) node[anchor=north west] {$w=\eta$};
   \draw (4.26,4.62) node[anchor=north west] {$q$};
   \draw (-4.2,5.5) node[anchor=north west] {\showclock{3}{00}};
   \draw [shift={(-2.26,4.98)}] plot[domain=2.9812172096138427:4.769469762790954,variable=\t]({1.*1.377679207943562*cos(\t r)+0.*1.377679207943562*sin(\t r)},{0.*1.377679207943562*cos(\t r)+1.*1.377679207943562*sin(\t r)});
   \draw [->] (-2.2796791235706597,3.6024613500538396) -- (-2.1,3.64);
   \draw [->] (-3.62,5.2) to [out=70] (0.8,4);
  \end{tikzpicture}
 }
 \quad
 \subfloat[The clocks of $(q,\,s)$ in $U$ and $(q,\,w)$ in $V$ are synchronised after $(r,\,w)$ dies in $U$.]{
  \definecolor{qqqqff}{rgb}{0.,0.,1.}
  \definecolor{ffqqtt}{rgb}{1.,0.,0.2}
  \begin{tikzpicture}[line cap=round,line join=round,>=triangle 45,x=1.0cm,y=1.0cm]
   \clip(-4.5,-0.58) rectangle (6.64,5.92);
   \draw [line width=1.2pt,color=ffqqtt,dashed] (-2.,5.)-- (-2.,3.);
   \draw [line width=1.2pt,dashed] (1.,5.)-- (1.,3.);
   \draw [line width=1.2pt] (-4.,3.)-- (5.,3.);
   \draw [line width=1.2pt] (-4.,0.)-- (5.,0.);
   \draw (5.52,3.5) node[anchor=north west] {$U_{T^U_M}$};
   \draw (5.54,0.62) node[anchor=north west] {$V_{T^U_M}$};
   \draw [line width=1.2pt] (-0.62,5.02)-- (-0.62,3.02);
   \draw [line width=1.2pt] (1.,2.)-- (1.,0.);
   \draw [line width=1.2pt,color=ffqqtt] (4.,5.)-- (4.,3.);
   \draw [line width=1.2pt,color=ffqqtt] (4.,2.)-- (4.,0.);
   \draw (-0.45,4.36) node[anchor=north west] {$s$};
   \draw (4.26,4.36) node[anchor=north west] {$q$};
   \draw (1.26,1.38) node[anchor=north west] {$w$};
   \draw [->] (-0.58,2.28) -- (-0.62,2.92);
   \draw [->] (-0.58,2.28) to [out=70] (4.,3.1);
   \draw [->] (0.74,2.03) to [out=90] (1.0,0.66);
   \draw [->] (0.73,2.12) to [out=70] (3.92,1.26);
   \begin{scriptsize}
    \draw (-0.35,2.10) node {\showclock{4}{50}       \quad=};
    \draw (0.50,2.05) node {\showclock{4}{50}};
   \end{scriptsize}
  \end{tikzpicture}
 }
 \caption{An illustration of the coupling dynamics.}\label{fig:syncrocouplRWZ}
\end{figure}
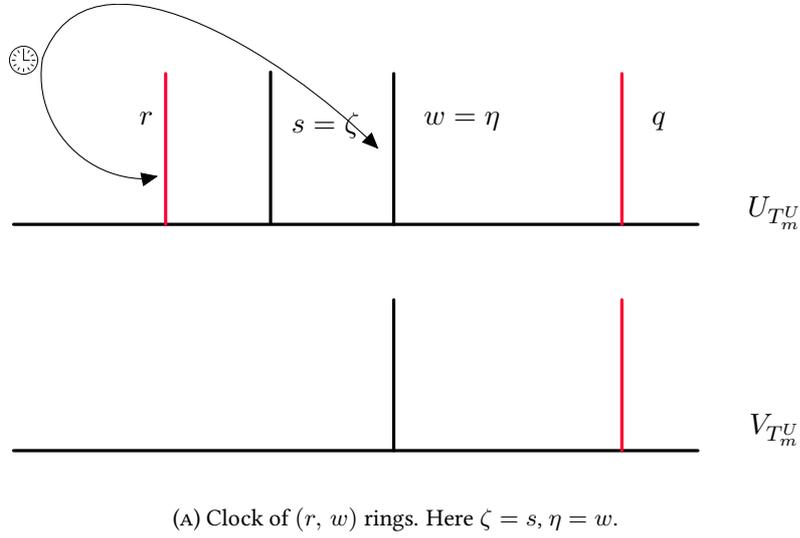
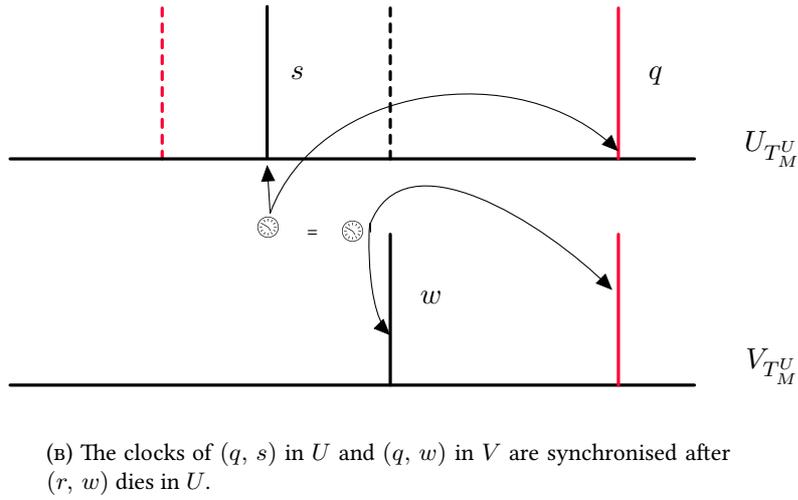

Mimicking the proof of Proposition~\ref{prop:SteinSol} we can now state the following consequence of Proposition~\ref{prop:lipCRWZ}:
\begin{prop} Let $\mu,\nu\in \mathscr{P}(\mathscr{U})$ be probability measures. Then,
\begin{equation}\label{eq:conveq}
  d_{W,1}(\nu_\# S_t,\mu_\# S_t) \leq (4 \exp(-t/2)+\exp(-t)) d_{W,1}(\mu,\nu).
\end{equation}
In particular, $P^{00}$ is the only invariant distribution of $S_t$. Furthermore, for any $f\in \mathrm{Lip}_1(\mathscr{U})$ such that $E_{P^{00}}[f]=0$
\[
g(U)\ldef -\int_0^\infty S_t f(U)\De t,\quad U\in\mathscr{U}
\]
solves $\mathscr{L}g(U)=f(U)$ for all $U\in\mathscr{U}$ and
\begin{equation}\label{eq:lipsol}
  |g(U)-g(V)|\leq 9 d(U,V),\quad \forall U,\,V\in\mathscr{U}.
\end{equation}
\end{prop}

\subsection{Non-homogeneous continuous-time random walks}\label{subsec:revCTRW}

In this Subsection we consider continuous time random walks bridges on the integers with possibly non-homogeneous jump rates. Generalizations to higher dimensions can be obtained in the same fashion as Corollary~\ref{cor:higher_dim}.
Recall that $P^{00}$ of Subsection~\ref{subsec:CTRW} is the law of a random walk bridge with homogeneous jump rates $j_+$, $j_-$. To simplify matters, we fix the rates as $j_-=j_+:=1$.

We will use the same setting of subsection~\ref{subsubsec:notCTRW}.
In particular we will consider the set $\mathscr U$ of jump times that uniquely identifies a bridge, the map $\bbU$ that associates to a bridge $X$ its jump times $\bbU(X)=(\bbU{(X)}^+, \bbU{(X)}^-)$ and its inverse $\bbX$ which allows to reconstruct the path from the jump times.
We will often regard measures on the path space as measures on the set $\mathscr U$ via the pushforward $\bbU$.
Define on $\mathbb D([0,\,1];\,\Z)$ the law $\mathbf P$ of a random walk $X$ on $\Z$ with infinitesimal generator
\begin{equation}\label{eq:genRWnonhom}
 \mathcal G f(j):=a(j)(f(j+1)-f(j))+b(j)(f(j-1)-f(j)),\;j\in \Z
\end{equation}
where $f:\Z\to\R$ has bounded support and $a,\,b:\,\Z\to(0,\,+\infty)$ are the jump rates.

Let us write $\bfP^{00}$ for the bridge of $X$ conditioned to be at $0$ at time $1$:
\[
\bfP^{00}\left(X\in \cdot\right):=\bfP\left(X\in \cdot|\,X_0=0,\,X_1=0\right).
\]
We want to get bound in the Wasserstein distance on $(\mathscr{U},d)$ between $P^{00}$ and $\bfP^{00}$.
For that we shall implement the strategy presented in Remark~\ref{rem:perturbation_operator}. The first step is to identify an operator which admits $\bfP^{00}$ as invariant measure. This is achieved using the observation~\ref{A)}.
Define for $X\in \bbD([0,1];\bbZ)$
\begin{equation}\label{eq:RNrandomwalks}
M(X) \ldef \exp\left(-\int_0^1 \Xi(X_{t-}) \De t\right)\prod_{t\in  \bbU{(X)}^+}a(X_{t-})\prod_{s\in \bbU{(X)}^-}b(X_{s-}),
\end{equation}
where $\Xi(j) := a(j)+b(j)$ is the total jump rate at $j\in \bbZ$.

 Then, we have the following Proposition.
\begin{prop}\label{prop:SteinRevCTRW}
$\bfP^{00}$ is an invariant law for the generator $\mathscr G$ on $\mathscr U$ defined by
 \begin{align}\label{eq:genRN}
  \mathscr G f(U) :=\int_{{[0,\,1]}^2}\left(f(\Psi_{u,\,v}U)-f(U)\right) &\frac{M(\bbX(\Psi_{u,v}U))}{M(\bbX(U))} \De u \De v  \nonumber \\ &+\sum_{u\in U^+,v\in U^-}\left(f(\Psi_{u,\,v}U)-f(U)\right)
 \end{align}
for any $f:\,\mathscr U\to\R$ bounded and measurable.
\end{prop}
\begin{remark}
 Note that we do not need to know that $\bfP^{00}$ is the unique law satisfying the above Proposition. We will only use for our purposes that $\bfP^{00}$ is one such law.
\end{remark}
\begin{remark}
  It is possible to extend the proposition above and the considerations that follow to jump rates $a(t,j),\,b(t,j)$, $j\in \bbZ$, that also depend on time. For that it suffices to identify the suitable change of measure $\De \bfP / \De P$, which in fact is available in~\cite{CL15}.
\end{remark}

\begin{proof}
The main idea of this proof is the following: we begin by working on the path space, where IBP formulas are available, and in the end we will transfer the results to the set $\mathscr U$,
finally proving that $E_{\bfP^{00}}[\mathscr Gf]=0$ for all $f:\mathscr U\times [0,1]\times [0,\,1]\to\R$ bounded and measurable. We begin by noticing that Girsanov's formula (cf.~\citet[Section 3, Eq. (13)]{CL15}) yields
 \begin{align*}
  \frac{\De \bP^{00}}{\De P^{00}}(X)&\propto
  \exp\left(
  \sum_{t:\,X_{t_-}\neq X_t} \log\left(a(X_{t_-})\mathbbm{1}_{\{t\in \bbU{(X)}^+\}}+b(X_{t_-})\mathbbm{1}_{\{t\in \bbU{(X)}^-\}}\right)\right.\\
  &
  \left.-\int_0^1 (a(X_{t-})+b(X_{t-})) \,\De t\right)\\
  &=\exp\left(-\int_0^1 (a(X_{t-})+b(X_{t-})) \De t\right)\prod_{t\in  \bbU{(X)}^+}a(X_{t-})\prod_{s\in \bbU{(X)}^-}b(X_{s-}) = M(X).
  \end{align*}
Consider now~\eqref{eq:RWIBP} for a random walk bridge with unit jump rates. Take as test function
\[
F(X,\,u,\,v)M(X)
\]
where $F$ is any bounded and measurable function.
By multiplying and dividing the left-hand side by the Radon-Nikodym derivative~\eqref{eq:RNrandomwalks} we obtain
 \begin{align*}
E_{\bfP^{00}}   &\left[ \int_{{[0,\,1]}^2}F(X+\mathbbm{1}_{[u,\,1]}-\mathbbm{1}_{[v,\,1]},\,u,\,v)\frac{M(X+\mathbbm{1}_{[u,\,1]}-\mathbbm{1}_{[v,\,1]})}{M(X)}\De u \De v\right] \\
         & =E_{\bfP^{00}}\left[\sum_{(u,\,v)\in \bbU{(X)}^+\times \bbU{(X)}^-}F(X,\,u,\,v)\right].
 \end{align*}
As before, we now pass to the image measure, i.e.\ we choose $F(X,\,u,\,v):=G(\bbU(X),\,u,\,v)$ with $G:\mathscr{U}\times(0,1)\times(0,1)\to \bbR$. We thus obtain
 \begin{align*}
  E_{\bfP^{00}} \left[ \int_{{[0,\,1]}^2}G(\Psi_{u,\,v}U,\,u,\,v)\frac{M(\bbX(\Psi_{u,v}U))}{M(\bbX(U))}\De u \De v\right]
    =E_{\bfP^{00}}\left[\sum_{(u,\,v)\in U^+\times U^-}G(U,\,u,\,v)\right].
\end{align*}
The conclusion follows by choosing
 \[
  G(U,\,u,\,v):=f(U)-f(\Psi_{u,\,v}U).\qedhere
 \]
\end{proof}
Having the generator, we can now employ the Stein-Chen method to obtain a bound in the Wasserstein distance as desired. Recall that $\Delta$ is the diagonal of ${[0,\,1]}^2$.
\begin{cor}\label{prop:finalBoundCTRWrev}
Let $P^{00}$ be as in Subsection~\ref{subsec:CTRW} with unit jump rates.
Let $\bfP^{00}$ be the law of a continuous-time random walk bridge with rates $a(\cdot),\,b(\cdot)$ as above.
If $M$ is as in~\eqref{eq:RNrandomwalks} and $\nabla_{u,\,v}\log M(U):={(M(\bbX(U)))}^{-1}(M(\bbX(\Psi_{u,\,v}U))-M(\bbX(U)))$ then
 \[
  d_{W,\,1}\left(\bfP^{00},\,P^{00}\right)\le 9 E_{\bfP^{00}}\left[\sup_{(u,\,v)\in {(0,\,1)}^2\setminus \Delta}\left|\nabla_{u,\,v}\log M(U)\right|\right].
 \]
\end{cor}
\begin{proof} Let
 \begin{equation}\label{eq:ratio}
     H(U,\,u,\,v):= \frac{M(\bbX(\Psi_{u,v}U))}{M(\bbX(U))}
 \end{equation}
 for any $U\in \mathscr U$, $(u,\,v)\in {(0,\,1)}^2\setminus \Delta$.
 We begin by observing that if $g$ solves
 \[
  \mathscr L g=f,\quad f\in \mathrm{Lip}_1(\mathscr U),\, \,E_{P^{00}}[f]=0,
 \]
(cf.~\eqref{eq:sol_CTRWZ}) we can bound the $1$-Wasserstein distance between $\bfP^{00}$ and $P^{00}$ by computing
 \begin{align}
  \left|E_{\bfP^{00}}\left[\mathscr L g-\mathscr Gg\right]\right| & \leq E_{\bfP^{00}}\left[\int_{{[0,\,1]}^2}\left|g(\Psi_{u,v}U)-g(U)\right|\left|H(U,\,u,\,v)-1\right| \De u \De v\right]\nonumber     \\
  & \stackrel{\text{Eq.~\eqref{eq:lipsol} }}{\le}9
  E_{\bfP^{00}}\left[\int_{{[0,\,1]}^2}\left|H(U,\,u,\,v)-1\right|\De u \De v\right].\label{eq:yetanotherbound}
 \end{align}
 To conclude, observe that
 \[
  H(U,\,u,\,v)-1=\nabla_{u,\,v}\log M(U).\qedhere
 \]
\end{proof}
\begin{remark}[Bridges and reciprocal characteristics]
The bound in Corollary~\ref{prop:finalBoundCTRWrev} is compatible with the results of~\cite{CL15} on conditional equivalence for bridges, in the sense that we show that two random walks with the same bridges on $\bbZ$ satisfy the same estimate.
In fact ~\citet[Theorem 2.4]{CL15} show that two random walks have the same bridges if and only if the quantities $a(i)b(i+1)$ and $\Xi(i+1)-\Xi(i)$ coincide for all $i\in \bbZ$.
The functions $i\mapsto a(i)b(i+1)$ and $i\mapsto \Xi(i+1)-\Xi(i)$ are known in the literature under the name of \emph{reciprocal characteristics}. They naturally identify two important families of random walk bridges:
\begin{itemize}
  \item the bridges of reversible random walks, when $a(i)b(i+1)$ is constant,
  \item the bridges of constant speed random walks, when  $\Xi(i+1)-\Xi(i)$ is zero.
\end{itemize}
\end{remark}
We will now consider these two types of bridges and will give quantitative bounds on the approximation by the bridge of the simple random walk in the $1$-Wasserstein distance.
\subsubsection{The case of continuous-time reversible random walks}
Assume that
\begin{equation}\label{eq:ass_revRW}
a(j) b(j+1) = 1\end{equation}
for all $j\in \bbZ$. In this case, $\bfP$ is the law of a reversible random walk on $\bbZ$. A reversible measure $\pi$ can be found, up to a multiplicative constant, by imposing
\[
\pi(j+1) = {\pi(j)} a{(j)}^2,\quad \forall \, j\in \bbZ.
\]
Moreover, $M$ as defined in~\eqref{eq:RNrandomwalks} takes the form
\[
M(X)\ldef \exp\left(-\int_0^1 \big(a(X_{t-})+b(X_{t-})\big) \De t \right).
\]
This is due to the fact that, since $|\bbU{(X)}^+|=|\bbU{(X)}^-|$ and $X$ is a bridge, we can define a bijection $m:\bbU{(X)}^+\to\bbU{(X)}^-$ such that $X_{m(t)}=X_{t}+1$ for all $t\in \bbU{(X)}^+$ and use $a(j) b(j+1) = 1$ to simplify. In particular,
\[
 \nabla_{u,\,v}\log M(U) = \exp\left(-\int_{\min \{u,\,v\}}^{\max \{u,\,v\}}\nabla^{\mathrm{sgn}(v-u)}\Xi(\bbX{(U)}_{t-})\De t\right) -1
\]
where $\Xi(j) := a(j)+b(j)$ is the total jump rate at $j\in\bbZ$ and we have set $\nabla^\pm h(i):=h(i\pm 1)-h(i)$ for $h:\Z\to\R$.
This can be used as a starting point to get distance bounds. For example we can immediately prove the following universal bound which depends only on the speed of the random walk.
\begin{prop}
Let $\mathbf P^{00}$ be the law of the bridge of a continuous-time random walk satisfying~\eqref{eq:ass_revRW} and for which there exists $\kappa>0$ such that for all $j\in \bbZ$
\[
|\Xi(j+1)-\Xi(j)|\leq \kappa.
\]
Then
\[
  d_{W,\,1}\left(\bfP^{00},\,P^{00}\right)\le 9 \left(2\cdot\frac{ e^\kappa-1-\kappa}{\kappa^2}-1\right).
\]
\end{prop}
\begin{proof}
This is a direct consequence of~\eqref{eq:yetanotherbound} and the bound
\[
 \exp(-\kappa|u-v|)-1\leq \nabla_{u,\,v}\log M(U) \leq \exp(\kappa|u-v|)-1,\quad \forall\, u,\, v\in(0,1). \qedhere
\]
\end{proof}

\subsubsection{The case of continuous-time constant-speed random walk}
We would like now to provide some explicit bounds for a certain class of random walk bridges on $\Z$ whose underlying random walk measure has constant speed. Namely, also in the measure-theoretic setting of subsection~\ref{subsubsec:notCTRW}, we will consider the random walk $\bfP$ whose generator $\mathcal{G}$ is given in equation~\eqref{eq:genRWnonhom} and whose jump rates $a,\,b:\,\Z\to(0,\,+\infty)$ satisfy
\begin{equation}\label{eq:condRWvarying}
  \Xi(j+1)-\Xi(j) = \kappa,\quad
  \nu \le a(j)b(j+1)\le \mu,\quad\forall \,j\in \Z,
\end{equation}
where  $\mu\geq \nu>0$. Notice that the walk does not need to be reversible. Its bridge will, as before, have law
\[
\bfP^{00}(X\in\cdot):=\bfP(X\in\cdot|X_0=0,\,X_1=0).
\]
Our target process will remain the same of the previous pages, that is, the random walk bridge with unit jump rates  whose law is $P^{00}$; any choice of homogeneous jump rates is also possible.
In fact, the only thing that we need is that we can solve the Stein's equation for $\mathscr{L}$ associated to $P^{00}$ and provide estimates for the solution.
Generalization to higher dimensions, e.g.\ random walks on $\bbZ^d$, are also possible.

\begin{theorem} In the above setting
 \begin{align*}
 d_{W,\,1}(\bfP^{00},\,P^{00}) & \leq 9 \left(\mu\cdot \frac{I_0(2\mu/\sqrt{\nu})}{I_0(2\sqrt{\nu})} - \sqrt{\mu \nu} + |1-\sqrt{\mu \nu}| \right)
 \end{align*}
 where $I_0$ is the modified Bessel function of the first kind.
\end{theorem}
\begin{remark} Observe that with the choice $\mu = \nu$ we find back the bound as in Proposition~\ref{prop:boundWasCTRW}.
\end{remark}
\begin{proof}
Let $P^{00}_{\lambda}$ be the bridge of a random walk on $\bbZ$  with jump rates $j_+$, $j_-$ such that $j_+ j_- = \lambda$, $\lambda>0$. Then
\begin{equation}\label{eq:triangular}
  d_{W,1}(\bfP^{00},P^{00}) \leq d_{W,1}(P^{00}_\lambda,P^{00})  + d_{W,1}(\bfP^{00},P_\lambda^{00}).
\end{equation}
Clearly, by Proposition~\ref{prop:boundWasCTRW} we have $d_{W,1}(P^{00}_\lambda,P^{00}) \leq 9|1-\lambda|$. We now proceed to estimate the second contribution, which boils down to getting estimates for the ratio~\eqref{eq:ratio}.
 Since~\eqref{eq:condRWvarying} ensures that $a(X_{t-})+b(X_{t-}) = \kappa$  for all $t$, we can replace this in~\eqref{eq:RNrandomwalks} and get
 \[
  M(X) = \exp(-\kappa) \prod_{t\in \mathbb{U}{(X)}^+}a(X_{t-})\prod_{s\in \mathbb{U}{(X)}^-}b(X_{s-}),
 \]
We claim that
 \begin{claim}\label{claim:CTRW}For every $X\in \Pi([0,\,1];\Z)$ and uniformly over $u,\,v\in (0,\,1)$, $u\neq v$
 \[
 \nu\cdot{(\nu/\mu)}^{|\mathbb{U}{(X)}^+|}\le \frac{M(X+\mathbbm{1}_{[u,\,1]}-\mathbbm{1}_{[v,\,1]})}{M(X)}
 \le \mu\cdot{(\mu/\nu)}^{|\mathbb{U}{(X)}^+|}.
 \]
 \end{claim}
It follows by the same technique as in Corollary~\ref{prop:finalBoundCTRWrev} and Claim~\ref{claim:CTRW}
that
\[
d_{W,1}(P^{00}_\lambda,\bfP^{00}) \leq 9 E_{\bfP^{00}}\left[\left|\mu{(\mu/\nu)}^{|\mathbb{U}{(X)}^+|}-\lambda\right|\vee
\left|\nu{(\nu/\mu)}^{|\mathbb{U}{(X)}^+|}-\lambda\right|\right].
\]
We see that choosing $\lambda: = \sqrt{\mu \nu}$ entails
\[
\left|\mu{(\mu/\nu)}^{|\mathbb{U}{(X)}^+|}-\sqrt{\mu\nu}\right|\vee \left|\nu{(\nu/\mu)}^{|\mathbb{U}{(X)}^+|}-\sqrt{\mu\nu}\right|
\leq {\mu(\mu/\nu)}^{|\mathbb{U}{(X)}^+|} - \sqrt{\mu \nu},
\]
thus, all is left to do is to find a bound for
 \[
  \mu E_{\bfP^{00}}\left[{(\mu/\nu)}^{|U^+|}\right],
 \]
 where with a slight abuse of notation we have pushed forward the measure $E_{\bfP^{00}}$ via $\mathbb{U}$, thus calling $U:=\mathbb{U}(X)$.
 Hence it will suffice to bound the exponential moments of $|U^+|$. Introduce, for $t\in \R$, the Laplace transform $\phi(t):=\mathbb E_{\bfP^{00}}[\exp(t|U^+|)]$ under $\bfP^{00}$  as well as $\xi(t):=E_{P^{00}}\left[\exp(t|U^+|)\right]$. By change of measure
 \begin{align*}
  \phi(t)=\frac{E_{P^{00}}[\e^{t|U^+|+\log M(\mathbb{X}(U))}]}{E_{P^{00}}[\e^{\log M(\mathbb{X}(U))}]}.
 \end{align*}
 Since $|U^+|=|U^-|$ and~\eqref{eq:condRWvarying} holds, one can derive the bound
 \[
 {|U^+|}\log \nu -\kappa\le\log M(\mathbb{X}(U))\le {|U^+|}\log \mu-\kappa
 \]
 for every $U$, so that
 we get the following two-sided estimate:
 \[
  \frac{\xi(\log \nu+t)}{\xi(\log\mu)}\leq\phi(t)\le\frac{\xi(\log \mu+t)}{\xi(\log\nu)}.
 \]
 Under the law $P^{00}$, $|U^+|$ has the law described in Subsection~\ref{subsec:PoiDiag}, that is, $\mathrm{Poi}(1)\otimes\mathrm{Poi}(1)$ conditioned on the diagonal.
 A direct computation on the Laplace transform following from~\eqref{eq:rep_Q_Poisson} yields that
 \[
  \xi(t)=\frac{I_0(2 \e^{t/2})}{I_0(2)}.
 \]
 Therefore we can set $t:=\log (\mu/\nu)$ and obtain
\begin{align}\label{eq:final}
 \mu E_{\bfP^{00}}\left[{(\mu/\nu)}^{d(U,\,\mathbf{0})}\right] = \mu \phi(\log (\mu/\nu))
  \leq \mu\cdot \frac{I_0(2\mu/\sqrt{\nu})}{I_0(2\sqrt{\nu})}.
\end{align}
Finally~\eqref{eq:final}, together with~\eqref{eq:triangular} and the fact that we chose $\lambda = \sqrt{\mu \nu}$, gives the bound.
\end{proof}
\begin{proof}[Proof of Claim~\ref{claim:CTRW}]
 We prove that $\bfP^{00}$-almost surely
\begin{equation}\label{eq:loopcount} M(X) \leq \exp(- \kappa)\mu^{|\bbU^+(X)|}.
\end{equation}
  An identical arguments can then be used to show that
$M(X) \geq \nu^{U^+(X)}$; the claim then easily follows observing that $|\bbU^+(X+\mathbbm{1}_{[u,\,1]}-\mathbbm{1}_{[v,\,1]}))| = |\bbU^+{(X)}|+1$. Let us prove~\eqref{eq:loopcount} by induction on $|\bbU^+{(X)}|$.
The case $|\bbU^+ (X)|=0$ is obvious, since the only path verifying this condition which is also in the support of $\bfP^{00}$ is the zero path.
 Let $|\bbU^+ |=n+1$. Then either $X_t>0$ for some $t \in (0,1)$ or $X_t<0$ for some $t \in (0,1)$; we assume w.l.o.g.\ that the first condition is met. Define $ M:= \max_{t \in [0,1] } X_t$, $\tau_M := \inf \{ t: X_t=M \}$,  and $\theta_M$ as the first jump time of $X$ after $\tau_M$. Observe that, by construction,
 \begin{equation}\label{eq:loopcount2}
 \tau_M \in \bbU{(X)}^+,\,\theta_M \in \bbU{(X)}^-, \quad  \text{and }\, (X_{\tau_M-},X_{\tau_M},X_{\theta_M})=(M-1,M,M-1).
 \end{equation}
 Consider now the path
 $Z$ obtained by removing the jumps at $\tau_M,$ $\theta_M$, i.e.
 \[ Z = X-\mathbbm{1}_{[\tau_M,\,1]}+\mathbbm{1}_{[\theta_M,\,1]}.\]
 By construction $X_t$ and $Z_t$ coincide outside $[\tau_M,\theta_M)$ and $Z_t$ makes no jumps in $[\tau_M,\theta_M]$, whereas in the same interval $X$ goes first from $M-1$ to $M$ (at $\tau_M$) and then from $M$ to $M-1$ (at $\theta_M$), see~\eqref{eq:loopcount2}. Thus we have
 \[ M(X) = M(Z) a(M-1)b(M).\]
  Since $|\bbU{(Z)}^+| = n$, the conclusion follows using the inductive hypothesis and~\eqref{eq:condRWvarying}.
\end{proof}


\subsection{An approximation scheme for the simple random walk bridge}\label{subsec:schemesRW}

In this Subsection we will be interested in schemes for approximating the continuous-time random walk bridge with rates $j_+=j_-:=1$. Its law $P^{00}$ has been defined in Subsection~\ref{subsec:CTRW}.

Let $N\in \bbN$ be fixed. Consider a sufficiently large probability space $(\Omega,\mathcal{F},Q)$ on which we can define independent random variables $\xi_1,\,\ldots,\,\xi_N$, $\tau_1,\,\ldots \,\tau_N$ such that for all $j=1,\,\ldots,\,n$
\[
Q(\xi_j =1)=Q(\xi_j =-1)=1/N,\quad Q(\xi_j=0)=1-2/N
\]
and $\tau_{j}$ is uniformly distributed on $I_j:= ((j-1)/N,\,j/N]$. We define the process $Y$ with values in $\bbD([0,1];\bbZ)$ via
\[ Y_t:=\sum_{j=1}^N \xi^N_j \IND_{[\tau_j,\,1]}(t),\quad t\in[0,\,1], \]
and call $P^0_N$ its law. Let $P^{00}_N$ be the distribution of its bridge:
\[ P^{00}_N(\cdot)=P^{0}_N(  \cdot |Y_1=0).  \]
The bridge measure $P^{00}_N$ is clearly supported in $\Pi([0,1];\bbZ)$ which is in bijection with $\mathscr{U}$. As in the previous sections, we shall still use the notation $P^{00}_N$ for the pushforward of $P^{00}_N$ through $\bbU$.
In this subsection we shall prove the following.
\begin{theorem}\label{thm:scheme}
For all $N\in \bbN$ we have
\[ d_{W,1}(P^{00}_N ,P^{00}) \leq \frac{1}{N}\cdot \frac{9 \left(9 N^3-54 N^2+64 N-16\right)}{{(N-2)}^3}.
\]
\end{theorem}

The theorem will be proved at the end of the Section. As usual, we make no distinction between $P^{00}_N$ and its push forward through $\bbU$.

The first step towards the proof of the result is exhibiting a dynamics for which $P^{00}_N$ is invariant. Therefore for every $U\in \mathscr{U}$ define
\[
\mathscr{B}(U):= \Big \{(r,s) \in \mathscr{A}:\,\lceil r N\rceil \neq \lceil s N\rceil\text{ and }(I_{\lceil r N\rceil } \cup I_{\lceil s N\rceil}) \cap (U^+ \cup U^-) =\emptyset\Big \}
\]
where recall that $\mathscr A:={(0,\,1)}^2\setminus\Delta$.
 Consider the operator defined for any bounded measurable function $f:\mathscr{U} \rightarrow \bbR$ as
\begin{equation}
\mathscr{L}^N f(U) \ldef{ \left( 1-\frac{2}{N}\right)}^{-2} \int_{\mathscr{B}(U)} (f( \Psi_{r,s} U ) - f(U))   \De r \De s+ \sum_{(r,s) \in U^+ \times U^-} (f( \Psi_{r,s}U )-f(U)).\label{eq:formLN}
\end{equation}
We will show below that such an operator admits $P^{00}_N$ as invariant distribution. To do so, first we want to calculate $\De P^{00}_N/\De P^{00}$. In fact, the knowledge of the Radon--Nikodym derivative can be used to derive an integration by parts formula for $P^{00}_N$ by bootstrapping that of $P^{00}$ in the spirit of~\eqref{eq:genLeibniz} and subsequent discussion.

\begin{lemma}
Let
\begin{equation}\label{eq:scheme_support} \mathcal{S}= \{ U: |(U^+ \cup U^-) \cap I_{j}| \leq 1 \,\text{ for all } \, j=1,\ldots,N   \}.
\end{equation}
We have
\begin{equation*} \frac{\De P^{00}_N}{\De P^{00}}(U) = \frac{1}{Z} \IND_{\{U \in \mathcal{S}\}}{\left(1-\frac{2}{N}\right)}^{N-2|U^+|}.
\end{equation*}
\end{lemma}

\begin{proof}
Recall that $P^0$ is the law of the continuous time random walk started at $0$, without conditioning at the terminal point. We prove that
\begin{equation}\label{eq:scheme_dns}
 \frac{\De P^0_N}{\De P^0}(U) = \e^2{\left(1-\frac{2}{N}\right)}^{N-2|U^+|}\IND_{\{U \in \mathcal{S}\}}=:M(U),
\end{equation}

The conclusion then follows from the fact that the conditional density $\frac{\De P^{00}}{\De P^{00}_N}$ is equal to  $\frac{\De P}{\De P_N}$ up to a multiplicative constant, and that $P^{00}_N(|U|^+=|U^-|)=1$. It follows from the construction of $Y$ that $P^{00}_N(\mathcal{S})=1$. Moreover, we observe that a basis for the restriction to $\mathcal{S}$ of the canonical sigma algebra is given by events of the form
\begin{equation}\label{eq:event_form} A=\bigcap_{\substack{j=1, \ldots ,N \\ i=1,\ldots, L}} \{ |U^+ \cap (a^+_{ij}, b^+_{ij}]| = k^+_{ij} \} \cap \{ |U^- \cap (a^-_{ij},b^-_{ij}]| =k^-_{ij}  \}
\end{equation}
where, for all $1 \leq j \leq N$, we have that $\sum_{i=1}^L (k^+_{ij}+k^{-}_{ij}) \leq 1$, that the intervals $\{(a^+_{ij}, b^+_{ij}]: \ i=1,\ldots,\, L\}$ form a disjoint partition of $I_{j}$ and so do the intervals $\{(a^-_{ij},\, b^-_{ij}]: \ i=1,\ldots ,\,L\}$.
Thus all what we have to show is that for an event $A$ as in~\eqref{eq:event_form} we have
\[ P^{0}_N(A)=E_{P^{0}}\left[M \IND_A\right] . \]
Since
\[M\equiv\e^2 {\left(1-\frac{2}{N}\right)}^{N-\sum_{i,j} (k^+_{ij}+k^{-}_{ij})}\]
on $A$, we can equivalently show that
\[ P^0_N(A)=\e^2 {\left(1-\frac{2}{N}\right)}^{N-\sum_{i,j} (k^+_{ij}+k^{-}_{ij})}P^0(A). \]
To check this define
\[J^+ := \left \{ j: \sum_{i=1}^L k^+_{ij}=1 \right \}, \quad J^- := \left \{ j: \sum_{i=1}^L k^-_{ij}=1 \right \}  \]
and for all $j \in J^+$ (resp. $J^-$) define $i_j$ as the only index such that $k^{+}_{i_j j}=1$ (resp. $k^{-}_{i_j j}=1$). Then, since $U^+$ and $U^-$ under $P^0$ are distributed as a Poisson point process with mean measure the Lebesgue measure,
\[ P^0(A) = \e^{-2}\prod_{j \in J^+} (b^+_{i_j j} - a^+_{i_j j} )\prod_{j \in J^-} (b^-_{i_j j} - a^-_{i_j j} ) \]
On the other hand, using the explicit construction of $Y$,
\begin{align*}
P^0_N(A) &= \prod_{j\in J^+} Q\left(\xi_j=1,\tau_j \in [a^+_{i_j j},b^+_{i_j j} ) \right)\, \prod_{j\in J^-} Q\left(\xi_j= -1,\tau_j \in [a^-_{i_j j},b^-_{i_j j} ) \right) \times \\
&\qquad\qquad\times \prod_{j\notin J^- \cup J^+} Q(\xi_j = 0) \\
&= {\left (1-\frac{2}{N} \right)}^{N- \sum_{i,j} (k^+_{ij} +k^-_{ij}) }\prod_{j \in J^+} (b^+_{i_j j} - a^+_{i_j j} )\prod_{j \in J^-} (b^-_{i_j j} - a^-_{i_j j})\\
& = \e^2 {\left(1-\frac{2}{N} \right)}^{N- \sum_{i,j} (k^+_{ij} +k^-_{ij}) } P^0(A),
\end{align*}
where we used the fact that $|J^+\cup J^-| =\sum_{i,j} (k^+_{i,j} +k^-_{ij}) $.
The Lemma is now proven.
\end{proof}

\begin{prop}
$P^{00}_N$ is invariant for $\mathscr{L}^N$, i.e. for all $f:\mathscr U\to \R$ bounded and measurable
\begin{equation}\label{eq:scheme_dynamics} E_{P^{00}_N}[\mathscr{L}^N f]=0.
\end{equation}
\end{prop}
\begin{proof}
Let $M$ be the density given in~\eqref{eq:scheme_dns} and $\mathcal{S}$ as in~\eqref{eq:scheme_support}. Using the fact that
\[
M(U)=0 \Rightarrow M(\Psi_{r,s}U) = 0, \quad (r,s)-\text{almost everywhere},
\]
we can reason exactly as in Proposition~\ref{prop:SteinRevCTRW}, to obtain that for all $f$ bounded and measurable
\[ E_{P^{00}_N} \left[ \int_{\mathscr{A}} (f( \Psi_{r,s} U ) - f(U)) \frac{M(\Psi_{r,s} U)}{M(U)}   \De r\De s+ \sum_{(r,s) \in U^+ \times U^-} (f( \Psi_{r,s}U )-f(U)) \right]=0.
\]
Assume that $U \in \mathcal{S}$ and $(r,s)\notin \mathscr{B}(U)$; then either one among $I_{\lceil N r \rceil } \cap U^+,\, I_{\lceil N r \rceil } \cap U^-,\,I_{\lceil N s \rceil } \cap U^+,\, I_{\lceil N s\rceil } \cap U^-$ is non-empty or $\lceil N r \rceil = \lceil N s \rceil$.
Assume that $I_{\lceil r N\rceil } \cap U^+ \neq \emptyset$, the other cases being completely analogous. Then ${(\Psi_{r,s}U)}^+ \cup{(\Psi_{r,s}U)}^-$ has at least two points in $I_{\lceil r N\rceil }$ and thus is not in $\mathcal{S}$. Therefore
\[
 (r,s) \notin \mathscr{B}(U) \Rightarrow \frac{M(\Psi_{r,s} U)}{M(U)} =0.
\]
In the same way, one can show that $(r,s) \in \mathscr{B}(U) \Rightarrow \Psi_{r,s}U \in \mathcal{S}$. Moreover, since ${(\Psi_{r,s}U)}^+$ has exactly one element more than $U^+$,
\[ \frac{M(\Psi_{r,s} U)}{M(U)} = {\left(1-\frac{2}{N}\right)}^{-2}. \]
Summing up,
\[ \frac{M(\Psi_{r,s} U)}{M(U)} = {\left(1-\frac{2}{N}\right)}^{-2} \IND_{\{(r,s) \in \mathscr{B}(U)\}} \quad P^{00}_N \text{-a.s.} \]
from which the conclusion follows.
\end{proof}

Having identified in $\mathscr{L}^N$ an operator which has $P^{00}_N$ as invariant distribution, we are ready to prove Theorem~\ref{thm:scheme}.

\begin{proof}[Proof of Theorem~\ref{thm:scheme}]
Arguing as in Corollary~\ref{prop:finalBoundCTRWrev}, we are left to evaluate
\[ d_{W,1}( P^{00}_N ,P^{00} ) \leq 9 \sup_{g \in \text{Lip}_1(\mathscr{U})} E_{P^{00}_N}\left[| \mathscr{L}g-\mathscr{L}^N g| \right]. \]
Using the explicit form of $\mathscr{L}$ in \eqref{eq:genRWonZ}, $\mathscr{L}^N$ in \eqref{eq:formLN} and the fact that $g$ is $1$-Lipschitz, we readily obtain the bound ($\lambda$ here denotes the Lebesgue measure on $[0,\,1]^2$)
\begin{eqnarray} |\mathscr{L}g-\mathscr{L}^N g| &\leq& \left({\left(1-\frac{2}{N}\right)}^{-2}-1  \right) \lambda(\mathscr{B}(U))   + \lambda(\mathscr{A} \setminus \mathscr{B}(U))\nonumber\\
&=&\left({\left(1-\frac{2}{N}\right)}^{-2}-1  \right) + \left(2-{\left(1-\frac{2}{N}\right)}^{-2}  \right) \lambda (\mathscr{A} \setminus \mathscr{B}(U)).\label{eq:Lf-LNf}
\end{eqnarray}
Define for $1\leq i,\,j \leq N$, the square $S_{ij} :=   I_{i}\times I_{j}$ and for any $v \in U^+ \cup U^-$ the index $k^v$ as that of the interval $I_{k^v}$ containing $v$. Note that $k^v$ is $P^{00}_N$-almost surely a bijection. As a consequence the family ${\{ k^v \}}_{v \in U^+ \cup U^-}$ is made of $|U^+|+|U^-|=2|U^+|$ elements.
Also observe that, by definition of $\mathscr{B}(U)$, we have that
\begin{equation}\label{eq:occ_cond} S_{ij} \not\subset  \mathscr{B}(U) \Rightarrow i=j\ \text{or one between}\  i,\,j  \text{ equals $k^v$ for some
} v \in U^+ \cup U^-.
\end{equation}
Since there are less than $4N|U^+|+N$ pairs $(i,j)$ verifying~\eqref{eq:occ_cond}, then $\mathscr{A}\setminus \mathscr{B}(U)$ is contained in the union of at most $4N|U^+|+N$ squares, each having area $N^{-2}$. Thus we obtain the bound
\begin{equation}\label{eq:lAB}\lambda (\mathscr{A} \setminus \mathscr{B}(U)) \leq \frac{4}{N}|U^+| + \frac{1}{N}.
\end{equation}
All what is left to do is to estimate $E_{P^{00}_N}[|U^+|]$. Plugging $f(U):=|U^+|$ into~\eqref{eq:scheme_dynamics}, we obtain
\[  E_{P^{00}_N}[|U^+|^2] =  {\left(1-\frac{2}{N}\right)}^{-2} E_{P^{00}_N}[ \lambda(\mathscr{B} ) ] \leq {\left(1-\frac{2}{N}\right)}^{-2}, \]
from which we deduce, after an application of Jensen's inequality that
\begin{equation} E_{P^{00}_N}[|U^+|]\leq  {\left(1-\frac{2}{N}\right)}^{-1}.\label{eq:mean_salti} \end{equation}
The conclusion then follows taking the expectation under $P_N^{00}$ in~\eqref{eq:Lf-LNf} and using \eqref{eq:lAB}-\eqref{eq:mean_salti}.
\end{proof}


\appendix
\section{Proof of Lemma~\ref{lem:clockswitch}}\label{app:clockswitch}
We recall here the statement of the Lemma, for the reader's convenience.
\begin{lemma*}
Let $T$ be a $\cF_t$-stopping time and $\sigma:\Omega\rightarrow \mathscr{A}^{\mathscr{A}}$ a random bijection  which is $\cF_T$-measurable.
Define the family ${\{\rho^A \}}_{A \in \mathscr{A}}$ by (recall~\eqref{eq:xi_with_indices})
\[ \rho^A_t:= \xi^A_t \mathbbm{1}_{\{t< T\}} + ( \xi^A_T + \xi^{\sigma(A),T}_{t-T} ) \mathbbm{1}_{\{t \geq T\}}. \]
Then ${\{ \rho^A \}}_{A \in \mathscr{A}}$ is distributed as $\Xi$.
\end{lemma*}
\begin{proof}Observe that $\rho$ coincides with $\xi$ up to time $t<T$, and for $t\ge T$
\[
\rho^A_t=\xi^A_T+\xi^{\sigma(A)}_t-\xi^{\sigma(A)}_T.
\]
Since the family ${\{\xi^A \}}_{A\in \mathscr{A}}$ is obtained by choosing $\sigma \equiv \mathbf{id}_{\mathscr{A}}$ it is sufficient to show that for any family $\mathscr{A}':=\{ A_1,..,A_n \} \subseteq \mathscr{A}$ of pairwise disjoint subsets,
any $0\le t_0 < t_1 < \cdots <t_m$ and any bounded $\cF$-measurable functions ${(f_{kj})}_{1\leq j \leq n,1\leq k \leq m}$, the quantity
\[
  \mathbb E \left[ \prod_{\shortstack{$\scriptstyle1\leq k\leq m $ \\ $ \scriptstyle  1 \leq j \leq n$}}  f_{kj}\left(\rho^{A_j}_{t_k}-\rho^{A_j}_{t_{k-1}}\right) \right]
 \]
is independent of the the particular random bijection $\sigma:\Omega\to \mathscr{A}^\mathscr{A}$. We can w.l.o.g.\ restrict to the case when $\mathbb P(T \in \Theta) =1$ and $\bbP(\sigma \in \mathfrak{S} )=1$  for some finite sets $\Theta\subset[0,\,+\infty)$ and $\mathfrak{S} $ a subset of bijections of $\scrA$. The general case follows with a standard approximation argument. We have

\begin{align}\label{eq:clock1}
 \mathbb E & \left[ \prod_{\shortstack{$ \scriptstyle 1 \leq k \leq m$                                                      \\ $ \scriptstyle  1 \leq j \leq n$}} f_{kj}\left(\rho^{A_j}_{t_k}-\rho^{A_j}_{t_{k-1}}\right) \right] = \sum_{\theta \in \Theta,\, \pi\in \mathfrak{S}} \mathbb E\left[ \prod_{\shortstack{$ \scriptstyle 1 \leq k \leq m$ \\ $ \scriptstyle  1 \leq j \leq n$}}  f_{kj}\left(\rho^{A_j}_{t_k}-\rho^{A_j}_{t_{k-1}}\right) \IND_{\{T=\theta, \,\sigma = \pi \}} \right].
 \end{align}
Fix now $\theta$ and $\pi$ and call $k_{\theta}$ the smallest index $k$ such that $t_k \geq \theta$.
Since $\sigma$ is a bijection of $\mathscr{A}$ there exist $B_1,..,B_n \in \mathscr{A}$ all different such that $\pi(A_j) = B_j$ for all $1 \leq j \leq n$.
Using the definition of $\rho$ one deduces
\begin{align}
   & \nonumber  \mathbb E \left[\prod_{\shortstack{$ \scriptstyle 1 \leq k \leq m$                                                                      \\ $ \scriptstyle  1 \leq j \leq n$}} f_{kj}\left(\rho^{A_j}_{t_k}-\rho^{A_j}_{t_{k-1}}\right)\IND_{\{T=\theta, \sigma = \pi \}}
  \right]= \mathbb E \left[ \prod_{\shortstack{$ \scriptstyle 1 \leq k \leq  k_\theta-1$                                                                       \\ $ \scriptstyle 1 \leq j \leq n$}}f_{kj}\left(\xi^{A_j}_{t_k}-\xi^{A_j}_{t_{k-1}}\right)\times\right.\nonumber \\
   & \left.\times\prod_{1\leq j \leq n}f_{kj}\left(\rho^{A_j}_{t_{k_\theta}}-\rho^{A_j}_{t_{k_{\theta}-1}}\right) \prod_{\shortstack{$ \scriptstyle k_\theta+1 \leq k \leq m$                                                         \\ $ \scriptstyle  1 \leq j \leq n$}}f_{kj}\left(\xi^{B_j}_{t_k}-\xi^{B_j}_{t_{k-1}}\right) \IND_{\{T=\theta, \sigma = \pi \}}\right].
  \label{eq:clock2}
\end{align}
Since $\sigma$ is $\cF_T$-measurable and $\theta \leq t_{k_\theta}$ the event $ \{T=\theta, \sigma = \pi \}$ belongs to ${\cF}_{t_{k_{\theta}}}$. The random variable $\xi^{B_j}_{t_k}-\xi^{B_j}_{t_{k-1}}$ is a Poisson of parameter $ t_k-t_{k-1}$ independent from $\cF_{t_{k-1}}$ for all $k\ge k_\theta+1$, and hence from $\cF_{t_{k_\theta}}$.
We can therefore use the independence of the increments of the Poisson processes and their stationarity to conclude that the right-hand side of \eqref{eq:clock2} equals
\begin{align}
  &\prod_{\shortstack{$ \scriptstyle k_\theta+1 \leq k \leq m$                                                                                                   \\ $ \scriptstyle  1 \leq j \leq n$}}\mu_{ t_k-t_{k-1}}(f_{kj}) \times\nonumber\\
  &\mathbb E\left[  \prod_{\shortstack{$ \scriptstyle 1 \leq k \leq  k_\theta-1$                                                                       \\ $ \scriptstyle 1 \leq j \leq n$}}f_{kj}\left(\xi^{A_j}_{t_k}-\xi^{A_j}_{t_{k-1}}\right)\prod_{1\leq j \leq n}f_{kj}\left(\rho^{A_j}_{t_{k_\theta}}-\rho^{A_j}_{t_{k_\theta-1}}\right)
  \IND_{\{T=\theta, \sigma = \pi \}}\right].\label{eq:clock3}
 \end{align}
Now for all $1\le j\le n$ one sees that
\[ \Big(\rho^{A_j}_{t_{k_\theta}}-\rho^{A_j}_{t_{k_\theta-1}}\Big) \IND_{\{T= \theta,\sigma= \pi \}} = \Big( \xi^{B_j}_{t_{k_{\theta}}}-\xi^{B_j}_{\theta} + \xi^{A_j}_{\theta}-\xi^{A_j}_{t_{k_{\theta}-1}}  \Big)
  \IND_{\{ T=\theta,\sigma = \pi \}}.
\]
The random variable $ \xi^{B_j}_{t_{k_{\theta}}}-\xi^{B_j}_{\theta}$ is a Poisson of parameter $ t_{k_{\theta} }-\theta$ independent from ${\cF_{\theta}}$, whereas both $\xi^{A_j}_{\theta}$ and $\xi^{A_j}_{t_{k_{\theta}-1}} $ are $ {\cF_{\theta}}$-measurable. This observation gives, using again independence, that
\begin{align}\label{eq:clock5}
&\mathbb E\left[  \prod_{\shortstack{$ \scriptstyle 1 \leq k \leq  k_\theta-1$                                                                       \\ $ \scriptstyle 1 \leq j \leq n$}}f_{kj}\left(\xi^{A_j}_{t_k}-\xi^{A_j}_{t_{k-1}}\right)\prod_{1\leq j \leq n}f_{kj}\left(\rho^{A_j}_{t_{k_\theta}}-\rho^{A_j}_{t_{k_\theta-1}}\right)
  \IND_{\{T=\theta, \sigma = \pi \}}\right]\nonumber\\
            &   = \mathbb E \Bigg[\prod_{\shortstack{$ \scriptstyle 1 \leq k \leq  k_\theta-1$                         \\ $ \scriptstyle  1 \leq j \leq n$}}f_{kj}\left(\xi^{A_j}_{t_k}-\xi^{A_j}_{t_{k-1}}\right)\prod_{1 \leq j \leq n}\mu_{ t_{k_{\theta}} - \theta }(f^{*}_{k_\theta j })    \IND_{\{T=\theta, \sigma = \pi \}} \Bigg].
\end{align}
Here for all $j=1,\,\ldots,\,n$ we have defined $f^*_{k_{\theta},j}$ as the function
\begin{equation}\label{eq:clock6}  z\mapsto f_{k_{\theta},j}\Big( \xi^{A_j}_{\theta}-\xi^{A_j}_{t_{k_{\theta}-1}}+z\Big).
\end{equation}
Putting together~\eqref{eq:clock1}-\eqref{eq:clock6} we arrive to
\begin{eqnarray*}
  &&\mathbb E  \left[ \prod_{\shortstack{$ \scriptstyle 1 \leq k \leq m$                                                      \\ $ \scriptstyle  1 \leq j \leq n$}} f_{kj}\left(\rho^{A_j}_{t_k}-\rho^{A_j}_{t_{k-1}}\right) \right]
  = \sum_{\theta \in \Theta} \;\prod_{\shortstack{$ \scriptstyle k_\theta+1 \leq k \leq m$                                                                                                   \\ $ \scriptstyle  1 \leq j \leq n$}}\mu_{t_k-t_{k-1}}(f_{kj}) \, \times \\
  &&\times
  \mathbb E \Bigg[\prod_{\shortstack{$ \scriptstyle 1 \leq k \leq  k_\theta-1$                         \\ $ \scriptstyle  1 \leq j \leq n$}}f_{kj}\left(\xi^{A_j}_{t_k}-\xi^{A_j}_{t_{k-1}}\right)\prod_{1 \leq j \leq n}\mu_{ t_{k_{\theta}} - \theta }(f^{*}_{k_\theta j })    \IND_{\{T=\theta\}} \Bigg].
 \end{eqnarray*}
Since this last expression is independent from $\sigma$ the conclusion follows.
\end{proof}

\bibliographystyle{abbrvnat}
\bibliography{Ref,BibliographyStein}
\end{document}